\documentclass[final,1p,times,numbers,sort]{elsarticle}

\usepackage{yhmath} 

\usepackage{amsfonts}
\usepackage{amsxtra}

\usepackage{dsfont}
\usepackage{mathrsfs}
\usepackage{amsthm}
\usepackage[T1]{fontenc}
\usepackage[utf8]{inputenc}
\usepackage[french,english]{babel}
\usepackage{amsmath}
\usepackage{bm}
\usepackage[colorlinks=true,linkcolor=blue]{hyperref}
\usepackage{color}
\usepackage{relsize}
\usepackage{stmaryrd}
\usepackage{listings}
\usepackage[shortlabels]{enumitem}
\usepackage{url}

\numberwithin{equation}{section}

\makeatletter
\def\ps@pprintTitle{%
  \let\@oddhead\@empty
  \let\@evenhead\@empty
  \def\@oddfoot{}
  \let\@evenfoot\@oddfoot}
\makeatother

\newtheorem{Thm}{Theorem}[section]
\newtheorem{Cor}[Thm]{Corollary}
\newtheorem{Prop}[Thm]{Proposition}
\newtheorem{Lem}[Thm]{Lemma}
\theoremstyle{definition}
\newtheorem{Rem}{Remark}      
\newtheorem{Defn}[Thm]{Definition}
\newtheorem{Ex}[Thm]{Example}
\newtheorem{Fact}[Thm]{Fact}
\newtheorem*{Prob*}{Problem}

\newtheorem*{dsq}{Stein's question}

\newcommand{\Z}{\mathbb{Z}}
\newcommand{\D}{\mathbb{D}}
\newcommand{\kr}{\mathds{kr}}
\newcommand{\N}{\mathbb{N}}
\newcommand{\T}{\mathbb{T}}

\newcommand{\R}{\mathbb{R}}
\newcommand{\mC}{\mathcal{C}}
\newcommand{\E}{\mathop{\mathbb{E}}}
\DeclareMathOperator{\Id}{Id}

\DeclareMathOperator{\Sym}{Sym}
\DeclareMathOperator{\supp}{supp}
\DeclareMathOperator{\Ima}{Im}
\DeclareMathOperator{\Rea}{Re}

\newcommand{\se}[1]{\underset{_{\substack{#1}}}{\mathlarger{\mathlarger{\E}}}}

\begin{document}

\begin{frontmatter}

\title{Dimension-free estimates for discrete maximal functions and lattice points in high-dimensional spheres and balls with small radii\tnoteref{t1}}

\tnotetext[t1]{Both authors were supported by the National Science Centre, Poland, grant Sonata Bis 2022/46/E/ST1/00036. Jakub Niksiński was also supported by the NSF CAREER grant DMS-2236493.}

\author[inst1]{Jakub Niksiński}
\address[inst1]{Institute of Mathematics, University of Wrocław, Plac Grunwaldzki 2, 50-384 Wrocław, Poland}

\author[inst2,inst1]{Błażej Wróbel}
\address[inst2]{Institute of Mathematics of the Polish Academy of Sciences, Śniadeckich 8, 00-656 Warsaw, Poland\\
Email addresses: trolek1130@gmail.com (JN), blazej.wrobel@math.uni.wroc.pl (BW)}



\selectlanguage{english}

\begin{abstract}
We prove that the discrete Hardy-Littlewood maximal function associated with Euclidean spheres with small radii has dimension-free estimates on $\ell^p(\mathbb{Z}^d)$ for $p\in[2,\infty).$ This implies an analogous result for the Euclidean balls, thus making progress on a question of E.M. Stein from the mid 1990s. Our work provides the first dimension-free estimates for full discrete maximal functions related to spheres and balls without relying on comparisons with their continuous counterparts. An important part of our argument is a uniform (dimension-free) count of lattice points in high-dimensional spheres and balls with small radii. We also established a dimension-free estimate for a multi-parameter maximal function of a combinatorial nature, which is a new phenomenon and may be useful for studying similar problems in the future.
\vspace{1.0em}

\end{abstract}

\end{frontmatter}



\section{Introduction}

\subsection{Statement of the results}

Let $G$ be a subset of $\R^d$. For every $t\ge 0$ we let $tG=\{tx\colon x\in\R^d\}$ be the dilation of $G.$ Let $\mathbb{J}\subset [0,\infty)$ be the set of those $t\ge 0$ such that $tG\cap \Z^d$ is non-empty. For $t\in \mathbb{J}$ we consider the discrete Hardy-Littlewood averaging operator
\begin{align*}
	\mathcal M_t^Gf(x):=\frac{1}{|tG\cap \Z^d|}\sum_{y\in tG\cap\Z^d}f(x-y),
	 \qquad f\in \ell^1(\Z^d),\quad 
	x\in\Z^d.
\end{align*}
The symbol $|tG\cap \Z^d|$ above stands for the number of elements of the set $tG\cap \Z^d.$ For $p\in (1,\infty]$ and $\mathbb{I}\subseteq \mathbb{J}$ we also let
\begin{equation}
\label{eq:mCp}
{\mC}(p,\mathbb I, G)=\sup_{\|f\|_{\ell^p(\Z^d)}\le1}\big\|\sup_{t\in\mathbb I}|\mathcal M_t^Gf|\big\|_{\ell^p(\Z^d)}.
\end{equation}
Note that trivially, ${\mC}(\infty,\mathbb{J}, G)\le 1.$

When  $G=B=B^2(d)=\{x\in\R^d\colon \sum_{j=1}^d x_j^2 \le 1\}$ is the Euclidean ball of radius one then it is well known  that in all dimensions $d$ we have ${\mC}(p,\R_+, B^2(d))<\infty$ for $1<p\le\infty,$ where $\R_+=(0,\infty).$  However, classical arguments rely on interpolating  the $\ell^{\infty}(\Z^d)$ bound with a weak-type $(1,1)$ estimate, leading to bounds for ${\mC}(p,\R_+, B^2(d))$ that grow exponentially in $d.$  In the mid 1990's E. M. Stein asked  about independent of the dimension estimates for ${\mC}(2,\R_+, B^2(d)).$ Below and throughout the paper we abbreviate
\[
\mathcal M_t:=\mathcal M_t^{B^2(d)}.
\]
\begin{dsq}[\cite{SteinPC}]
	\label{con:dsq}
Does there exist a universal constant $C>0$ such that
	\begin{equation}
		\label{eq:36'}
            \tag{SQ}
    \Big\| \sup_{t \ge 0} |\mathcal{M}_tf|\Big\|_{\ell^2(\Z^d)} \leq C \big\| f \big\|_{\ell^2(\Z^d)}
	\end{equation}
holds uniformly, independent of the dimension $d\in \N$?
\end{dsq}

Our paper makes progress on this question in the small scales regime $t<d^{1/2-\varepsilon}.$ 
\begin{Thm} \label{thm:ball}
For any $\varepsilon>0$ there exists $C_\varepsilon>0$  depending only on $\varepsilon$ and such that 
    \[
    \Big\| \sup_{0\le t \leq d^{\frac{1}{2}- \varepsilon}} |\mathcal{M}_tf| \Big\|_{\ell^p(\Z^d)} \leq C_{\varepsilon} \big\| f \big\|_{\ell^p(\Z^d)},\qquad f\in \ell^p(\Z^d),
    \]
   holds uniformly in the dimension $d \in \N$ and $p\in[2,\infty]$. 
\end{Thm}
  In fact, our method yields a stronger result when balls are replaced by spheres. Namely, let $S^{d-1}$ denote the $(d-1)$-dimensional unit sphere in $\R^d,$ i.e. $S^{d-1}=\{x\in\R^d\colon \sum_{j=1}^d x_j^2 = 1\}.$ For $d\ge 4$ and $t\in \sqrt{\N_0}$ we abbreviate $\mathcal M_t^{S^{d-1}}$ as
 \[ 
 \mathcal S_t :=\mathcal M_t^{S^{d-1}},
 \]
 where $\sqrt{\N_0}$ denotes the set of square roots of non-negative integers $\N_0$. Note that for $d\ge4$ the set $tS^{d-1}\cap\Z^d$ is non-empty precisely for $t\in \sqrt{\N_0}.$  We establish the following result.
\begin{Thm} \label{thm:sphere}
    For any $\varepsilon>0$ there exists $C_\varepsilon>0$ depending only on $\varepsilon$ and such that 
     \[
    \Big\| \sup_{n\in \N_0\colon n\leq d^{1- \varepsilon}} |\mathcal{S}_{\sqrt{n}}f| \Big\|_{\ell^p(\Z^d)} \leq C_{\varepsilon} \big\| f \big\|_{\ell^p(\Z^d)}, \qquad f\in \ell^p(\Z^d),
    \]
    holds uniformly in dimensions $d\ge 4$ and $p\in[2,\infty].$
\end{Thm}

 It is easy to see that  \[\sup_{0\le t\leq d^{\frac{1}{2}- \varepsilon}} |\mathcal{M}_tf(x)|= \sup_{n\in \N_0\colon n\leq d^{1- 2\varepsilon}} |\mathcal{M}_{\sqrt{n}}f(x)|\le \sup_{n\in \N_0\colon n\leq d^{1- 2\varepsilon}} |\mathcal{S}_{\sqrt{n}}f(x)|,\]
 where the inequality above follows from the fact that averages over discrete balls are themselves averages of avereges over discrete spheres. This observation shows that Theorem \ref{thm:sphere} implies Theorem \ref{thm:ball}. However, in the small scales regime $t=O(d^{1/2-\varepsilon})$ these two results are in fact closely related. This will become evident at a later stage in the paper.

As a key tool in proving Theorem \ref{thm:sphere}, we establish two types of results that we believe are interesting in their own right; therefore, we explicitly highlight them here.

The first type of result concerns counting lattice points in high-dimensional spheres and balls with small radii. A representative example of this kind is given below. In what follows we write $\N$ to denote the set of positive integers and for $n,d\in \N$ we let
\[
\alpha=\frac{n}{d}.
\]
For $K\in \N$ the symbol $O_K$ stands for a big $O$ notation in which the implicit constant depends only on $K.$ 
We will write $C,C',C_1,C_2,...,$ to denote positive universal constants and we abbreviate $B=B^2(d),$ $S=S^{d-1}.$

\begin{Thm}[Part of Theorem \ref{thm:3.4}]
\label{thm:lat:extract} There exist real numbers $b_k,$ $k=1,2,\ldots,$ and
a universal constant $c\in (0,1)$ such that for all $K\in \N$ and all $1\le n\le cd$ we have
\begin{equation}
\label{eq: compagen BS}
\begin{split}
|\sqrt{n}S\cap\Z^d|&\le |\sqrt{n}B\cap\Z^d| \\
&\le C_1 2^n e^n \alpha^{-n}  \frac{1}{\sqrt{n}
}\exp \Big( \sum_{k=1}^{K} b_k n\alpha^k + O_K(n \alpha^{K+1}) \Big) \le C_2|\sqrt{n}S\cap\Z^d|.
\end{split}
\end{equation}
The numbers $b_k$ are coefficients of a power series with
$b_1=-\frac{1}{2},$ $ b_2=-\frac{1}{6},$ $ b_3=\frac{1}{24}$.
\end{Thm}
\noindent The above formula may be thought of as a quantitative uniform (dimension-free) asypmtotic formula in the classical Waring's problem for the squares within the regime $1 \leq n \leq c d.$ Such a uniform formula was proved before in the range $n>Cd^3$ in  \cite[Theorem 3.1]{MiSzWr}. Moreover, \cite[Theorem 1]{Rou} implies that under mild arithmetic assumptions on $n,d$ we have a dimension-free asymptotic formula in Waring's problem in the broader range $n>Cd^{2}$. The estimate from \cite[Theorem 3.1]{MiSzWr} combined with \cite[Lemma 5.3.]{BMSW4} shows that 
\[C|\sqrt{n}B\cap\Z^d|\le (n/d)|\sqrt{n}S\cap\Z^d|\le C'|\sqrt{n}B\cap\Z^d|,\qquad \textrm{for }n>Cd^3.\]
The above estimate contrasts with the behavior for $n\le cd$ in which case according to \eqref{eq: compagen BS} the quantities $|\sqrt{n}B\cap\Z^d|$ and $|\sqrt{n}S\cap\Z^d|$ are comparable up to universal constants.

Rewriting Theorem \ref{thm:lat:extract} in terms of binomial coefficients and taking $K=3$ leads to the following corollary.
\begin{Cor} \label{cor:1.4}
There exists a universal constant $c \in (0,1)$ such that for  $n,d \in \N$ satisfying  $1 \leq n \leq c d$ we have 
\begin{equation*}
|\sqrt{n}S\cap\Z^d|\le |\sqrt{n}B\cap\Z^d| \le C_3  2^n \binom{d}{n} \exp \Big( \frac{n \alpha^3}{8} + O(n \alpha^{4}) \Big)\le C_4 |\sqrt{n}S\cap\Z^d|.
\end{equation*}
In particular, if $n=O(d^{3/4}),$ then
\begin{equation}
\label{eq:snBZ}
 |\sqrt{n}S\cap \{-1,0,1\}^d|\le |\sqrt{n}S\cap\Z^d| \le |\sqrt{n}B\cap\Z^d| \le  C_5 2^n \binom{d}{n}=C_5|\sqrt{n}S\cap \{-1,0,1\}^d|.
\end{equation}
\end{Cor}
\noindent Corollary \ref{cor:1.4} provides an explicit and uniform comparison between the number of lattice points in  $\sqrt{n}S$ and $\sqrt{n}S\cap \{-1,0,1\}^d.$ Previous results of this kind were less uniform; for instance \cite[Proposition 2.8.2]{NSDthesis} implies a variant of \eqref{eq:snBZ} with implicit constants depending on $n$ and $d$. The uniformity of these constants in $n$ and $d$ is crucial for our applications in the proof of Theorem \ref{thm:sphere}.  The authors are clueless, whether any results similar to Theorem \ref{thm:lat:extract} or Corollary \ref{cor:1.4} have been previously established in the literature. However, the case $n=cd$ for fixed $c>0$ as $d \to \infty$ was studied in \cite{HRou} and \cite{MO}.

It is worth to remark that Theorem \ref{thm:lat:extract} and Corollary \ref{cor:1.4} may be thought of as a concentration of measure results, stating that lattice points in the ball $\sqrt{n}B$ are concentrated on the sphere $\sqrt{n}S$ or even on the set $\sqrt{n}S\cap \{-1,0,1\}^d$ (when $n=O(n^{3/4})).$ This is a much stronger analog of the concentration of measure phenomenon for the ball in $\R^d$, where most of its mass is concentrated near the sphere.

\par
The second type of results concerns a multi-parameter maximal operator of a combinatorial nature, which may be useful for studying similar problems involving sets other than balls and spheres in the future.
Fix $K \in \N$ and take $\overline{n}=(n_1,n_2,...,n_K) \in (\N_0)^K.$ Let $D_{\overline{n}}$ denote the number of lattice points in $\{-K,\ldots,K\}^d$ such that exactly $n_1$ coordinates are equal to $\pm 1,$ $n_2$ coordinates are equal to $\pm 2,$ and so on; formally
\begin{equation}
\label{eq:def Dn}
D_{\overline{n}}= \bigcap_{j=1}^K\Big\{ x \in \{-K,...,K\}^d:  \#\{i \in [d]: |x_i|=j\}=n_j \Big\}.
\end{equation}
Consider the operator
\[
\mathcal{D}_{\overline{n}}f(x)= \frac{1}{|D_{\overline{n}}|} \sum_{y \in D_{\overline{n}}} f(x-y).
\]
Notice that $\mathcal{D}_{\overline{n}}$ is a multi-parameter averaging operator depending on the parameters $d$ and $K$. We prove that the corresponding multi-parameter maximal function has a dimension-free bound on $\ell^p(\Z^d),$ $p\in[2,\infty].$ Below the supremum is taken with respect to non-negative integers $n_1,\ldots,n_K.$
\begin{Thm} \label{thm:1.5} 
For all $K \in \N$ there is a constant $C_K>0$ depending only on $K$ and such that for all dimension $d\in\N$ and all exponents $p\in[2,\infty]$ we have
\[
 \Big\| \sup_{ n_1,n_2,...,n_K \leq \frac{d}{2K}} |\mathcal{D}_{n_1,...,n_K}f| \Big\|_{\ell^p(\Z^d)} \leq C_K \big\| f \big\|_{\ell^p(\Z^d)}, \qquad f\in\ell^p(\Z^d).
\]
\end{Thm}
\noindent Later we will see that the lattice point count mentioned in the previous paragraph allows us to reduce Theorem \ref{thm:sphere} to Theorem \ref{thm:1.5}.

It is worth highlighting that Theorem \ref{thm:1.5} gives a multi-parameter and dimension-free estimate. Such an estimate is rare for multi-parameter maximal functions as their estimates are usually dimension dependent.

\subsection{State of the art}

Dimension-free estimates for Hardy-Littlewood maximal operators were first studied in the continuous
context and thus we briefly describe the current state of the art. For every $t>0$ and $x\in\R^d$ we define the continuous Hardy--Littlewood averaging operator over balls by
\begin{align*}
M_t^{B^2(d)}f(x):=\frac{1}{|t B^2(d)|}\int_{tB^2(d)}f(x-y)dy, \qquad f\in L^1_{\rm loc}(\R^d).
\end{align*}
We denote by
$C(p,\R_+, B^2(d))$ the smallest constant in
the  maximal inequality
\begin{align}
\label{eq:34}
\big\|\sup_{t\in \R_+}|M_t^{B^2(d)}f|\big\|_{L^p(\R^d)}\le C(p,\R_+, B^2(d))\|f\|_{L^p(\R^d)},
\qquad f\in L^p(\R^d).
\end{align}
As in the discrete setting, standard arguments show that $C(p,\R_+, B^2(d))<\infty$ for every $p\in(1, \infty].$
In 1982 Stein \cite{SteinMax} (see
also Stein and Str\"omberg \cite{StStr}) proved that there
exists a constant $C_p>0$ depending only on $p\in(1, \infty]$ such
that
\begin{align}
\label{eq:41}
\sup_{d\in\N}C(p,\R_+, B^2(d))\le C_p.
\end{align}
Subsequently, Bourgain \cite{B1,B2,B3}, Carbery \cite{Car1}, and M\"uller \cite{Mul1} significantly extended Stein's result by replacing the Euclidean ball with various symmetric convex bodies $G.$ A symmetric convex body $G$ is a convex compact non-empty subset of $\R^d,$ which is symmetric and has non-empty interior. These sets are relevant, because any norm on $\R^d$ is associated with such a set $G$ via the Minkowski functional. Let $C(p,\R_+,G)$ be the best constant in \eqref{eq:34} with the ball $B^2(d)$ replaced by a general symmetric convex body $G$. In \cite{B1} Bourgain proved that $C(2,\R_+,G)\le C,$ were $C$ is a universal constant independent of both the symmetric convex body $G$ and the  dimension $d.$ This result was extended by Borugain \cite{B2} and independently by Carbery \cite{Car1}, who showed that  $C(p,\R_+,G)\le C_p$ for $p>3/2,$ where $C_p$ depends only on $p.$ For general symmetric convex bodies, it remains an open question whether Bourgain and Carbery's results can be extended beyond the threshold $p>3/2$. The conjecture that $C(p,\R_+,G)\le C_p,$ for $p>1$ is one of the major unsolved problems in the field.

However, a dimension-free estimate for all $p>1$ can be obtained for specific classes of symmetric convex bodies. The case of $\ell^q$ balls 
\begin{equation} 
\label{eq: Bqballs}
B^q=B^q(d)=\{x\in\R^d\colon \sum_{j=1}^d |x_j|^q \le 1\},\qquad 1\le q < \infty,
\end{equation}
was treated by M\"uller \cite{Mul1} who showed that $C(p,\R_+,B^q(d))\le C_{p,q}$ for $p\in (1,\infty].$ The picture for these balls was complemented by Bourgain \cite{B3} who proved that also for the cubes $B^{\infty}(d)=[-1,1]^d$ we have $C(p,\R_+,B^{\infty}(d))\le C_{p},$ $1<p<\infty.$ Additionally, Aldaz \cite{Ald1} showed that the weak-type $(1,1)$ constant for the cubes grows to infinity with the dimension. More recently, the second author in collaboration with Bourgain, Mirek, and Stein \cite{BMSW1} proved analogous dimension-free estimates for variational seminorm counterparts of these maximal inequalities. For a broader perspective on dimension-free estimates for continuous maximal functions, we refer the reader to the survey \cite{BMSW4} and the references therein.

The study of dimension-free inequalities for discrete Hardy-Littlewood maximal functions was initiated by the second author in collaboration with Bourgain, Mirek, and Stein. We now mention existing results, highlighting those that are most relevant to our work. Recall that  $\mC(p,\R_+,G)$ is defined in \eqref{eq:mCp}. In the following, we denote by $\mathbb D$ the set of dyadic numbers $\mathbb D=\{2^n\colon n\in \N\}.$

\begin{enumerate}

\item In \cite{BMSW3} Bourgain, Mirek, Stein, and the second author, constructed a counterexample - a family of ellipsoids $E(d)$ -  for which there is no dimension-free estimate on any $\ell^p(\Z^d),$ $1<p<\infty.$ Specifically, they showed that
$\mC(p,\R_+,E(d))\ge C_p (\log d)^{1/p}.$ This implies that a general result as in the continuous setting is impossible. Instead, one must focus on specific symmetric convex bodies in the discrete setting. Arguably the most natural choices are the balls $B^2(d)$ and the cubes $B^{\infty}(d).$

\item In \cite{balls} the same authors proved that the discrete maximal function over balls and dyadic scales has dimension-free estimates on $\ell^p(\Z^d),$ $p\ge 2,$ namely, $\mC(p,\mathbb D,B^2(d))\le C,$ for $p\ge 2.$ Later in \cite{MiSzWr} Mirek, Szarek, and the second author, proved an analogous result with spheres replacing balls showing that $\mC(p,\mathbb D,S^{d-1})\le C,$ for $p\ge 2$ and $d\ge 5.$ However, neither of these papers address dimension-free estimates for the full maximal functions over balls or spheres.

\item In \cite{BMSW3} the authors obtained a comparison principle between the discrete and continuous settings, valid for all symmetric convex bodies. This comparison principle states that $\mC(p,[c(G)d,\infty),G)\le e^6 C(p,\R_+,G),$ $p\in(1,\infty],$ where $c(G)$ is a constant depending on the symmetric convex body. Notably for $B^q$ balls we have $c(B^q)=d^{1/q}/2.$ Moreover, for $q\ge 2$ this bound was improved to the estimate $\mC(p,[d,\infty),B^q)\lesssim C(p,\R_+,B^q),$ see \cite{BMSW4} (case $q=2$) and \cite{KMPW} (case $q\in(2,\infty]$). Thus, in view of Stein's continuous dimension-free estimate \eqref{eq:41} the remaining range of radii to consider in Stein's Question \eqref{eq:36'} is $t=O(d).$

\item It should be noted that the comparison principle from (3) relies on the nestedness of the dilations $t_1G\subseteq t_2 G ,$ $0<t_1<t_2,$ which does not hold for all sets $G.$ This is the case with Euclidean spheres $G=S^{d-1},$ where even the (non-dimension-free) $\ell^p$ boundedness of the maximal function cannot be deduced from a corresponding continuous result and is a challenging problem by itself. It turns out that the spherical maximal function $\sup_{n\in \N}|\mathcal S_{\sqrt{n}}f|$ is bounded in dimensions $d\ge 5$ in the range $p>d/(d-2).$ This has been established by Magyar, Stein, and Wainger \cite{MaStWa}, see also Magyar \cite{Ma} for a previous result of this kind and Ionescu \cite{Ion1} for an endpoint estimate. 

\item In \cite{KMPW} Kosz, Mirek, Plewa, and the second author proved that the norm of the continuous maximal function is always smaller than the norm of the corresponding discrete maximal function, namely, $C(p,\R_+,G)\le \mC(p,\R_+,G),$ $p\in(1,\infty],$ for all symmetric convex bodies $G$ and $p\in(1,\infty].$ This shows in particular that a discrete dimension-free estimate is harder to achieve than its continuous counterpart.

\item In \cite{Ni1} and \cite{Ni2} the first author proved dimension-free estimates for dyadic maximal functions associated with $B^q$ balls $q \ge 1$ in various small scales regimes.

\item In \cite{BMSW3} it was shown that a dimension-free estimate holds for the full maximal function over the discrete cubes in the range $p>3/2.$ Namely, we have $\mC(p,\R_+,B^{\infty}(d))\le C_p,$ for $p>3/2$ together with a dyadic variant $\mC(p,\mathbb D,B^{\infty}(d))\le C_p,$ which is valid for all $p>1.$ This stands in sharp contrast with the analogous question for the balls for which no such results are known.

\end{enumerate}

\subsection{Our motivations and methods}

When we learned about Stein's Question \eqref{eq:36'} it was natural to conjecture that it had an affirmative answer. However,  the counterexample constructed in \cite{BMSW3} cast doubt on this presumption. The ellipsoids considered in that work are given by
\[
E(d)=\{x\in \R^d\colon \sum_{j=1}^d \lambda_j^2 x_j^2 \le 1\},
\]
where $1\le \lambda_1<\ldots\lambda_d<\sqrt{2}.$ At first glance, these ellipsoid appear to be only a minor distortion of the ball $B^2(d).$ Moreover, the failure of the dimension-free estimate occurs on very small scales, specifically for $t\le \sqrt{2}.$ More precisely, \cite[proof of Theorem 2, p.\ 871]{BMSW3} shows that already
\[
    \Big\| \sup_{t \in(0, \sqrt{2})} |\mathcal{M}^{E(d)}_tf| \Big\|_{\ell^2(\Z^d)} \geq C(\log d)^{1/2} \big\| f \big\|_{\ell^2(\Z^d)},
    \]
for a non zero function $f.$ This observation is part our motivation for studying a version of Stein's question restricted to small scales.

On the other hand \cite{balls} and \cite{MiSzWr} gave a positive answer to a dyadic variant of Stein's Question \eqref{eq:36'} for both balls and spheres. Moreover, in a recent preprint \cite{MiSzWrGau} Mirek, Szarek, and the second author considered the full maximal function associated with normalized averages over discrete Gasussians defined as $G_t f(x)= \sum_{y\in \Z^d}g_t(y)f(x-y),$ where for $t>0,$ 
\[
g_t(n) := \frac{1}{\Theta_{1/t}(0)}\exp\bigg(- \frac{\pi|n|^2}{t}\bigg),\quad \Theta_{1/t}(0)=\sum_{j\in \Z^d}\exp\bigg(- \frac{\pi|j|^2}{t}\bigg),\quad n\in \Z^d.
\]
Such a maximal function is pointwise dominated by the discrete maximal function corresponding to the Euclidean balls, namely $\sup_{t>0}\mathcal |G_t f(x)|\le \sup_{t>0}|\mathcal M_t^{B^2} f(x)|.$ It was proved in \cite{MiSzWrGau} that a dimension-free estimate holds for $\|\sup_{t>0} |G_t f|\|_{\ell^p(\Z^d)}$ on all $\ell^p(\Z^d)$ spaces, $1<p<\infty.$ This result further supports an affirmative answer to Stein's question.

Taken together, the above findings suggest that Question \eqref{eq:36'} may in fact have a positive answer. Our main results,  Theorem \ref{thm:ball} and \ref{thm:sphere} are steps towards confirming this supposition.

We shall now outline our methods.  The maximal operators from Theorems \ref{thm:ball}, \ref{thm:sphere} and \ref{thm:1.5} are clearly contractions on $\ell^{\infty}(\Z^d).$ Thus, by interpolation it is enough to establish these theorems for $p=2$. It is therefore not surprising that one of the most important tools in our paper and also in \cite{balls}, \cite{BMSW3}, \cite{MiSzWr}, are Fourier methods. A key component of this approach is obtaining pointwise estimates for the multiplier symbol \[m_t^G(\xi)=\frac{1}{|tG\cap \Z^d|}\sum_{x\in tG\cap \Z^d}\exp\big(-2\pi i x\cdot \xi\big),\qquad \xi\in [-1/2,1/2)^d,\] corresponding to the operator $\mathcal M_t ^G f$. The general strategy is as follows: 
\begin{enumerate}[(a)]
    \item To estimate the dyadic maximal function $\sup_{t\in \mathbb D} |\mathcal M_t^G f|$ one needs
appropriate (dimension-free) estimates for $|m_t^G(\xi)-1|$ when $|\xi|$ is small and for $m_t^G(\xi)$ when $|\xi|$ is large. This strategy has been successfully applied in \cite{balls}, \cite{BMSW3}, \cite{MiSzWr} among others. 
\item In order to estimate the full maximal function one must control, in an appropriate dimension-free manner, the difference $m_{t_1}^G(\xi)-m_{t_2}^G(\xi),$ where $t_1\neq t_2.$ This task has only been successfully accomplished for the cubes in \cite[Section 3]{BMSW3}.
\end{enumerate}
Unfortunately, it is not clear how to apply strategy (b) for any maximal function corresponding to the discrete averages $\mathcal M_t^G$ other than the one for the cubes $G=B^{\infty}(d)$. The primary obstacle is the absence of an explicit and useful formula for $m_{t}^G$ when $G$ is not the cube.

One of the main ideas of our paper is the departure from trying to establish (b) for multipliers corresponding to discrete balls or spheres. Instead, we compare their corresponding maximal functions with the multi-parameter maximal function from Theorem \ref{thm:1.5}. Remarkably, a multi-parameter variant of strategy (b) does apply for the multipliers $\beta_{n_1,\ldots,n_K}$  defined by \eqref{eq:bon}, which correspond to the averages $\mathcal{D}_{n_1,\ldots,n_K}.$ 

Our main source of inspiration is \cite[Section 3]{balls} though interpreted in an appropriate way. In this section of \cite{balls} the authors established the bound
\[
 \Big\| \sup_{t \leq d^{\frac{1}{2}}, t \in \mathbb{D}} |\mathcal{M}_tf| \Big\|_{\ell^2(\Z^d)} \leq C  \big\| f \big\|_{\ell^2(\Z^d)},\qquad f\in \ell^2(\Z^d).
\]
The first step in proving this result is \cite[Lemma 3.2]{balls} (restated here as Lemma \ref{lem:4.1}), which states that if $n$ is small, then a large proportion of points in $\sqrt{n}B \cap \Z^d$ has many coordinates equal to $\pm 1$. Then from the above mentioned lemma one may deduce that
\begin{equation} \label{Mb_to_D}
\Big\| \sup_{t \leq d^{\frac{1}{2}}, t \in \mathbb{D}} |\mathcal{M}_tf| \Big\|_{\ell^2(\Z^d)} \lesssim \Big\| \sup_{t \leq d/2, t \in \mathbb{D}} |\mathcal{D}_tf| \Big\|_{\ell^2(\Z^d)}+ \|f \|_{\ell^2(\Z^d)}
\end{equation}
and the proof is completed by showing that 
\begin{equation} \label{D_bound}
\Big\| \sup_{t \leq d/2, t \in \mathbb{D}} |\mathcal{D}_tf| \Big\|_{\ell^2(\Z^d)} \lesssim \|f \|_{\ell^2(\Z^d)}.
\end{equation}
The symbol $\mathcal{D}_t$ above denotes the operator \eqref{eq:def Dn} in the one-parameter case with $K=1$ and $t=n_1.$ Inequalities \eqref{Mb_to_D} and \eqref{D_bound} are not explicitly established in \cite{balls}, but they essentially follow from the methods developed there. Our approach is to extend these three results: \cite[Lemma 3.2]{balls}, \eqref{Mb_to_D} and \eqref{D_bound}, to remove the restriction $t \in \mathbb{D}$ and obtain an estimate for the full maximal function. More precisely, our strategy consists of the following steps:

\begin{enumerate}
\item Our improvement of \cite[Lemma 3.2]{balls} is Theorem \ref{thm:3.1} (for balls) and Corollary \ref{cor:3.2} (for spheres). These results aim to capture  simultaneously two different phenomena. The first one is a concentration of lattice points in a narrow annulus near the boundary. The second one is the observation that vast majority of points has nearly all their coordinates bounded by $K,$ provided the radius is sufficiently small in terms of $K$ and $d$. Note that the assumptions in Theorem \ref{thm:3.1} and Corollary \ref{cor:3.2} are stronger then those in  \cite[Lemma 3.2]{balls} but so is the conclusion. 

\item Our main tool for proving Theorem \ref{thm:3.1} and deducing Corollary \ref{cor:3.2} is the count for lattice points in balls and spheres of small radii \\ $n<cd$ established in Theorem \ref{thm:lat:extract} (see also Theorem \ref{thm:3.4} for a detailed statement). The feasibility of obtaining such estimates for $n\approx d$ was noted by Mazo and Odlyzko in \cite[Section 3]{MO}. Our approach, based on the saddle-point method, follows an idea briefly mentioned in \cite{MO} but without explicit details.

\item Our improvement of \eqref{Mb_to_D} is Proposition \ref{lem:4.3}, where we reduce the maximal function corresponding to discrete spheres to the maximal function corresponding to $D_{\overline{n}}.$ Here we  use Plancherel's theorem and a square function argument to pass from the multiplier symbol $m_t^S$ to other multipliers. Ultimately, the process leads to two multiplier symbols $\varphi_{\sqrt{n},0}$ and $\varphi_{\sqrt{n},1},$ whose corresponding maximal function is dominated in $\ell^2(\Z^d)$ by $\sup_{n_1,n_2,...,n_K \leq \frac{d}{2K}}|D_{\overline{n}}f|.$ The reduction is accomplished by examining lattice points in $\sqrt{n}S$ with coordinates equal to $\pm 1, \pm 2,..., \pm K,$ instead of only those with coordinates $\pm 1$ as in \cite[Section 3]{balls}. The fact that such lattice points constitute a large proportion of $\sqrt{n}S$ follows from Corollary \ref{cor:3.2}. 

\item   The final component of our argument is a generalization of \eqref{D_bound}, which is established in Theorem \ref{thm:1.5}. This result relies on analyzing Fourier multipliers $\beta_{\overline{n}}$ corresponding to the multi-parameter averages $D_{{\overline{n}}}$  with $\overline{n}=(n_1,\ldots,n_K).$ Our approach combines Plancherel's theorem, a square function argument, a comparison with certain discrete semigroups in Theorem \ref{thm:2.11}, and a multi-parameter Rademacher-Menshov type inequality from Lemma \ref{lem:2.15}. These tools are applicable because, as we already mentioned,  a multi-parameter version of strategy b) can be employed for the multipliers $\beta_{\overline{n}}.$ Specifically, we obtain appropriate estimates for $\beta_{\overline{n}}-\beta_{\overline{n'}},$ see Corollary \ref{cor:2.18}. This is because the multipliers $\beta_{\overline{n}}$ have combinatorial nature and by using combinatorial arguments one may largely reduce their analysis to properties of Krawtchouk polynomials (see Section \ref{subsec:2.1} for their definition and key properties).

\end{enumerate}

 It should be noted that in the range $n=O(d^{3/4-\varepsilon})$ it suffices to deal with a one-parameter maximal function  in Theorem \ref{thm:1.5}, with $K=1$. In this case the difference $\beta_{n}-\beta_{n'}$ can be essentially expressed in terms of Krawtchouk polynomials. However, going beyond this threshold to $d^{3/4-\varepsilon}<n<d^{1-\varepsilon}$ creates significant new difficulties and is the reason for introducing the multi-parameter maximal function.  The treatment of this maximal function in Theorem \ref{thm:1.5} is feasible because many desired properties of the multipliers $\beta_{\overline{n}}$ can be deduced from a proper understanding of the one-parameter multipliers $\beta_n.$ This is the case in Lemmas \ref{lem:2.8} and \ref{lem:2.17}. We remark that Lemma \ref{lem:2.8} is an important step for bounding both the dyadic and the full maximal function, while the formula devised in Lemma \ref{lem:2.17} is crucial for controlling the full maximal function. Both these lemmas are instrumental in the proof of Corollary \ref{cor:2.18}.

We conclude this section with a few remarks.
\begin{Rem}
\label{Rem:1}
Theorem \ref{thm:ball}, together with \cite[Theorem 2]{BMSW4}, implies that the remaining range of radii relevant to Stein's Question is $t\in(d^{1/2-\varepsilon},Cd).$ However, using our methods we do not even see any hope of removing $\varepsilon$ from Theorems \ref{thm:ball} and \ref{thm:sphere}. This is because, as $\varepsilon$ goes to $0,$ the behavior of lattice points in a ball of radius $d^{1/2-\varepsilon}$ becomes increasingly complex, making a reduction to Theorem \ref{thm:1.5} seem no longer possible. 
  \end{Rem}
\begin{Rem}
\label{Rem:2}
By tracing the constants it can be seen that Theorems \ref{thm:ball} and \ref{thm:sphere} involve implicit constants of the order $(C\varepsilon^{-1/2})^{\varepsilon^{-1/2}}.$ This is because the constant in our proof of Theorem \ref{thm:1.5} is of order $(CK)^{K}.$
\end{Rem}
\begin{Rem}
\label{Rem:3}
Our methods should also extend to maximal functions associated with other $B^q$ balls \eqref{eq: Bqballs}, provided the supremum in their definition is restricted to small radii. We plan to study this topic in future work.
\end{Rem}

\subsection{Structure of the paper}

We now briefly describe the structure of our paper. The notation is explained in the next Section \ref{ssec:not}.

Section \ref{sec:2} begins with preliminary results regarding the multipliers $\beta_{\overline{n}}$ and definitions and properties of Krawtchouk polynomials. The goal of this section is to prove Theorem \ref{thm:1.5}. The proof is based on estimating the full maximal function (over small scales) by its dyadic variant plus a sum of short maximal functions over dyadic intervals. The estimate for the dyadic maximal function is presented in Section \ref{sec:22}. In Section \ref{sec:23} we prove difference estimates for the multipliers $\beta_{\overline{n}}$ which are crucial for bounding the short maximal functions. In Section \ref{sec:24} we collect all these tools and complete the proof of Theorem \ref{thm:1.5}.

Section \ref{sec:3} is devoted to counting lattice points in Euclidean balls and spheres with small radii. The main tools which are needed for the proofs of Theorems \ref{thm:ball} and \ref{thm:sphere} are given by Theorem \ref{thm:3.1} and Corollary \ref{cor:3.2}. These results are justified with the aid of the lattice point count provided in Theorem \ref{thm:3.4}. In Section \ref{sec:3} we also justify Corollary \ref{cor:1.4}. We remark that Sections \ref{sec:2} and \ref{sec:3} are not related and can be read independently of each other.

Finally, in Section \ref{sec:4} we combine the results of Sections \ref{sec:2} and \ref{sec:3} and prove Theorem \ref{thm:sphere}. The connection between Theorem \ref{thm:1.5} and Theorem \ref{thm:sphere} is established in Proposition \ref{lem:4.3}.

\subsection{Notation}
\label{ssec:not}
We finish the introduction with the description of our notation.

{\bf Sets and their size}
\begin{enumerate}
\item We denote by $\N= \{1,2,... \}$ the set of positive integers and by $\N_0=\N \cup \{ 0\}$ the set of non-negative integers. Letters $n$ and $n_1,\ldots,n_K$ will always denote elements of $\N_0,$ in particular under suprema in maximal functions.
\item By $\mathbb{D}=\{2^{n}: n \in \N\}$ we denote the set of dyadic integers.
\item For $N\in \N$ we abbreviate $[N]=\{1,...,N\}.$
\item For two sets $X,Y$, by $X^Y$ we will denote the set of functions from $Y$ to $X$.
\item For a finite set $A$ we denote by $\#A$ or $|A|$ the number of its elements. The notation $|A|$ is used especially when $A$ is a subset of the integer lattice $\Z^d$.
\item In the reminder of the paper we abbreviate $tB=tB^2(d) \cap \Z^d$ and $tS=tS^{d-1}\cap \Z^d.$ This will cause no confusion as from now on we will only work on $\Z^d$ or its subsets. In particular in what follows $|tB|$ and $|tS|$ denote the number of lattice points in ball of radius $t$ and sphere of radius $t,$ respectively.
\end{enumerate}

{\bf Expectations}
\begin{enumerate}\addtocounter{enumi}{+6}
\item Let $A$ be a finite set of $\#A$ elements and let $h\colon A\to \mathbb{C}.$ We abbreviate \[\se{x \in A}h(x):=\frac{1}{\#A}\sum_{x\in A} h(x).\]

\item In terms of sums, suprema and expected values (normalized sums) (see point (7) above), the convention is the following. In the first line under a sum we write over which objects we are summing, then after the symbol ";" but still in the first line we write with which parameters are our objects parametrized. In next lines we write conditions on our objects. For instance
\begin{equation*}
\sum_{\substack{I_k \subseteq J  ; k \in V , \\
    (\forall k \in V ) |I_k|=n_k, \\
    (\forall j,k \in V , j \neq k) I_j \cap I_k=\emptyset
    }} h((I_k)_{k\in V})
\end{equation*}
means that we're taking a  sum  of $h((I_k)_{k\in V})$ over subsets $I_k$ of the set $J$ and these subsets are parametrized by $k$ from the set $V$, moreover these subsets satisfy some condition written in the second and third line under the sum. This notation also applies to suprema and expected values. In the case when $V=[K]$ we often write
$I_1,I_2,\ldots, I_K \subseteq J$  instead of $I_k \subseteq J; k \in V.$ A typical example of this convention appears already in Definition \ref{def:2.1} with $h((I_k)_{k\in [K]})=\prod_{k=1}^K \prod_{i \in I_k} \cos(2k \pi \xi_i),$ and $\xi$ being a fixed element of $\T^d.$
\end{enumerate}

{\bf Fourier transforms and related}
\begin{enumerate}\addtocounter{enumi}{+8}
\item For vectors $x,y\in\R^d$ we let $x \cdot y=\sum_{j=1}^d x_j y_j$ be their inner product.
\item For $s\in \R$ we abbreviate $e(s)=e^{-2\pi i s}.$
\item We identify the $d$-dimensional torus $\T^d:=\R^d/\Z^d$
with the unit cube $Q:=[-1/2, 1/2)^d.$

\item For $f\in \ell^1(\Z^d)$ we let 
\begin{equation*}
 \hat{f}(\xi) := \sum_{n\in\Z^d} f(n) e(n\cdot \xi),\qquad \xi\in\T^d,
\end{equation*}
be its Fourier series. Similarly, for $g\in L^1(\T^d)$ we let 
\begin{equation*}
 \mathcal F^{-1} g(n) := \int_{\T^d} g(\xi) e(-n\cdot \xi)\,d\xi,\qquad n\in\Z^d,
\end{equation*}
be its Fourier coefficient. 
\item Parseval's theorem takes the form
\[
 \|\hat{f}\|_{L^2(\T^d)}=\|f\|_{\ell^2(\Z^d)},\qquad \|g\|_{L^2(\T^d)}=\|\mathcal F^{-1} g\|_{\ell^2(\Z^d)}
\]
while inversion formula reads
\[
 \mathcal F^{-1}\hat{f}=f, \qquad \widehat{\mathcal{F}^{-1}g}=g. 
\]

\end{enumerate}

{\bf Operators}
\begin{enumerate}\addtocounter{enumi}{+13}
\item For a family of bounded commuting operators $T_i,$ $i\in U,$ on $\ell^2(\Z^d),$ we denote by  $\prod_{i\in U} T_i$ their composition product. We use the standard convention that  $\prod_{i \in \emptyset} T_i$ represents the identity operator. This convention also applies for numbers $a_i\in \mathbb{C}$ in which case $\prod_{i \in \emptyset} a_i=1.$

\item Let $T$ be a bounded operator on $\ell^2(\Z^d).$ We say that $T$ is positivity preserving if for every non-negative $f\in \ell^2(\Z^d)$ we have  $Tf(x)\ge 0,$ $x\in \Z^d.$

\item Recall that the spherical averaging operator $\mathcal S_t$ is defined for $t>0$ by
\[
\mathcal S_t f(x)=\frac{1}{|tS\cap \Z^d|}\sum_{y\in tS\cap \Z^d} f(x-y),\qquad x\in \Z^d.
\]

\end{enumerate}

{\bf Combinatorial symbols}
\begin{enumerate}\addtocounter{enumi}{+16}
\item For $d \in \N$ the symbol $\Sym(d)$ will denote the set of all permutations of the set $\{1,...,d \}$.
\item For $K \in \N$ and $n_1,...,n_K, d \in \N$ such that $n_1+...+n_K=d$ the multinomial coefficient is given by 
\[
\binom{d}{n_1,n_2,...,n_K}= \frac{d!}{\prod_{i=1}^K n_i!}.
\]
Note that for $K=1$ the symbol $\binom{d}{n_1,d-n_1}$ is exactly the binomial coefficient. 
\item We will frequently use the fact that $\binom{d}{n_1,n_2,...,n_K}$ is the number of partitions of a $d$ element set into $K$ subsets $I_1,\ldots,I_K$ such that $\#I_j=n_j,$ $j=1,\ldots,K.$   
\end{enumerate}

{\bf Asymptotic notation}
\begin{enumerate} \addtocounter{enumi}{+19}
\item Throughout the paper the letter $d\in \N$ is reserved for the dimension and all inexplicit constants will be
independent of $d$.  
\item For two nonnegative quantities $X, Y$
we write $X \lesssim_{\delta} Y$ if there is an absolute constant
$C_{\delta}>0$ depending only on $\delta>0$ such that $X\le C_{\delta}Y$ .
We  write $X \approx_{\delta} Y$ when
$X \lesssim_{\delta} Y\lesssim_{\delta} X$. We will omit the subscript
$\delta$ if the implicit constants are universal.
\item The above convention is also used for Big $O$ notation. In particular $X=O_{\delta}(Y)$ means that $|X|\lesssim_{\delta} |Y|.$
\end{enumerate}

{\bf Other abbreviations}
\begin{enumerate} \addtocounter{enumi}{+22}
\item For $x \in \mathbb{R}^d$ we define
\[
\supp(x)=\{i \in [d]: x_i \neq 0\}.
\]
\item 
\label{p: nota}
For each $K\in \N,$ $U\subseteq [K]$ and $\overline{n} \in (\N_0)^K$ we denote by  $\overline{n}(U) \in \N_0^K$ the vector which has $j$-th coordinate equal to $n_j$ if $j  \in U$ and $0$ if $j \not\in U$. On the other hand if we have numbers $n_j \in \N_0$ defined only for $j \in U$, then we will use the same notation $\overline{n}(U)$ for the vector of length $K$, whose $j$-th coordinate is $n_j$ for $j \in U$ and $0$ otherwise.
\item For $\xi\in \R^d$ we let $\| \xi \|^2:=\sum_{j=1}^d \sin^2(\pi \xi_j).$ Furthermore we abbreviate \[\xi+1/2=(\xi_1+1/2,\ldots,\xi_d+1/2), \|\xi +1/2 \|:=\|(\xi_1+1/2,\ldots,\xi_d+1/2)\|.\] We will frequently use the fact that
$\|\xi +1/2 \|^2= \sum_{j=1}^d \cos^2(\pi \xi_j)$.
\item For $p\in [1,\infty]$ we often write $\ell^p:=\ell^p(\Z^d)$ and abbreviate $\|f\|_p:=\|f\|_{\ell^p}$ for the norm of $f\in \ell^p.$
\end{enumerate}

\section{Bounds for the multi-parameter maximal function}
\label{sec:2}
In this section we will prove Theorem \ref{thm:1.5}.
Recall that the operator $\mathcal{D}_{\overline{n}},$ $\overline{n}\in (\N_0)^K,$ is given by convolution with $\frac{1}{|D_{\overline{n}}|} \mathds{1}_{D_{\overline{n}}}$. Hence we can study it by looking at its multiplier symbol
\begin{equation}
  \label{eq:bon}
    \beta_{\overline{n}}(\xi)=\se{x \in D_{\overline{n}}} e(x \cdot \xi),\qquad \xi\in\T^d,
   \end{equation}
on the Fourier transform side.  Our entire analysis will revolve around gathering information about $\beta_{\overline{n}}$.

\subsection{Preliminary results regarding the multiplier $\beta_{\overline{n}}$.}
\label{subsec:2.1}

We shall need the following family of multipliers.
\begin{Defn}
    \label{def:2.1}
   
    For a finite set $J \subseteq \N$, $\overline{n}=(n_1,n_2,...,n_K) \in \N_0^K$ satisfying $n_1+...+n_K \le |J| $ we define
    \[
    \beta_{\overline{n}}^J(\xi)=\se{\substack{I_1,I_2,...,I_K \subseteq J, \\
    (\forall j \in [K]) |I_j|=n_j, \\
    (\forall i,j \in [K], i \neq j) I_i \cap I_j=\emptyset
    }} \prod_{k=1}^K \prod_{i \in I_k} \cos(2k \pi \xi_i).
    \]
\end{Defn}
 Exploiting symmetries of the set $D_{\overline{n}}$ it is easy to see that $ \beta_{\overline{n}}=\beta_{\overline{n}}^{[d]}.$
\begin{Lem}
    \label{rem:2.2}
    For any $K \in \N$, $\overline{n}=(n_1,n_2,...,n_K) \in \N_0^K$ satisfying $n_1+...+n_K \le d $ and any $\xi \in \T^d$
    we have
    \[
    \beta_{\overline{n}}(\xi)=\beta_{\overline{n}}^{[d]}(\xi)=\se{\substack{I_1,I_2,...,I_K \subseteq [d], \\
    (\forall j \in [K]) |I_j|=n_j, \\
    (\forall i,j \in [K], i \neq j) I_i \cap I_j=\emptyset
    }} \prod_{k=1}^K \prod_{i \in I_k} \cos(2k \pi \xi_i).
    \]
\end{Lem}
\begin{proof}
 For any $\epsilon \in \{-1,1\}^d$ we have that
 \[
 x \in D_{\overline{n}} \Longleftrightarrow \epsilon \circ x \in D_{\overline{n}},
 \]
 where $\epsilon \circ x=(\epsilon_1 x_1,..., \epsilon_d x_d)$. This implies
 \[
 \beta_{\overline{n}}(\xi)= \se{\epsilon \circ x \in D_{\overline{n}}} e(( \epsilon \circ x) \cdot \xi)= \se{ x \in D_{\overline{n}}} e(( \epsilon \circ x) \cdot \xi).
 \]
 Hence 
 \begin{align*}
 \beta_{\overline{n}}(\xi)&= \se{ \epsilon \in \{-1,1 \}^d }\, \se{ x \in D_{\overline{n}}} e(( \epsilon \circ x) \cdot \xi)= \se{ x \in D_{\overline{n}}} \se{ \epsilon \in \{-1,1 \}^d } e(( \epsilon \circ x) \cdot \xi)
\\
 &=\se{ x \in D_{\overline{n}}} \prod_{i=1}^d \cos(2 \pi x_i\xi_i)= \frac{1}{|D_{\overline{n}}|} \sum_{\substack{I_1,I_2,...,I_K \subseteq [d], \\
    (\forall j \in [K]) |I_j|=n_j, \\
    (\forall i,j \in [K], i \neq j) I_i \cap I_j=\emptyset
    }} \prod_{k=1}^K \prod_{i \in I_k} \cos(2k \pi \xi_i)\\
    &\hspace{4cm}\cdot \#\{ x \in D_{\overline{n}}: (\forall k \in [K])\, i \in I_k \Leftrightarrow |x_i|=  k \}.
 \end{align*}
 Above we used the fact that cosine is an even function.
 Notice that for all $I_1,I_2,...,I_K \subseteq [d]$, such that $|I_j|=n_j$ and $I_i \cap I_j= \emptyset$ for $i \neq j$, we have that
 \[
 \#\{ x \in D_{\overline{n}}: (\forall k \in [K])\, i \in I_k \Leftrightarrow |x_i|=  k \}= 2^{\sum_{i=1}^K n_i}.
 \]
 Moreover it is easy to see that
 \[
 |D_{\overline{n}}|= 2^{\sum_{i=1}^K n_i} \cdot \binom{d}{n_1,n_2,...,n_K,d-\sum_{i=1}^Kn_i},
 \]
 combining all of this we get that
 \[
 \beta_{\overline{n}}(\xi)= \se{\substack{I_1,I_2,...,I_K \subseteq [d], \\
    (\forall j \in [K]) |I_j|=n_j, \\
    (\forall i,j \in [K], i \neq j) I_i \cap I_j=\emptyset
    }} \prod_{k=1}^K \prod_{i \in I_k} \cos(2k \pi \xi_i).
 \]
\end{proof}
Considering functions $\beta^J_{\overline{n}}$, which generalize $\beta_{\overline{n}}$, will come in handy in various arguments using induction.
It turns out that functions $\beta_{\overline{n}}^J$ have useful representation using Krawtchouk polynomials, which will be crucial ingredient in the proof of Theorem \ref{thm:1.5}.
\begin{Defn} \label{def:2.3}
 For every $n \in \N_0$ and integers $x,k \in [0,n]$ we define $k-th$ Krawtchouk polynomial
 \[
 \kr_{k}^{(n)}(x)= \frac{1}{\binom{n}{k}} \sum_{j=0}^k(-1)^j \binom{x}{j} \binom{n-x}{k-j}.
 \]
\end{Defn}
Here are all the facts about Krawtchouk polynomials that we will need.
\begin{Thm}\label{thm:2.4} For every $n \in \N _0$ and integers $x,k \in [0,n]$ we have
\begin{enumerate}
    \item{Symmetry:} $\kr _k^{(n)}(x)= \kr_x^{(n)}(k)$.
    \item{Reflection symmetry:} $\kr ^{(n)}_k(n-x)=(-1)^k \kr _k^{(n)}(x)  $.
    \item{Uniform bound:} There is $c \in (0,1)$ (independent of $k,n,x$) such if $0 \leq x,k \leq n/2$, then we have
    \[|\kr ^{(n)}_k(x)| \leq e^{-ckx/n}.\]
    \item{Formula for difference in $k$:} If $n \geq 2$, then
    \[
    \kr_{k+2}^{(n)}(x)-\kr_{k}^{(n)}(x)= \frac{-4x(n-x)}{n(n-1)} \kr_{k}^{(n-2)}(x-1),
    \]
    where the right-hand side is 0 if  $x=0 $ or $x=n$. 
\end{enumerate}
\end{Thm}
\noindent Proofs of the first two points are straightforward, third point is \cite[Lemma 2.2]{HKS}, last point follows from symmetry combined with equalities (44) and (45) of \cite[Corollary 3.1]{Lev} (note that their definition of Krawtchouk polynomials is a bit different). We remark that in point (3) one may take $c=-2 \log 0.93 \approx 0.145$.  This is not explicitly stated in \cite[Lemma 2.2]{HKS} but can be deduced from its proof.
\par Now we can state the representation of $\beta_{\overline{n}}^J(\xi)$ using Krawtchouk polynomials, which in lesser generality was noticed in the proof of \cite[Proposition 3.3]{balls}. 
\begin{Lem} \label{rem:2.5}
For any $K \in \N$, finite set $J \subseteq \N$, $\overline{n}=(n_1,n_2,...,n_K) \in \N_0^K$ satisfying $n_1+...+n_K \le |J| $, $j \in [K]$ and any $\xi \in \T^d$ we have
\begin{align*}
&\beta^J_{\overline{n}}(\xi)= \se{\substack{I_1,...,I_{j-1},I_{j+1},...,I_K \subseteq J, \\
    (\forall k \in [K] \setminus \{j \}) |I_k|=n_k, \\
    (\forall i,k \in [K] \setminus \{j\}, i \neq k) I_i \cap I_k=\emptyset
    }} \Big( \prod_{k=1, k \neq j}^K \prod_{i \in I_k} \cos(2k \pi \xi_i) 
  \\
&\cdot \sum_{U \subseteq J \setminus ( \cup_{k \neq j} I_k )} \kr_{n_j}^{(|J|- \sum_{k \neq j} n_k)}(|U|) \cdot \prod_{i \in J \setminus ( \cup_{k \neq j}I_k) \setminus U} \cos^2(j \pi \xi_i) \cdot \prod_{i \in U} \sin^2(j \pi \xi_i) \Big).
\end{align*}
\end{Lem}
\begin{proof}
Fix $j$, for simplicity of notation for fixed sets $I_1,...,I_{j-1},I_{j+1},...,I_K$ we denote
\[
\widetilde{I}= \cup_{k \neq j} I_k.
\]
Then we notice that
\begin{equation} \label{eq:2.1}
\beta_{\overline{n}}^J(\xi)=\se{I_1,...,I_{j-1},I_{j+1},...,I_K \subseteq J, \\
    (\forall k \in [K] \setminus \{j \}) |I_k|=n_k, \\
    (\forall i,k \in [K] \setminus \{j\}, i \neq k) I_i \cap I_k=\emptyset
    } \prod_{k=1, k \neq j}^K \prod_{i \in I_k} \cos(2k \pi \xi_i) \cdot
    \se{\substack{I_j \subseteq J \setminus \widetilde{I}, \\ |I_j|=n_j}} \prod_{i \in I_j} \cos(2j \pi \xi_i) .
\end{equation}
Fixing $I_1,...,I_{j-1},I_{j+1},...,I_K$ we will look only at last term of the equality above. Let $\epsilon(I_j) \in \{-1,1\}^J$ be such that $\epsilon(I_j)_i=\epsilon(I_j)(i)=-1$ exactly when $i \in I_j$, then we have
\begin{align*}
\prod_{i \in I_j} \cos(2j \pi \xi_i)&= \prod_{i \in J \setminus \widetilde{I}} \Big(  \frac{1+\cos(2 j \pi \xi_i)}{2}+ \epsilon(I_j)_i \frac{1- \cos(2 j \pi \xi_i)}{2} \Big)\\
&=\prod_{i \in J \setminus \widetilde{I}} \Big(  \cos^2(j \pi \xi_i)+ \epsilon(I_j)_i \sin^2(j \pi \xi_i) \Big)
\\
&= \sum_{U \subseteq J \setminus \widetilde{I}} w_U(\epsilon(I_j)) \prod_{i \in J \setminus \widetilde{I} \setminus U} \cos^2(j \pi \xi_i) \prod_{i \in U} \sin^2(j \pi \xi_i), 
\end{align*}
where we defined $w_U: \{-1,1 \}^{J} \to \{-1, 1\}$ by formula $w_U(\epsilon)= \prod_{i \in U} \epsilon_{i}$. Therefore we see that
\begin{equation}
\label{eq:Lem2.5calc}
\begin{split}
&\se{\substack{I_j \subseteq J \setminus \widetilde{I}, \\ |I_j|=n_j}} \prod_{i \in I_j} \cos(2j \pi \xi_i)=\frac{1}{\binom{|J|- \sum_{i \neq j}n_i}{n_j}} \sum_{\substack{I_j \subseteq J \setminus \widetilde{I}, \\ |I_j|=n_j}} \prod_{i \in I_j} \cos(2j \pi \xi_i)
\\
&= \frac{1}{\binom{|J|- \sum_{i \neq j}n_i}{n_j}} \sum_{\substack{I_j \subseteq J \setminus \widetilde{I}, \\ |I_j|=n_j}}  
\sum_{U \subseteq J \setminus \widetilde{I}} w_U(\epsilon(I_j)) \prod_{i \in J \setminus \widetilde{I} \setminus U} \cos^2(j \pi \xi_i) \prod_{i \in U} \sin^2(j \pi \xi_i)
\\
&=\sum_{U \subseteq J \setminus \widetilde{I}} \prod_{i \in J \setminus \widetilde{I} \setminus U} \cos^2(j \pi \xi_i) \prod_{i \in U} \sin^2(j \pi \xi_i) \frac{1}{\binom{|J|- \sum_{i \neq j}n_i}{n_j}} \sum_{\substack{I_j \subseteq J \setminus \widetilde{I}, \\ |I_j|=n_j}}
 w_U(\epsilon(I_j)).
 \end{split}
 \end{equation}
Denote $m_j=\sum_{i \neq j}n_i$ and observe that for any $U \subseteq J \setminus \widetilde{I}$ we have
\begin{align*}
& \frac{1}{\binom{|J|- m_j}{n_j}} \sum_{\substack{I_j \subseteq J \setminus \widetilde{I}, \\ |I_j|=n_j}} w_U(\epsilon(I_j))= \frac{1}{\binom{|J|-m_j}{n_j}}\sum_{\substack{I_j \subseteq J \setminus \widetilde{I}, \\ |I_j|=n_j}} (-1)^{|U \cap I_j|}
\\
&=\frac{1}{\binom{|J|- m_j}{n_j}}\sum_{m=0}^{n_j} (-1)^m \cdot \# \{I_j \subseteq J \setminus \widetilde{I} : |I_j|=n_j, |U \cap I_j|=m  \}
\\
&= \frac{1}{\binom{|J|- m_j}{n_j}}\sum_{m=0}^{n_j} (-1)^m \binom{|U|}{m} \binom{|J|-m_j-|U|}{n_j-m}= \kr_{n_j}^{(|J|- m_j)}(|U|).
\end{align*}
Thus, coming back to \eqref{eq:Lem2.5calc} we obtain
\begin{equation}
\label{eq:Lem2.5calc'}
\begin{split}
&\se{\substack{I_j \subseteq J \setminus \widetilde{I}, \\ |I_j|=n_j}} \prod_{i \in I_j} \cos(2j \pi \xi_i)\\
&= \sum_{U \subseteq J \setminus \widetilde{I}} \prod_{i \in J \setminus \widetilde{I} \setminus U} \cos^2(j \pi \xi_i) \prod_{i \in U} \sin^2(j \pi \xi_i) \kr_{n_j}^{(|J|- \sum_{k \neq j}n_k)}(|U|).
\end{split}
\end{equation}
Finally plugging the above into \eqref{eq:2.1} we get the conclusion of the lemma.
\end{proof}
Using Lemma \ref{rem:2.5} and bounds for Krawtchouk polynomials we will be able to deduce crucial Lemma \ref{lem:2.8} exactly in the same manner as in the proof of \cite[Proposition 3.3]{balls}, however before that we need to introduce an auxiliary lemma from \cite{balls}.
\begin{Lem}[{\cite[Lemma 2.6]{balls}}]
\label{lem:2.6} Assume that we have a sequence $(u_j: j \in [d])$ with $0 \leq u_j \leq \frac{1- \delta_0}{2}$ for some $\delta_0 \in (0,1)$. Suppose that $I \subseteq [d]$ satisfies $\delta_1d \leq |I| \leq d$ for some $ \delta_1 \in (0,1]$. Then for every $J=(d_0,d] \cap \Z$ with $0 \leq d_0 \leq d$ we have
\[
\se{\tau \in \Sym(d)} \exp\Big(- \sum_{j \in \tau(I) \cap J} u_j\Big) \leq 3 \exp\Big(-\frac{\delta_0 \delta_1}{20} \sum_{j \in J} u_j \Big).
\]
\end{Lem}
\noindent In \cite[Lemma 2.6]{balls} there is an extra assumption, that $u_j$ is decreasing, however this assumption can be removed by rearranging the sequence $(u_j: j \in [d])$, see \cite[Lemma 6.7]{MiSzWr}. We will use the following version of Lemma \ref{lem:2.6}.
\begin{Cor} \label{cor:2.7}
Fix a finite set $J \subseteq \N$. Assume that we have a sequence $(u_j :j \in J)$ with $0 \leq u_j \leq \frac{1-\delta_0}{2}$ for some $\delta_0 \in (0,1)$. Then for any nonnegative integer $k \leq |J|$ we have
\[
\se{\substack{I\subseteq J, \\ |I|=k}} \exp\Big( - \sum_{j \in I} u_j\Big) \leq 3 \exp\Big(- \frac{\delta_0k}{20|J|} \sum_{j \in J} u_j \Big)
\]
\end{Cor}
\begin{proof}
Let $m=|J|$. Since there is a bijection between $J$ and $[m]$, without loss of generality  we can assume that $J=[m]$. Fix $I_0 \subseteq [m]$, $|I_0|=k$. For each $I \subseteq [m]$, $|I|=k$ there are exactly $k!(m-k)!$ permutations $\sigma \in \Sym(m)$ such that $\sigma(I_0)=I$. Using Lemma \ref{lem:2.6} with $\delta_1= \frac{k}{m}$ we obtain
\[
\se{\substack{I \subseteq [m], \\ |I|=k}} \exp\Big(- \sum_{j \in I} u_j \Big)= \se{\sigma \in \Sym(m)} \exp\Big(- \sum_{j \in \sigma(I_0)} u_j \Big) \leq 3
\exp\Big(- \frac{\delta_0k}{20m}\sum_{j \in [m]} u_j \Big)
\]
\end{proof}
\par Now we are ready to state a crucial bound for $\beta_{\overline{n}}^J(\xi)$, which will come in handy in the following subsections.
\begin{Lem} \label{lem:2.8}
For any $K \in \N$, finite set  $J \subseteq \N$, $K$-tuple  $\overline{n} \in \N_0^K$ satisfying \\ $ \sum_{i=1}^K n_i + \max_{1 \leq i \leq K} n_i \leq |J|$ and any $\xi \in \T^d$ we have
    \[
    |\beta^J_{\overline{n}}(\xi)| \leq 6 \prod_{j=1}^K \exp\Big(- \frac{cn_j}{80 K|J|} \min\Big( \sum_{i \in J} \sin^2(j \pi \xi_i), \sum_{i \in J} \cos^2(j \pi \xi_i)\Big) \Big),
    \]
    where $c$ is the constant from Theorem \ref{thm:2.4}.
\end{Lem}
\begin{proof}
Fix $j \in [K]$ and denote $m_j=\sum_{k \neq j} n_k.$ From Lemma \ref{rem:2.5} we get that
\begin{align*}
&| \beta_{\overline{n}}^J(\xi)| \leq \se{\substack{I_1,...,I_{j-1},I_{j+1},...,I_K \subseteq J, \\
    (\forall k \in [K] \setminus \{j \}) |I_k|=n_k, \\
    (\forall i,k \in [K] \setminus \{j\}, i \neq k) I_i \cap I_k=\emptyset
    }} 1\\
& \hspace{0.5cm}\cdot   
\sum_{U \subseteq J \setminus ( \cup_{k \neq j} I_k )} |\kr_{n_j}^{(|J|- m_j)}(|U|)| \cdot \prod_{i \in J \setminus ( \cup_{k \neq j}I_k) \setminus U} \cos^2(j \pi \xi_i) \cdot \prod_{i \in U} \sin^2(j \pi \xi_i)
\\
&= \frac{\prod_{i \neq j} n_i! \cdot (|J|-m_j)!}{|J|!} \sum_{\substack{\widetilde{I} \subseteq J, \\ |\widetilde{I}|= m_j}} 
\sum_{U \subseteq J \setminus \widetilde{I}} |\kr_{n_j}^{(|J|- m_j)}(|U|)|\\
&\hspace{0.5cm}\cdot \prod_{i \in J \setminus \widetilde{I} \setminus U} \cos^2(j \pi \xi_i) \cdot \prod_{i \in U} \sin^2(j \pi \xi_i)  
\cdot \sum_{\substack{I_1,...,I_{j-1},I_{j+1},...,I_K \subseteq J, \\
    (\forall k \in [K] \setminus \{j \}) |I_k|=n_k, \\
    (\forall i,k \in [K] \setminus \{j\}, i \neq k) I_i \cap I_k=\emptyset, \\ \cup_{k \neq j} I_k=\widetilde{I}
    }} 1.
\end{align*}
However notice that for a fixed $\widetilde{I} \subseteq J$, $|\widetilde{I}|= m_j$ we have
\[\sum_{\substack{I_1,...,I_{j-1},I_{j+1},...,I_K \subseteq J, \\
    (\forall k \in [K] \setminus \{j \}) |I_k|=n_k, \\
    (\forall i,k \in [K] \setminus \{j\}, i \neq k) I_i \cap I_k=\emptyset, \\ \cup_{k \neq j} I_k= \widetilde{I}
    }} 1= \frac{(m_j)!}{\prod_{k \neq j}n_k!}
    \]
and thus we obtain
\begin{equation} \label{eq:2.2}
|\beta_{\overline{n}}^J(\xi)| \leq \se{\substack{\widetilde{I} \subseteq J, \\ |\widetilde{I}|=m_j}} \sum_{U \subseteq J \setminus \widetilde{I}} |\kr_{n_j}^{(|J|- m_j)}(|U|)| \cdot \prod_{i \in J \setminus \widetilde{I} \setminus U} \cos^2(j \pi \xi_i) \cdot \prod_{i \in U} \sin^2(j \pi \xi_i).  
\end{equation}
By assumptions on $n_j$ we have 
\[
n_j \leq \frac{|J|-m_j}{2}.
\]
Then for any $U \subseteq J \setminus \widetilde{I}$ such that $|U| \leq \frac{|J|-m_j}{2}$ by item (3) of Theorem \ref{thm:2.4} we get
\[
|\kr_{n_j}^{(|J|- m_j)}(|U|)| \leq \exp\Big( - \frac{cn_j|U|}{|J|-m_j} \Big),
\]
on the other hand if $|U|> \frac{|J|-m_j}{2}$, then by items $(2)$ and $(3)$ of Theorem \ref{thm:2.4} we obtain
\[
|\kr_{n_j}^{(|J|- m_j)}(|U|)|=|\kr_{n_j}^{(|J|- m_j)}(|J|-m_j-|U|)| \leq \exp\Big( - \frac{cn_j(|J|-m_j-|U|)}{|J|-m_j} \Big),
\]
so in both cases we have
\[
|\kr_{n_j}^{(|J|- m_j)}(|U|)| \leq \exp\Big( - \frac{cn_j|U|}{|J|-m_j} \Big)+ \exp\Big( - \frac{cn_j(|J|-m_j-|U|)}{|J|-m_j} \Big).
\]
From the above inequalities for fixed $\widetilde{I} \subseteq J, |\widetilde{I}|=m_j$ we obtain
\begin{equation*}
\sum_{U \subseteq J \setminus \widetilde{I}} |\kr_{n_j}^{(|J|- m_j)}(|U|)| \cdot \prod_{i \in J \setminus \widetilde{I} \setminus U} \cos^2(j \pi \xi_i) \cdot \prod_{i \in U} \sin^2(j \pi \xi_i) \\ \le S_1+S_2,
\end{equation*}
where
\begin{align*}
&S_1:= \sum_{U \subseteq J \setminus \widetilde{I}} \exp\Big( - \frac{cn_j|U|}{|J|-m_j} \Big)  \cdot \prod_{i \in J \setminus \widetilde{I} \setminus U} \cos^2(j \pi \xi_i) \cdot \prod_{i \in U} \sin^2(j \pi \xi_i)
\\
&S_2:=  \sum_{U \subseteq J \setminus \widetilde{I}} \exp\Big( - \frac{cn_j(|J|-m_j-|U|)}{|J|-m_j} \Big)  \cdot \prod_{i \in J \setminus \widetilde{I} \setminus U} \cos^2(j \pi \xi_i) \cdot \prod_{i \in U} \sin^2(j \pi \xi_i).
\end{align*}
For $S_1$ we have
\begin{align*}
&S_1= \sum_{U \subseteq J \setminus \widetilde{I}}   \prod_{i \in J \setminus \widetilde{I} \setminus U} \cos^2(j \pi \xi_i) \cdot \prod_{i \in U} e^{-\frac{cn_j}{|J|-m_j}}\sin^2(j \pi \xi_i)\\
&= \prod_{i \in J \setminus \widetilde{I}} \big(\cos^2(j \pi \xi_i)+e^{-\frac{cn_j}{|J|-m_j}}\sin^2(j \pi \xi_i) \big)=\prod_{i \in J \setminus \widetilde{I}}\big(1-(1-e^{-\frac{cn_j}{|J|-m_j}}) \sin^2(j \pi \xi_i) \big)\\
&\leq \exp\Big( -\frac{cn_j}{2(|J|-m_j)} \sum_{i \in J \setminus \widetilde{I}} \sin^2(j \pi \xi_i) \Big),
\end{align*}
where at the end we have used two elementary inequalities: $xe^{-x/2} \leq 1-e^{-x}$ and \\ $1-x \leq e^{-x}$. Similarly we obtain
\begin{align*}
&S_2=  \prod_{i \in J \setminus \widetilde{I}} \big(e^{-\frac{cn_j}{|J|-m_j}}\cos^2(j \pi \xi_i)+\sin^2(j \pi \xi_i) \big)
\\
&=\prod_{i \in J \setminus \widetilde{I}}\big(1-(1-e^{-\frac{cn_j}{|J|-m_j}}) \cos^2(j \pi \xi_i) \big) \leq \exp\Big( -\frac{cn_j}{2(|J|-m_j)} \sum_{i \in J \setminus \widetilde{I}} \cos^2(j \pi \xi_i) \Big).
\end{align*}
Plugging this into \eqref{eq:2.2} we get
\begin{align*}
|\beta_{\overline{n}}^J(\xi)| &\leq \se{\substack{\widetilde{I} \subseteq J, \\ |\widetilde{I}|=m_j}} \Big(\exp\Big( -\frac{cn_j}{2(|J|-m_j)} \sum_{i \in J \setminus \widetilde{I}} \sin^2(j \pi \xi_i) \Big)\\
&\hspace{5cm}+\exp\Big( -\frac{cn_j}{2(|J|-m_j)} \sum_{i \in J \setminus \widetilde{I}} \cos^2(j \pi \xi_i) \Big) \Big)
\\
&=\se{\substack{\widetilde{I} \subseteq J, \\ |\widetilde{I}|=|J|-m_j}} \Big(\exp\Big( -\frac{cn_j}{2(|J|-m_j)} \sum_{i \in \widetilde{I}} \sin^2(j \pi \xi_i) \Big)\\
&\hspace{4cm}+\exp\Big( -\frac{cn_j}{2(|J|-m_j)} \sum_{i \in  \widetilde{I}} \cos^2(j \pi \xi_i) \Big) \Big).
\end{align*}
Now using Corollary \ref{cor:2.7} twice for $J=J$, $k=|J|-m_j$ $\delta_0=\frac{1}{2}$, and sequences $(u_i: i \in J)=(  \frac{cn_j}{2(|J|-m_j)}\sin^2(j \pi \xi_i): i \in J)$ and $(u_i: i \in J)=(\frac{cn_j}{2(|J|-m_j)}\cos^2(j \pi \xi_i): i \in J),$ respectively, we obtain
\begin{align*}
|\beta_{\overline{n}}^J(\xi)| &\leq 3  \exp\Big( - \frac{cn_j}{80|J|} \sum_{i \in J} \sin^2(j \pi \xi_i) \Big)+ 3\exp\Big( - \frac{cn_j}{80|J|} \sum_{i \in J} \cos^2(j \pi \xi_i) \Big)
\\
&\leq 6 \exp\Big(- \frac{cn_j}{80|J|} \min\Big( \sum_{i \in J} \sin^2(j \pi \xi_i), \sum_{i \in J} \cos^2(j \pi \xi_i)\Big) \Big).
\end{align*}
Since $j \in [K]$ was chosen arbitrarily we can take the geometric mean of such quantities for all $j \in [K]$ and obtain
\[
|\beta_{\overline{n}}^J(\xi)| \leq 6 \prod_{j=1}^K \exp\Big(- \frac{cn_j}{80 K|J|} \min\Big( \sum_{i \in J} \sin^2(j \pi \xi_i), \sum_{i \in J} \cos^2(j \pi \xi_i)\Big) \Big).
\]
\end{proof}
\subsection{Estimates for the dyadic maximal operator}
\label{sec:22}
In this subsection we will obtain the following useful partial result towards Theorem \ref{thm:1.5}.
\begin{Prop} \label{prop:2.9}
For all $K \in \N$ there is a constant $C_K>0$ such that for all $\epsilon \in \{0,1 \}^K$, $d \in \N$ and $f \in \ell^2(\Z^d)$ we have
\[
 \Big\| \sup_{\substack{
 n_1+\epsilon_1,n_2+\epsilon_2,...,n_K+\epsilon_K \in \D, \\
 n_1,n_2,...,n_K \leq \frac{d}{2K}}} |\mathcal{D}_{n_1,...,n_K}f| \Big\|_{2} \leq C_K \big\| f \big\|_{2}.
\]
\end{Prop}
\noindent Necessity of including $\epsilon \in \{0,1 \}^K$ in the theorem above will be clear in the next subsection. It is due to the fact that in the  Lemma \ref{lem:2.17} it is crucial to compare two multipliers, whose coordinates have the same parity. \par
Let us introduce multipliers, which in some sense (see Lemma \ref{lem:2.12})  will approximate $\beta_{\overline{n}}(\xi)$.
\begin{Defn} \label{def:2.10}
For $t \ge 0$ we define $p_t: \T^d \to \mathbb{C}$ and $\widetilde{p_t}: \T^d \to \mathbb{C}$ by formulas
\begin{align}
\label{eq: ptxi} 
p_t(\xi)&= \exp\Big( - \frac{t}{d} \sum_{i=1}^d \sin^2(\pi \xi_i) \Big)=\exp\Big(- \frac{t}{d}  \| \xi \|^2 \Big),\\
\label{eq: pttilxi}
\widetilde{p_t}(\xi)&=\exp\Big( - \frac{t}{d} \min\Big( \sum_{i=1}^d \sin^2(\pi \xi_i),\sum_{i=1}^d \cos^2(\pi \xi_i)\Big) \Big)=\exp\Big(- \frac{t}{d} \min\big( \| \xi \|^2, \| \xi+1/2 \|^2\big) \Big)
\end{align}
For any $k \in \N$ and $t\ge 0 $ we define multiplier operators $P_{k,t,0} : \ell^2(\Z^d) \to \ell^2(\Z^d)$ and $P_{k,t,1}: \ell^2(\Z^d) \to \ell^2(\Z^d)$ given on the Fourier transform side by the formulas
\[
(\widehat{P_{k,t,0}f})(\xi)= p_t(k \xi) \widehat{f}(\xi),\qquad (\widehat{P_{k,t,1}f})(\xi)= p_t(k \xi+ 1/2) \widehat{f}(\xi).
\]
\end{Defn}

It is known that $\{P_{1,t,0}\}_{t>0}$ is a symmetric diffusion semigroup in the sense of Stein \cite[Chapter III]{Ste1}. This was observed e.g.\ in \cite[Lemma 4.1]{BMSW3} and became an important ingredient in \cite{balls,BMSW3,MiSzWr,Ni1}. It turns out that for each $k\in \N$ also $\{P_{k,t,0}\}_{t>0}$ is a symmetric diffusion semigroup.  This fact together with other important properties of the objects defined above is described in the following theorem.

\begin{Thm} \label{thm:2.11}
Fix $k\in \N.$ Then $\{P_{k,t,0}\}_{t>0}$ is a symmetric diffusion semigroup in the sense of Stein \cite[Chapter III]{Ste1}. In particular it is positivity preserving and there is a universal constant $C>0$ such that for all $d,k \in \N$ and $f \in \ell^2(\Z^d)$ we have
\begin{equation}
\label{eq:maxPkn0}
 \Big\| \sup_{t>0} |P_{k,t,0} f| \Big\|_{2} \leq C \big\|f \big\|_{2}
\end{equation}
Additionally, for each  $t>0$ the operator $P_{k,t,0}$ dominates pointwise $P_{k,t,1}$ in the following sense
\begin{equation}
\label{eq:Pkn1<Pkn0}
|P_{k,t,1}f(x)|\le P_{k,t,0}[|f|](x),\qquad x\in \Z^d.
\end{equation} 
Consequently, with the same constant $C$ we also have
\begin{equation}
\label{eq:maxPkn1}
 \Big\| \sup_{t>0} |P_{k,t,1} f| \Big\|_{2} \leq C \big\|f \big\|_{2}.
\end{equation}
\end{Thm}

\begin{proof}
Throughout the proof we fix $k\in \N.$ For each $i\in [d]$ we let $e_i$ be the $i$-th vector of the standard basis, define the operators 
    \[
    T_i f(x)=\frac12[f(x+ke_i)+f(x-ke_i)],\qquad T=\sum_{i=1}^d T_i,
    \]
    and let $\Delta_i=\frac12[\Id -T_i],$ $\Delta=\sum_{i=1}^d \Delta_i.$ Note that these operators are bounded on $\ell^p$ spaces and $\Delta=\frac12 [d\Id-T].$ Since $\widehat{(T_i f)}(\xi)=\cos(2\pi k\xi_i)\hat{f}(\xi)$ it is easy to see that
$\widehat{(\Delta_i f)}(\xi)=\sin^2(\pi k \xi_i)\hat{f}(\xi)$ so that $\widehat{(\Delta f)}(\xi)=\sum_{i=1}^d\sin^2(\pi k \xi_i)\hat{f}(\xi).$ Therefore, the operator $P_{k,t,0}$ may be rewritten as $P_{k,t,0}=\exp(-\frac{t}{d}\Delta).$ This expression implies that $P_{k,t,0}$ is positivity preserving. Indeed, taking $s=t/d>0$ we have
\begin{equation}
\label{eq:esDel}
P_{k,t,0}=\exp(-s\Delta)=\exp(-sd/2)\exp(sT/2)=\exp(-t/2)\sum_{r=0}^{\infty}\frac{t^rT^r}{(2d)^rr!}.
\end{equation}
Since each $T_i$ is positivity preserving so are $T,$ $T^r$ and $P_{k,t,0}.$

Expression \eqref{eq:esDel} also implies that $\{P_{k,t,0}\}_{t>0}$ is a symmetric diffusion semigroup in the sense of Stein \cite[Chapter III]{Ste1}. We have just proved that $P_{k,t,0}$ is positivity preserving. Obviously, it is  also self-adjoint and $P_{k,t,0}(\mathds{1})=\mathds{1}$ in view of \eqref{eq:esDel} and the observation that $T\mathds{1}=d \mathds{1}.$ Finally, the contraction property $\|P_{k,t,0} f\|_{p}\le \|f\|_p$ also follows from \eqref{eq:esDel} and the norm bound $\|T f\|_p\le d\|f\|_p.$ Thus, we have verified that 
$\{P_{k,t,0}\}_{t>0}$ is a symmetric diffusion semigroup. Hence, applying  \cite[Chapter
III, Section 3, p.73]{Ste1} or \cite[Theorem 7]{Cow1} we obtain \eqref{eq:maxPkn0}. We remark that the constant $C$ in \eqref{eq:maxPkn0} is universal among symmetric diffusion semigroups, in particular it is independent of both $k$ and $d.$ This is implicit in the Stein's proof \cite[Chapter
III, Section 3, p.73]{Ste1} and explicit in Cowling's \cite[Theorem 7]{Cow1}.

To complete the proof of the theorem it remains to demonstrate \eqref{eq:Pkn1<Pkn0}. Note that once this inequality is justified then \eqref{eq:maxPkn1} will follow from \eqref{eq:maxPkn0}. By definition we have \[
p_{k,t,1}(\xi)= \exp\Big( - \frac{t}{d} \sum_{i=1}^d \sin^2(\pi k\xi_i+\pi/2) \Big)= \exp\Big( - \frac{t}{d} \sum_{i=1}^d \cos^2(\pi k\xi_i) \Big).
\]
Denote $\tilde{\Delta}=\frac{1}{2}[d\Id+T].$ Then the multiplier symbol of $\tilde{\Delta}$ is $\sum_{i=1}^d \cos^2(\pi k\xi_i).$ Consequently, taking $s=t/d$ we have
\[
P_{k,t,1}=\exp(-s\tilde{\Delta})=\exp(-sd/2)\exp(-sT/2)=\exp(-t/2)\sum_{r=0}^{\infty}\frac{(-t)^r T^r}{(2d)^r r!}.
\]
Since $|T[f](x)|\le T[|f|](x)$ using \eqref{eq:esDel} we obtain
\[
|P_{k,t,1}f(x)|=\le \exp(-t/2)\sum_{r=0}^{\infty}\frac{t^r (T^r)[|f|](x)}{(2d)^r r!}=P_{k,t,0}(|f|)(x).
\]
This proves \eqref{eq:Pkn1<Pkn0} and completes the proof of the theorem.
\end{proof}

Now we are ready to explain the relation between $p_n(\xi)$ and $\beta_{\overline{n}}(\xi)$. In the lemma below we use the notation $\overline{n}(U),$ see (24) in Section \ref{ssec:not}. Recall that $\widetilde{p_{n_j}}$ is defined by \eqref{eq: pttilxi}.
\begin{Lem} \label{lem:2.12}
Let $K \in \N$. Take any $d \in \N$, $d \geq 2K$, $\xi \in \T^d$, $\overline{n}=(n_1,...,n_K) \in \N_0^K$, such that $n_j \leq \frac{d}{2K}$ for all $j \in [K]$. Fix $\xi\in \T^d$ and let
\[
A= A(\xi) =\{ j \in [K]: \|j \xi+ 1/2 \| \leq \| j \xi \| \}.
\]
Define $\epsilon \in \{-1,1\}^K$, by $\epsilon_j=(-1)^{n_j}$ if $j \in A$ and $\epsilon_j=1$ if $j \not\in A$. Then for any $V \subseteq [K]$ we have
\begin{align*}
&\Big|\sum_{U \subseteq V} \prod_{j \in V \setminus U} (-\epsilon_j \widetilde{p_{n_j}}(j \xi) ) \cdot \beta_{\overline{n}(U)}(\xi) \Big| \\
&\lesssim_K \prod_{j \in V} \min \Big( \frac{n_j}{d} \|j\xi +\frac{\mathds{1}_A(j)}{2} \|^2, \big( \frac{n_j}{d} \|j\xi +\frac{\mathds{1}_A(j)}{2} \|^2\big)^{-1} \Big),
\end{align*}
where the implied constant in the inequality above may depend on $K$, however it is uniform with respect to $d \in \N$ and $\xi \in \T^d$. Furthermore, if $n_j=0$ for some $j \in V$, then the left-hand side equals $0$.
\end{Lem}

\begin{proof}
The implied constant in our proof might depend on the set $V$, however we can just take maximum of such constants over all $V \subseteq [K]$. Fix $V \subseteq [K]$. We essentially have to prove $2^{|V|}$ different inequalities, one for each realization of the minimum among $\frac{n_j}{d} \|j\xi +\mathds{1}_A(j)/2\|^2 $ and its inverse. 

Take any set $E \subseteq V$, we will prove that
\begin{align*}
&\Big|\sum_{U \subseteq V} \prod_{j \in V \setminus U} (-\epsilon_j \widetilde{p_{n_j}}(j \xi) ) \cdot \beta_{\overline{n}(U)}(\xi) \Big| 
\\&\lesssim_K \prod_{j \in E}  \frac{n_j}{d} \|j\xi +\frac{\mathds{1}_A(j)}{2} \|^2 \prod_{j \in V \setminus E} \big( \frac{n_j}{d} \|j\xi +\frac{\mathds{1}_A(j)}{2} \|^2\big)^{-1},
\end{align*}
from this inequality the lemma will follow. Notice that for any $U \subseteq V$ we have
\begin{align*}
&\prod_{j \in V \setminus U} (-\epsilon_j \widetilde{p_{n_j}}(j \xi) ) \cdot \beta_{\overline{n}(U)}(\xi)\\
&= \se{\substack{I_1,I_2,...,I_K \subseteq [d], \\
    (\forall j \in [K]) |I_j|=n_j, \\
    (\forall i,j \in [K], i \neq j) I_i \cap I_j=\emptyset
    }} \prod_{k \in V \setminus U} \big(-\epsilon_k \widetilde{p_{n_k}}(k \xi) \big) \prod_{k \in U} \Big( \prod_{i \in I_k} \cos(2k \pi \xi_i) \Big),
\end{align*}
hence
\begin{equation*} 
\begin{split}
    &\sum_{U \subseteq V} \prod_{j \in V \setminus U} (-\epsilon_j \widetilde{p_{n_j}}(j \xi) ) \cdot \beta_{\overline{n}(U)}(\xi)\\
    &=\se{\substack{I_1,I_2,...,I_K \subseteq [d], \\
    (\forall j \in [K]) |I_j|=n_j, \\
    (\forall i,j \in [K], i \neq j) I_i \cap I_j=\emptyset
    }} \prod_{k \in V} \Big( \prod_{i \in I_k} \cos(2k \pi \xi_i)  -\epsilon_k \widetilde{p_{n_k}}(k\xi)  \Big)
\\
&=\se{\substack{I_j \subseteq [d]; j \in V, \\
    (\forall j \in V) |I_j|=n_j, \\
    (\forall i,j \in V, i \neq j) I_i \cap I_j=\emptyset
    }} \prod_{k \in V} \Big( \prod_{i \in I_k} \cos(2k \pi \xi_i)  -\epsilon_k \widetilde{p_{n_k}}(k\xi)  \Big)
    \end{split}
\end{equation*}
so that
\begin{equation}
\label{eq:2.7}
\begin{split}
&\sum_{U \subseteq V} \prod_{j \in V \setminus U} (-\epsilon_j \widetilde{p_{n_j}}(j \xi) ) \cdot \beta_{\overline{n}(U)}(\xi)\\
&=\se{\substack{I_j \subseteq [d]; j \in E, \\
    (\forall j \in E) |I_j|=n_j, \\
    (\forall i,j \in E, i \neq j) I_i \cap I_j=\emptyset
    }} \prod_{j \in E} \Big( \prod_{i \in I_j} \cos(2j \pi \xi_i)  -\epsilon_j\widetilde{p_{n_j}}(j\xi)  \Big) 
    \\
&\hspace{2.8cm}\cdot\se{\substack{I_k \subseteq [d] \setminus \cup_{j \in E} I_j ; k \in V \setminus E, \\
    (\forall k \in V \setminus E) |I_k|=n_k, \\
    (\forall j,k \in V \setminus E, j \neq k) I_j \cap I_k=\emptyset
    }} \prod_{k \in V \setminus E} \Big( \prod_{i \in I_k} \cos(2k \pi \xi_i)  -\epsilon_k \widetilde{p_{n_k}}(k\xi)  \Big).
    \end{split}
\end{equation}
If $n_k=0$ for some $k \in V$ then $I_k=\emptyset,$ $\epsilon_k=1$ and $\widetilde{p_{n_k}}(k\xi)=1,$ thus we would have
\[
\prod_{i \in I_k} \cos(2k \pi \xi_i)  -\epsilon_k \widetilde{p_{n_k}}(k\xi)=1-1=0,
\]
so the entire sum $\sum_{U \subseteq V} \prod_{j \in V \setminus U} (-\epsilon_j \widetilde{p_{n_j}}(j \xi) ) \cdot \beta_{\overline{n}(U)}(\xi)$ would vanish. From now on we assume that $n_j \geq 1$ for all $j \in V$. 
\par
Fix disjoint sets $(I_j: j \in E)$, let $J= \cup_{j \in E} I_j$, $m=\sum_{j \in E} n_j$. For now we will look only at the expression in the last line of \eqref{eq:2.7}, that is
\[
(*):=\se{\substack{I_k \subseteq [d] \setminus J ; k \in V \setminus E, \\
    (\forall k \in V\setminus E) |I_k|=n_k, \\
    (\forall j,k \in V \setminus E, j \neq k) I_j \cap I_k=\emptyset
    }} \prod_{k \in V \setminus E} \Big( \prod_{i \in I_k} \cos(2k \pi \xi_i)  -\epsilon_k \widetilde{p_{n_k}}(k\xi)  \Big).
\]
Notice that 
\begin{align*}
(*)&=\sum_{U \subseteq V \setminus E } \prod_{j \in V \setminus E \setminus U}  \big( -\epsilon_j \widetilde{p_{n_j}}(j\xi) \big) \cdot \hspace{-0.2cm}\se{\substack{I_k \subseteq [d] \setminus J ; k \in V \setminus E, \\
    (\forall k \in V \setminus E) |I_k|=n_k, \\
    (\forall j,k \in V \setminus E, j \neq k) I_j \cap I_k=\emptyset
    }}  \prod_{k \in U} \prod_{i \in I_k} \cos(2k \pi \xi_i),  
    \\
    &=\sum_{U \subseteq V \setminus E } \prod_{j \in V \setminus E \setminus U}  \big( -\epsilon_j \widetilde{p_{n_j}}(j\xi) \big) \se{\substack{I_k \subseteq [d] \setminus J ; k \in U, \\
    (\forall k \in U) |I_k|=n_k, \\
    (\forall j,k \in U, j \neq k) I_j \cap I_k=\emptyset
    }}  \prod_{k \in U} \prod_{i \in I_k} \cos(2k \pi \xi_i).
\end{align*}
By Lemma \ref{lem:2.8} with $[d]\setminus J$ in place of $J$ we obtain
\begin{align*}
&\Big| \hspace{-0.3cm}\se{\substack{I_k \subseteq [d] \setminus J ; k \in U, \\
    (\forall k \in U) |I_k|=n_k, \\
    (\forall j,k \in U, j \neq k) I_j \cap I_k=\emptyset
    }} \hspace{-0.3cm}  \prod_{k \in U} \prod_{i \in I_k} \cos(2k \pi \xi_i) \Big|
    = \Big| \hspace{-0.3cm}\se{\substack{I_k \subseteq [d] \setminus J ; k \in [K], \\
    (\forall k \in U) |I_k|=n_k, \\
    (\forall k \in [K]\setminus U) |I_k|=0, \\
    (\forall j,k \in [K], j \neq k) I_j \cap I_k=\emptyset
    }} \hspace{-0.3cm} \prod_{k \in [K]} \prod_{i \in I_k} \cos(2k \pi \xi_i) \Big|
\\
&\leq 6 \prod_{k \in U} \exp\Big(- \frac{cn_k}{80K(d-m)} \min\Big( \sum_{i \in [d] \setminus J } \sin^2(k \pi \xi_i), \sum_{i \in [d] \setminus J} \cos^2(k \pi \xi_i) \Big) \Big)
\\
&\leq 6 \prod_{k \in U} \exp\Big(- \frac{cn_k}{80Kd} \min\Big( \sum_{i \in [d] \setminus J } \sin^2(k \pi \xi_i), \sum_{i \in [d] \setminus J} \cos^2(k \pi \xi_i) \Big) \Big).
\end{align*}
Using the above with inequality $\widetilde{p_{n_j}}(j\xi) \leq  \big( \frac{n_j}{d} \|j\xi +\frac{\mathds{1}_A(j)}{2} \|^2\big)^{-1}$ we get
\begin{align*}
|(*)| &\lesssim \sum_{U \subseteq V \setminus E} \prod_{j \in V \setminus E \setminus U} \big( \frac{n_j}{d} \|j\xi +\frac{\mathds{1}_A(j)}{2} \|^2\big)^{-1}\\
&\hspace{0.8cm}\cdot \prod_{k \in U} \exp\Big(- \frac{cn_k}{80Kd} \min\Big( \sum_{i \in [d] \setminus J } \sin^2(k \pi \xi_i), \sum_{i \in [d] \setminus J} \cos^2(k \pi \xi_i) \Big) \Big).
\end{align*}
Using inequality above and equation \eqref{eq:2.7} we obtain
\begin{equation} \label{eq:2.8}
\begin{split}
&\Big|\sum_{U \subseteq V} \prod_{j \in V \setminus U} (-\epsilon_j \widetilde{p_{n_j}}(j \xi) ) \cdot \beta_{\overline{n}(U)}(\xi)\Big| \lesssim
\se{\substack{I_j \subseteq [d]; j \in E, \\
    (\forall j \in E) |I_j|=n_j, \\
    (\forall i,j \in E, i \neq j) I_i \cap I_j=\emptyset
    }} \\
& \prod_{j \in E}\Big| \prod_{i \in I_j} \cos(2j \pi \xi_i)  -\epsilon_j\widetilde{p_{n_j}}(j\xi)  \Big|\cdot \sum_{U \subseteq V \setminus E} \prod_{j \in V \setminus E \setminus U} \big( \frac{n_j}{d} \|j\xi +\frac{\mathds{1}_A(j)}{2} \|^2\big)^{-1} 
    \\
   &\cdot \prod_{k \in U} \exp\Big(- \frac{cn_k}{80Kd} \min\Big( \sum_{i \in [d] \setminus \cup_{j \in E} I_j } \sin^2(k \pi \xi_i), \sum_{i \in [d] \setminus \cup_{j \in E} I_j} \cos^2(k \pi \xi_i) \Big) \Big).
   \end{split}
\end{equation}
Recall that for any sequences of complex numbers $(a_j: j \in [d])$, $(b_j: j \in [d])$, which are bounded by $1$ we have
\[
\Big| \prod_{j=1}^d a_j - \prod_{j=1}^db_j \Big| \leq \sum_{j=1}^d |a_j-b_j|.
\]
Thus for any $j \in E \setminus A$ and any $I_j$ we have
\begin{align*}
&\Big| \prod_{i \in I_j} \cos(2j \pi \xi_i)  -\epsilon_j\widetilde{p_{n_j}}(j\xi)  \Big| \leq \Big|\prod_{ i \in I_j} \cos(2j \pi \xi_i)-1\Big|+|1-e^{- \frac{n_j}{d} \| j \xi \|^2} |
\\
&\leq  \sum_{i \in I_j} |\cos(2j \pi \xi_j)-1|+|1-e^{- \frac{n_j}{d}  \| j \xi \|^2} | \leq 2 \sum_{ i \in I_j} \sin^2(j \pi \xi_i)+ \frac{n_j}{d} \| j \xi \|^2.
\end{align*}
On the other hand if $j \in A \cap E$, then
\begin{align*}
&\Big| \prod_{i \in I_j} \cos(2j \pi \xi_i)  -\epsilon_j\widetilde{p_{n_j}}(j\xi)  \Big|\\
&=\Big| (-1)^{n_j}\prod_{i \in I_j} \cos\big(2 \pi (j\xi_i+1/2) \big)  -(-1)^{n_j}e^{- \frac{n_j}{d} \| j \xi +1/2 \|}  \Big|\\
&\leq  \sum_{i \in I_j} |\cos\big(2 \pi(j \xi_i +\frac{1}{2})\big)-1|+|1-e^{- \frac{n_j}{d}  \| j \xi+1/2 \|^2} |\\
&\leq 2 \sum_{ i \in I_j} \sin^2\big(\pi (j\xi_i+\frac{1}{2}) \big)+ \frac{n_j}{d} \| j \xi+\frac{1}{2} \|^2.
\end{align*}
In both cases we get
\[
\Big| \prod_{i \in I_j} \cos(2j \pi \xi_i)  -\epsilon_j\widetilde{p_{n_j}}(j\xi)  \Big|
\leq 2 \sum_{ i \in I_j} \sin^2\big(\pi (j\xi_i+\frac{\mathds{1}_A(j)}{2}) \big)+ \frac{n_j}{d} \| j \xi+\frac{\mathds{1}_A(j)}{2} \|^2
\]
and using this in \eqref{eq:2.8} we obtain
\begin{multline} 
\label{eq:2.9}
\Big|\sum_{U \subseteq V} \prod_{j \in V \setminus U} (-\epsilon_j \widetilde{p_{n_j}}(j \xi) ) \cdot \beta_{\overline{n}(U)}(\xi)\Big| \\
 \lesssim_K
\se{\substack{I_j \subseteq [d]; j \in E, \\
    (\forall j \in E) |I_j|=n_j, \\
    (\forall i,j \in E, i \neq j) I_i \cap I_j=\emptyset
    }} \prod_{j \in E} \Big( \sum_{ i \in I_j} \sin^2\big(\pi (j\xi_i+\frac{\mathds{1}_A(j)}{2}) \big)+ \frac{n_j}{d} \| j \xi+\frac{\mathds{1}_A(j)}{2} \|^2  \Big) \sum_{U \subseteq V \setminus E} 
    \\
  \prod_{j \in V \setminus E \setminus U} \big( \frac{n_j}{d} \|j\xi +\frac{\mathds{1}_A(j)}{2} \|^2\big)^{-1}  \prod_{k \in U} \exp\Big(- \frac{cn_k}{80Kd} \min\Big( \sum_{i \in [d] \setminus \cup_{j \in E} I_j } \sin^2(k \pi \xi_i), \sum_{i \in [d] \setminus \cup_{j \in E} I_j} \cos^2(k \pi \xi_i) \Big) \Big) \\
   = \sum_{U \subseteq V \setminus E} \sum_{F \subseteq E}  \prod_{j \in V \setminus E \setminus U} \big( \frac{n_j}{d} \|j\xi +\frac{\mathds{1}_A(j)}{2} \|^2\big)^{-1} \prod_{k \in E \setminus F} \big( \frac{n_k}{d} \| k \xi+\frac{\mathds{1}_A(k)}{2} \|^2 \big) 
   \\
 \cdot \se{\substack{I_j \subseteq [d]; j \in E, \\
    (\forall j \in E) |I_j|=n_j, \\
    (\forall i,j \in E, i \neq j) I_i \cap I_j=\emptyset
    }} \prod_{j \in F} \Big( \sum_{ i \in I_j} \sin^2\big(\pi (j\xi_i+\frac{\mathds{1}_A(j)}{2}) \big)  \Big) 
    \\
    \cdot 
    \prod_{k \in U} \exp\Big(- \frac{cn_k}{80Kd} \min\Big( \sum_{i \in [d] \setminus \cup_{j \in E} I_j } \sin^2(k \pi\xi_i), \sum_{i \in [d] \setminus \cup_{j \in E} I_j} \cos^2(k \pi \xi_i) \Big) \Big).
    \end{multline}

    Fix $U \subseteq V \setminus E$ and $F \subseteq E$, we will investigate the expression in the last two lines, call it $(**)$. This expression can be rewritten as
   \begin{multline} \label{eq:2.10}
   (**)= \frac{\prod_{j \in E}n_j!(d-\sum_{j \in E} n_j)!}{d!} \sum_{\substack{x_i \in [d]; i \in F, \\
    (\forall i,j \in F, i \neq j) x_i \neq x_j}} \prod_{j \in F} \sin^2\big(\pi(j \xi_{x_j}+\frac{\mathds{1}_A(j)}{2}) \big) 
    \\
    \cdot \sum_{\substack{J \subseteq [d], \\  \{x_j: j \in F\} \subseteq J, \\ |J|=\sum_{i \in E}n_i} } \prod_{k \in U} \exp\Big(- \frac{cn_k}{80Kd} \min\Big( \sum_{i \in [d] \setminus J } \sin^2(k \pi \xi_i), \sum_{i \in [d] \setminus J} \cos^2(k \pi \xi_i) \Big) \Big) 
    \\
    \cdot \#\{ (I_j: j \in E) : |I_j|=n_j, \cup_{j \in E} I_j=J, (\forall i,j \in E, i \neq j) I_i \cap I_j= \emptyset, (\forall j \in F)x_j \in I_j\}.
   \end{multline}  
   Notice that for each $J \subseteq [d]$, satisfying $|J|= \sum_{i \in E}n_i$ the quantity in the last line of \eqref{eq:2.10} equals
  \[
   \frac{\big(\sum_{i \in E}n_i-|F|\big)!}{\prod_{i \in E \setminus F}n_i! \cdot \prod_{i \in F}(n_i-1)!}.
  \]
   Plugging this into \eqref{eq:2.10} and replacing $J$ with $[d] \setminus J$ we obtain
   \begin{equation} \label{eq:2.11}
   \begin{split}
   (**)=&\frac{\prod_{j \in F}n_j \cdot (d-\sum_{j \in E} n_j)!(\sum_{i \in E} n_i-|F|)!}{d!} \\
   &\cdot \sum_{\substack{x_i \in [d]; i \in F, \\
    (\forall i,j \in F, i \neq j) x_i \neq x_j}} \prod_{j \in F} \sin^2\big(\pi(j \xi_{x_j}+\frac{\mathds{1}_A(j)}{2}) \big) 
    \\
   &\cdot \sum_{\substack{J \subseteq [d] \setminus  \{x_j: j \in F\}, \\ |J|=d-\sum_{i \in E}n_i} } \prod_{k \in U} \exp\Big(- \frac{cn_k}{80Kd} \min\Big( \sum_{i \in J } \sin^2(k \pi \xi_i), \sum_{i \in J} \cos^2(k \pi \xi_i) \Big) \Big).
   \end{split}
    \end{equation}
    Looking at the  expression in the last line above we notice that
   \begin{align*}
    &\sum_{\substack{J \subseteq [d] \setminus  \{x_j: j \in F\}, \\ |J|=d-\sum_{i \in E}n_i} } \prod_{k \in U} \exp\Big(- \frac{cn_k}{80Kd} \min\Big( \sum_{i \in J } \sin^2(k \pi \xi_i), \sum_{i \in J} \cos^2(k \pi \xi_i) \Big) \Big)
    \\
    &\leq \sum_{\substack{J \subseteq [d] \setminus  \{x_j: j \in F\}, \\ |J|=d-\sum_{i \in E}n_i} } \prod_{k \in U}\Big( \exp\Big(- \frac{cn_k}{80Kd} \sum_{i \in J } \sin^2(k \pi \xi_i) \Big)+ \exp\Big(  - \frac{cn_k}{80Kd}\sum_{i \in J} \cos^2(k \pi \xi_i) \Big) \Big) 
   \\
    &= \binom{d-|F|}{d-\sum_{i \in E} n_i}\sum_{C \subseteq U}\se{\substack{J \subseteq [d] \setminus  \{x_j: j \in F\}, \\ |J|=d-\sum_{i \in E}n_i} } \exp\Big( - \frac{cn_k}{80Kd} \sum_{i \in J}\Big( \sum_{k \in C} \sin^2(k \pi \xi_i) +\sum_{k \in U \setminus C} \cos^2(k \pi \xi_i) \Big) \Big).
    \end{align*}
    Using Corollary \ref{cor:2.7} to each inner expectation we see that the above sum in $C$ does not exceed $3$ times
   \begin{align*}
& \sum_{C \subseteq U}  \exp\Big( - \frac{c'n_k(d-\sum_{i \in E}n_i)}{Kd(d-|F|)} \sum_{i \in [d] \setminus \{x_j : j \in F\}} \Big( \sum_{k \in C} \sin^2(k \pi \xi_i) +\sum_{k \in U \setminus C} \cos^2(k \pi \xi_i) \Big) \Big)
   \\
    &\leq \sum_{C \subseteq U}  \exp\Big( - \frac{c''n_k}{Kd} \sum_{i \in [d] \setminus \{x_j : j \in F\}} \Big( \sum_{k \in C} \sin^2(k \pi \xi_i) +\sum_{k \in U \setminus C} \cos^2(k \pi \xi_i) \Big) \Big)
    \\
    &= \prod_{k \in U} \Big(\exp\Big( -\frac{c''n_k}{Kd} \sum_{i \in [d] \setminus \{x_j : j \in F\}}  \sin^2(k \pi \xi_i) \Big)+ \exp\Big( - \frac{c''n_k}{Kd} \sum_{i \in [d] \setminus \{x_j : j \in F\}}  \cos^2(k \pi \xi_i) \Big) \Big)
    \\
    &\lesssim_K  \prod_{k \in U} \exp\Big( - \frac{c''n_k}{Kd} \min\big( \| k \xi \|^2, \| k \xi+\frac{1}{2} \|^2 \big)\Big) \lesssim_K  \prod_{k \in U} \Big( \frac{n_k}{d}  \| k \xi + \frac{\mathds{1}_A(k)}{2} \|^2\Big)^{-1}.
    \end{align*}
Coming back to \eqref{eq:2.11} we get
    \begin{multline} \label{eq:2.12}
   \se{\substack{I_j \subseteq [d]; j \in E, \\
    (\forall j \in E) |I_j|=n_j, \\
    (\forall i,j \in E, i \neq j) I_i \cap I_j=\emptyset
    }} \prod_{j \in F} \Big( \sum_{ i \in I_j} \sin^2\big(\pi (j\xi_i+\frac{\mathds{1}_A(j)}{2}) \big)  \Big) 
    \\
    \cdot \prod_{k \in U} \exp\Big(- \frac{cn_k}{80Kd} \min\Big( \sum_{i \in [d] \setminus \cup_{j \in E} I_j } \sin^2(k \pi \xi_i), \sum_{i \in [d] \setminus \cup_{j \in E} I_j} \cos^2(k \pi \xi_i) \Big) \Big)=(**)
    \\
    \lesssim_K
    \frac{\prod_{j \in F}n_j \cdot (d-\sum_{j \in E} n_j)!(\sum_{i \in E} n_i-|F|)!}{d!} \cdot  \binom{d-|F|}{d-\sum_{i \in E}n_i}
    \\
    \sum_{\substack{x_i \in [d]; i \in F, \\
    (\forall i,j \in F, i \neq j) x_i \neq x_j}} \prod_{j \in F} \sin^2\big(\pi(j \xi_{x_j}+\frac{\mathds{1}_A(j)}{2}) \big) \prod_{k \in U} \Big( \frac{n_k}{d}  \| k \xi + \frac{\mathds{1}_A(k)}{2} \|^2\Big)^{-1}
    \\
    \lesssim_K   
    \frac{\prod_{j \in F}n_j \cdot (d-|F|)!}{d!} 
    \sum_{\substack{x_i \in [d]; i \in F}} \prod_{j \in F} \sin^2\big(\pi(j \xi_{x_j}+\frac{\mathds{1}_A(j)}{2}) \big) \prod_{k \in U} \Big( \frac{n_k}{d}  \| k \xi + \frac{\mathds{1}_A(k)}{2} \|^2\Big)^{-1} 
    \\
    \lesssim_K \prod_{j \in F} \big( \frac{n_j}{d} \| j \xi+ \frac{\mathds{1}_A(j)}{2} \|^2 \big) \prod_{k \in U} \Big( \frac{n_k}{d}  \| k \xi + \frac{\mathds{1}_A(k)}{2} \|^2\Big)^{-1}.
    \end{multline}
    
    Finally combining \eqref{eq:2.12} and \eqref{eq:2.9},  we reach
  \begin{align*}
    &\Big|\sum_{U \subseteq V} \prod_{j \in V \setminus U} (-\epsilon_j \widetilde{p_{n_j}}(j \xi) ) \cdot \beta_{\overline{n}(U)}(\xi)\Big| \lesssim_K \sum_{U \subseteq V \setminus E} \sum_{F \subseteq E} \prod_{j \in V \setminus E \setminus U} \big( \frac{n_j}{d} \|j\xi +\frac{\mathds{1}_A(j)}{2} \|^2\big)^{-1} 
   \\
    & \prod_{k \in E \setminus F} \big( \frac{n_j}{d} \| j \xi+\frac{\mathds{1}_A(j)}{2} \|^2 \big) \cdot \prod_{j \in F} \big( \frac{n_j}{d} \| j \xi+ \frac{\mathds{1}_A(j)}{2} \|^2 \big) \prod_{k \in U} \Big( \frac{n_k}{d}  \| k \xi + \frac{\mathds{1}_A(k)}{2} \|^2\Big)^{-1}, \end{align*}
  which leads to
    \[
  \Big|\sum_{U \subseteq B} \prod_{j \in V \setminus U} (-\epsilon_j \widetilde{p_{n_j}}(j \xi) ) \cdot \beta_{\overline{n}(U)}(\xi)\Big|
    \lesssim_K \prod_{j \in E} \big( \frac{n_j}{d} \| j \xi+ \frac{\mathds{1}_A(j)}{2} \|^2 \big) \prod_{k \in V \setminus E} \Big( \frac{n_k}{d}  \| k \xi + \frac{\mathds{1}_A(k)}{2} \|^2\Big)^{-1}.
    \]
    Since $E \subseteq V$ was chosen arbitrarily, this concludes the proof of Lemma \ref{lem:2.12}.
\end{proof}
We will also need the following simple lemma, which proof we skip.
\begin{Lem} \label{lem:2.13}
    For every $\varepsilon \in \{0,1\}$ and any $x \geq 0$ we have
    \[
    \sum_{n + \varepsilon \in \D} \min \big(n x, (n x)^{-1} \big)^2 \leq 10.
    \]
\end{Lem}
Finally we are able to prove Proposition \ref{prop:2.9}.
\begin{proof}
Fix $K \in \N$ and assume that $d \geq 2K$ (the other case is trivial). We will prove inductively over $|V|$ the following.\\
{\bf Statement:}\textit{ For any $ V \subseteq [K]$ there exists $C_{V,K}$ such that for all $\eta \in \{0,1\}^V$,  we have
\[
 \Big\| \sup_{\substack{
 n_j+\eta(j) \in \D; j \in V, \\
 (\forall j \in V) n_j \leq \frac{d}{2K}}} |\mathcal{D}_{\overline{n}(V)}f| \Big\|_{2} \leq C_{V,K} \big\| f \big\|_{2},\qquad f \in \ell^2(\Z^d).
\]
}\\
Note that we can take $C_K=\sup_{V \subseteq [K]} C_{V,K}$ and have a constant independent of $V$ in the statement above. Note that for any $U \subseteq V \subseteq [K]$ we have $\overline{n}(V)(U)=\overline{n}(U)$, where the notation $\overline{n}(V)(U)$ was introduced in item 24.\ on p.\ \pageref{p: nota}. \\
\textit{Inductive proof of the statement above.} \\
\textbf{1)} If $|V|=0$, then $\overline{n}(V)= (0,0,...,0)$, so $\mathcal{D}_{\overline{n}(V)}$ is the identity operator and the claim is obvious.
\\
\textbf{2) Inductive step.} Take $V \subseteq [K]$, $|V|=m$ and assume that the statement is true for all sets of size smaller than $m$. Take any $f \in \ell^2(\Z^d)$, we can decompose $f$ in the following way
\[
f= \sum_{A \subseteq [K]} f_A,
\]
where every $f_A$ satisfies
\[
\supp(\widehat{f_A}) \hspace{-0.1cm} \subseteq \hspace{-0.1cm}\Big\{\xi \in \T^d \hspace{-0.1cm}:  \hspace{-0.1cm}(\forall j \in A) \| j \xi+ \frac{1}{2} \| \hspace{-0.1cm}\leq \hspace{-0.1cm}\|j \xi \|,   (\forall j \in [K] \setminus A) \|j \xi \| \hspace{-0.1cm}< \hspace{-0.1cm}\| j \xi+ \frac{1}{2} \| \Big\}.
\]
By triangle inequality we have
\begin{equation*} 
 \Big\| \sup_{\substack{
 n_j+\eta(j) \in \D; j \in V, \\
 (\forall j \in V) n_j \leq \frac{d}{2K}}} |\mathcal{D}_{\overline{n}(V)}f| \Big\|_{2} \leq \sum_{A \subseteq [K] } \Big\| \sup_{\substack{
 n_j+\eta(j) \in \D; j \in V, \\
 (\forall j \in V) n_j \leq \frac{d}{2K}}} |\mathcal{D}_{\overline{n}(V)}f_A| \Big\|_{2}.
\end{equation*}
We will bound each individual summand for $A \subseteq [K].$ Using the decomposition
\[
\beta_{\overline{n}(V)}=\sum_{U\in V}  \prod_{j \in V \setminus U} (-\epsilon_j \widetilde{p_{n_j}}(j \xi))\cdot \beta_{\overline{n}(U)}-\sum_{\substack{U \subseteq V, \\ U \neq V}}\prod_{j \in V \setminus U} (-\epsilon_j \widetilde{p_{n_j}}(j \xi))\cdot \beta_{\overline{n}(U)}
\]
we obtain
\begin{align*} 
    &\Big\| \sup_{\substack{
 n_j+\eta(j) \in \D; j \in V, \\
 (\forall j \in V) n_j \leq \frac{d}{2K}}} |\mathcal{D}_{\overline{n}(V)}f_A| \Big\|_{2}=\Big\| \sup_{\substack{
 n_j+\eta(j) \in \D; j \in V, \\
 (\forall j \in V) n_j \leq \frac{d}{2K}}} \big|\mathcal{F}^{-1}\Big( \beta_{\overline{n}(V)}\widehat{f_A} \Big)\big| \Big\|_{2}
 \\
 &\leq 
 \Big\| \sup_{\substack{
 n_j+\eta(j) \in \D; j \in V, \\
 (\forall j \in V) n_j \leq \frac{d}{2K}}} \big|\mathcal{F}^{-1}\Big(\sum_{U \subseteq V} \prod_{j \in V \setminus U} (-\epsilon_j \widetilde{p_{n_j}}(j \xi) ) \cdot \beta_{\overline{n}(U)}(\xi)\widehat{f_A} \Big)\big|  \Big\|_{2}
 \\
 &+\sum_{\substack{U \subseteq V, \\ U \neq V}} \Big\| \sup_{\substack{
 n_j+\eta(j) \in \D; j \in V, \\
 (\forall j \in V) n_j \leq \frac{d}{2K}}} \big|\mathcal{F}^{-1}\Big( \prod_{j \in V \setminus U} (-\epsilon_j \widetilde{p_{n_j}}(j \xi) ) \cdot \beta_{\overline{n}(U)}(\xi)\widehat{f_A} \Big)\big|  \Big\|_{2} \hspace{-0.2cm}=: \hspace{-0.1cm}W_1 \hspace{-0.1cm}+ \hspace{-0.1cm}W_2,
\end{align*}
where $\epsilon\in \{-1,1\}^K$ is defined as in Lemma \ref{lem:2.12}, i.e.\ $\epsilon_j=(-1)^{n_j}$ if $j \in A$ and $\epsilon_j=1$ if $j \not\in A$.

Estimating $W_1$ we use a standard square function argument and Parseval's theorem to get
\begin{align*} 
W_1 &\leq \Big\| \Big( \sum_{\substack{n_j+\eta(j) \in \D;j \in V, \\ (\forall j \in V) n_j \leq \frac{d}{2K}}} \Big| \mathcal{F}^{-1}\Big(\sum_{U \subseteq V} \prod_{j \in V \setminus U} (-\epsilon_j \widetilde{p_{n_j}}(j \xi) ) \cdot \beta_{\overline{n}(U)}(\xi)\widehat{f_A} \Big) \Big|^2  \Big)^{1/2} \Big\|_{2}
\\
&= 
 \Big( \sum_{\substack{n_j+\eta(j) \in \D;j \in V, \\ (\forall j \in V) n_j \leq \frac{d}{2K}}} \Big\| \mathcal{F}^{-1}\Big( \sum_{U \subseteq V}\prod_{j \in V \setminus U} (-\epsilon_j \widetilde{p_{n_j}}(j \xi) ) \cdot \beta_{\overline{n}(U)}(\xi)\widehat{f_A} \Big) \Big\|_{2}^2  \Big)^{1/2} 
 \\
 &=
 \Big( \sum_{\substack{n_j+\eta(j) \in \D;j \in V, \\ (\forall j \in V) n_j \leq \frac{d}{2K}}} \int_{\T^d}\Big| \sum_{U \subseteq V}\prod_{j \in V \setminus U} (-\epsilon_j \widetilde{p_{n_j}}(j \xi) ) \cdot \beta_{\overline{n}(U)}(\xi)\widehat{f_A}(\xi) \Big|^2  d \xi \Big)^{1/2}.
\end{align*}
Thus, using Lemma \ref{lem:2.12}, the assumption on $\supp(\widehat{f_A}),$ and Lemma \ref{lem:2.13} we obtain
 \begin{equation*}
 \begin{split}
 W_1&\lesssim_K 
 \Big(  \int_{\T^d} \sum_{\substack{n_j+\eta(j) \in \D;j \in V, \\ (\forall j \in V) n_j \leq \frac{d}{2K}}} \prod_{j \in V} \min \Big( \frac{n_j}{d} \|j\xi +\frac{\mathds{1}_A(j)}{2} \|^2, \big( \frac{n_j}{d} \|j\xi +\frac{\mathds{1}_A(j)}{2} \|^2\big)^{-1} \Big)^2|\widehat{f_A}(\xi)|^2  d \xi \Big)^{1/2}
 \\
 &\lesssim_K 
 \Big(  \int_{\T^d}|\widehat{f_A}(\xi)|^2  d \xi \Big)^{1/2} = \|\widehat{f_A} \|_{L^2(\T^d)}= \| f_A \|_{2}.
 \end{split}
\end{equation*}
\par
Now we will bound $W_2.$ For any $U \subseteq V$, $U \neq V$ we have
\begin{equation} \label{eq:2.16}
\begin{split}
&\Big\| \sup_{\substack{
 n_j+\eta(j) \in \D; j \in V, \\
 (\forall j \in V) n_j \leq \frac{d}{2K}}} \big|\mathcal{F}^{-1}\Big( \prod_{j \in V \setminus U} (-\epsilon_j \widetilde{p_{n_j}}(j \xi) ) \cdot \beta_{\overline{n}(U)}(\xi)\widehat{f_A} \Big)\big|  \Big\|_{2}
 \\
 &=
 \Big\| \sup_{\substack{
 n_j+\eta(j) \in \D; j \in V, \\
 (\forall j \in V) n_j \leq \frac{d}{2K}}} \big|\mathcal{F}^{-1}\Big( \prod_{j \in V \setminus U} p_{n_j}(j \xi+ \frac{\mathds{1}_A(j)}{2})  \cdot \beta_{\overline{n}(U)}(\xi)\widehat{f_A} \Big)\big|  \Big\|_{2}
 \\
 &=\Big\| \sup_{\substack{
 n_j+\eta(j) \in \D; j \in V, \\
 (\forall j \in V) n_j \leq \frac{d}{2K}}} \big|\Big( \prod_{j \in V \setminus U} P_{j,n_j,\mathds{1}_A(j)} \Big)   \Big(\mathcal{D}_{\overline{n}(U)}f_A \Big)\big|  \Big\|_{2}.
 \end{split}
\end{equation}
Using \eqref{eq:Pkn1<Pkn0} and the positivity preserving property of $P_{j,n_j,0}$ from Theorem \ref{thm:2.11} we see that the last term in \eqref{eq:2.16} is bounded by
\[
\Big\| \sup_{\substack{
 n_j+\eta(j) \in \D; j \in V \setminus U, \\
 (\forall j \in V \setminus U) n_j \leq \frac{d}{2K}}} \Big( \prod_{j \in V \setminus U} P_{j,n_j,0} \Big)   \Big( \sup_{\substack{
 n_j+\eta(j) \in \D; j \in U, \\
 (\forall j \in U) n_j \leq \frac{d}{2K}}} \Big|\mathcal{D}_{\overline{n}(U)}f_A \Big|\Big)  \Big\|_{2}.
\]
By Theorem \ref{thm:2.11} used $|V \setminus U|$ times we have that the above is bounded by
\[
C^{|V \setminus U|} \Big\| \sup_{\substack{
 n_j+\eta(j) \in \D; j \in U, \\
 (\forall j \in U) n_j \leq \frac{d}{2K}}} \Big|\mathcal{D}_{\overline{n}(U)}f_A \Big|  \Big\|_{2}, 
\]
which by our induction hypothesis (note that $|U|<|V|$) is controlled by $C_{U,K} \| f_A \|_{\ell^2{(\Z)}}$. Combining everything, for every $A \subseteq [K]$ we obtain
\[
\Big\| \sup_{\substack{
 n_j+\eta(j) \in \D; j \in V, \\
 (\forall j \in V) n_j \leq \frac{d}{2K}}} |\mathcal{D}_{\overline{n}(V)}f_A| \Big\|_{2} \lesssim_{V,K} \| f_A \|_{2},
\]
hence
\[
\Big\| \sup_{\substack{
 n_j+\eta(j) \in \D; j \in V, \\
 (\forall j \in V) n_j \leq \frac{d}{2K}}} |\mathcal{D}_{\overline{n}(V)}f| \Big\|_{2} \lesssim_{V,K} \sum_{A \subseteq [K]} \| f_A \|_{2} \lesssim_{V,K} \|f \|_{2}.
\]
This concludes proof of inductive step, hence concludes the proof of Proposition \ref{prop:2.9}.
\end{proof}
\subsection{Difference estimates for multipliers}
\label{sec:23}
In this subsection we will prove crucial difference estimates for multipliers $\beta_{\overline{n}}(\xi)$, which will allow us to deduce Theorem \ref{thm:1.5} from Proposition \ref{prop:2.9}. Before that, we have to introduce some definitions and a multi-parameter Rademacher-Menshov type inequality. 
\begin{Defn} \label{def:2.14}
  For $K \in \N$, $k \in [K]$ and a sequence $(a_{\overline{n}}: \overline{n} \in \N_0^K)$ we define the $2$-difference operator in parameter $k$ 
 \[
  (\Delta^k(a))_{n_1,\ldots,n_K}=a_{n_1,...,n_K}-a_{n_1,...,n_{k-1},n_k-2,n_{k+1},...,n_K}.
  \]
  Slightly abusing the notation we will write
  \[
  \Delta^k_{n_1,...,n_K}(a):=(\Delta^k(a))_{n_1,\ldots,n_K}.
  \]
Moreover for $U \subseteq [K]$ we define
  \[
  \Delta^U_{\overline{n}}= \prod_{k \in U} \Delta^k_{\overline{n}},
  \]
  so that for example
  \[
  \Delta_{\overline{n}}^{\{1,2\}}(a)=a_{n_1,n_2,\ldots,n_K}-a_{n_1-2,n_2,\ldots,n_K}-a_{n_1,n_2-2,\ldots,n_K}+a_{n_1-2,n_2-2,\ldots,n_K}.
  \]
\end{Defn}
A simple calculation shows that for any $i,j \in [K]$ we have
\[
\Delta^i_{\overline{n}} \circ \Delta^j_{\overline{n}}=\Delta^j_{\overline{n}} \circ \Delta^i_{\overline{n}},
\]
hence in the definition of $\Delta^U_{\overline{n}}$ it does not matter in which order we compose $\Delta^k_{\overline{n}}$. Note that for a fixed $\overline{k} \in \N^K$ and a sequence $(a_{\overline{n}}: \overline{n} \in \N_0^K)$, if
\[
b_{\overline{n}}= a_{\overline{n}+\overline{k}},
\]
then for any $U \subseteq [K]$ we have
\begin{equation}
\label{eq:DSnk}
\Delta_{\overline{n}}^U(b)=\Delta_{\overline{n}+\overline{k}}^U(a).
\end{equation}
We also introduce the following notation for dyadic intervals
\[
I_{j}^i=\big((j-1)2^{i},j2^{i}\big].
\]
Notice that for $i \geq 1$ we have
\begin{equation}
\label{eq:sumDelkdif}
\sum_{n_k \in I_{j}^i \cap (2 \N_0)} \Delta^k_{n_1,...,n_K}(a) \hspace{-0.05cm}= \hspace{-0.05cm}a_{n_1,...,n_{k-1},j2^i,n_{k+1},...,n_K}-a_{n_1,...,n_{k-1},(j-1)2^i,n_{k+1},...,n_K}.
\end{equation}
We are ready to state a multi-parameter Rademacher-Menshov type inequality. In the two parameter case such an inequality was established in \cite[Lemma 3.1]{KMT1} (see also \cite[Section 3.4]{KLMP} and \cite[Lemma 8.1]{KrMiTa}). Compared to \cite[Lemma 3.1]{KMT1} in the two parameter case we only consider $n_1^0=n_2^0=0$, which suffices for our purposes, while slightly simplifying the proof.
\begin{Lem} \label{lem:2.15}
For any $K \in \N$ and every sequence of complex numbers $a=(a_{\overline{n}}:\overline{n} \in \N_0^K)$ and all $s_1,s_2,...,s_K \in \N_0$ and $\overline{m} \in (\N_0)^K \cap \prod_{j=1}^K [0,2^{s_j}]$ we have
\begin{align*}
&\sup_{\substack{0 \leq n_j \leq m_j; j \in [K], \\ (\forall j \in [K]) \,2\,|\,n_j}} |a_{n_1,...,n_K}|
\\
&\le \sum_{\substack{U \subseteq [K], \\ U \neq  \emptyset}} \sum_{1 \leq i_u \leq s_u; u \in U} \Big(\sum_{1 \leq j_u \leq 2^{s_u-i_u}; u \in U} 
\Big| \hspace{-0.3cm}\sum_{\substack{k_u \in I_{j_u}^{i_u}; u \in U, \\ (\forall u \in U) \,2\,|\,k_u,\\
(\forall u \in U) k_u \leq m_u}} \hspace{-0.3cm} \Delta^U_{\overline{k}(U)}(a) \Big|^2
\Big)^{1/2} \hspace{-0.3cm}+|a_{\overline{0}}|.
\end{align*}
\end{Lem}
\begin{proof}
    We will prove the statement by induction with respect to $K$.
    \par \textbf{1) Case K=1.}  We start with the following simple combinatorial \\ property: \\
    \textbf{Combinatorial property.} Any interval $(0,n ]$ such that $0 < n \leq 2^s $, $n \in 2 \N$ can be decomposed into disjoint union of intervals, where every interval belongs to 
    \[
    \mathcal{I}_i=\{ ((j-1)2^{i},j2^i] : 1 \leq j \leq 2^{s-i} \},
    \]
    for some $1 \leq i \leq s$, moreover each length appears at most once. 
   \par We prove the property above just by looking at the binary decomposition of $n$. Let 
    \[
    n=\sum_{i=1}^{\infty} a_i 2^{i} ,
    \]
     where $a_i\in \{ 0,1\} $ and all but finitely many of $a_i$ are equal to $0$ and, since $n\in 2\N,$ the sum starts from $i=1$ and not $0$.
    Let $U=\{ i \in \N: a_i \neq 0 \}$ and $i_1>i_2>...>i_t$ be its enumeration. Then we define $u_0=0$ and 
    \[
    u_j= \sum_{k \geq i_j} a_k 2^k
    \]
    for $j \in [t]$. We note that the decomposition
    \[
    (0,n]= \bigcup_{j=0}^{t-1} (u_{j},u_{j+1}]
    \]
     satisfies the desired properties.
     \par Let us go back to the proof of the original statement, take $s \in \N_0$ and  $m \in  \N_0, m \leq 2^s$ and  any sequence $(a_n: n \in \N_0)$, then we have
\[
\sup_{\substack{0 \leq n \leq m, \\ 2 \mid n}} |a_n| \leq  |a_{n_0}-a_{0}|+|a_0|,
\]
for some $n_0 \in 2\N$ such that $0<n_0\le m.$ Using the combinatorial property for the interval $(0,n_0]$ we get that
\[
(0,n_0]= \bigcup_{j=0}^{t-1} (u_{j},u_{j+1}],
\]
where each $(u_j,u_{j+1}]$ is dyadic of length $\geq 2$ and each length appears at most once. Then we have
\[
|a_{n_0}-a_{0}| \leq \sum_{j=0}^{t-1} |a_{u_{j+1}}- a_{u_j}|= \sum_{i=1}^s \sum_{\substack{j: (u_j,u_{j+1}] \in \mathcal{I}_i, \\ u_{j+1} \leq m}} |a_{u_{j+1}}- a_{u_j}|.
\]
Observe that the innermost sum above contains at most one term, hence by \eqref{eq:sumDelkdif} we get
\begin{align*}
\sup_{\substack{0 \leq n \leq m, \\ 2 \mid n}} |a_n| &\leq  \sum_{i=1}^s \Big( \sum_{\substack{j: (u_j,u_{j+1}] \in \mathcal{I}_i, \\ u_{j+1} \leq m}} |a_{u_{j+1}}- a_{u_j}|^2 \Big)^{1/2} +|a_0|\\
&\leq   \sum_{i=1}^s \Big( \sum_{\substack{1 \leq j \leq 2^{s-i}, \\ j2^i \leq m}}|a_{j2^i}- a_{(j-1)2^i}|^2 \Big)^{1/2} + |a_0|
\\
&\le \sum_{i=1}^s \Big( \sum_{1 \leq j \leq 2^{s-i}}\Big|\sum_{\substack{k \in I_{j}^i \cap (2\N_0) \cap [0,m]}} \Delta_k(a) \Big|^2 \Big)^{1/2} + |a_0|,
\end{align*}
which proves Lemma \ref{lem:2.15} in the case $K=1$.

\par \textbf{2) Inductive step.} Assume that Lemma \ref{lem:2.15} holds for some $K \geq 1$ and consider $K+1$ in place of $K$. Take any $s_1,...,s_{K+1}$ and $\overline{m}=(m_1,...,m_{K+1}) \in (\N_0)^{K+1} \cap \prod_{j=1}^{K+1} [0,2^{s_j}]$ and a sequence of complex numbers \\ $a=(a_{\overline{n}}:\overline{n} \in \N_0^{K+1})$. Our goal is to prove that 
\begin{align*}
&\sup_{\substack{0 \leq n_j \leq m_j; j \in [K+1], \\ (\forall j \in [K+1]) \,2\,|\,n_j}} |a_{n_1,...,n_{K+1}}| \\
&\le \sum_{\substack{U \subseteq [K+1], \\ U \neq  \emptyset}} \sum_{1 \leq i_u \leq s_u; u \in U} \Big(\sum_{1 \leq j_u \leq 2^{s_u-i_u}; u \in U} 
\Big| \sum_{\substack{k_u \in I_{j_u}^{i_u}; u \in U, \\ (\forall u \in U) \,2\,|\,k_u, \\ (\forall u \in U) k_u \leq m_u }} \hspace{-0.3cm}\Delta^U_{\overline{k}(U)}(a) \Big|^2
\Big)^{1/2}
\hspace{-0.3cm}+|a_{\overline{0}}|,
\end{align*}
 Notice that we have
\begin{multline} \label{eq:2.17}
    \sup_{\substack{0 \leq n_j \leq m_j; j \in [K+1], \\ (\forall j \in [K+1]) \,2\,|\,n_j}} |a_{n_1,...,n_{K+1}}| \leq \sup_{\substack{0 \leq n_j \leq m_j; j \in [K], \\ (\forall j \in [K]) \,2\,|\,n_j}} |a_{n_1,...,n_{K},0}| \\+ \sup_{\substack{0 \leq n_{K+1} \leq m_{K+1} , \\ \,2\,|\,n_{K+1}}} \sup_{\substack{0 \leq n_j \leq m_j; j \in [K], \\ (\forall j \in [K]) \,2\,|\,n_j}} 
 |a_{n_1,...,n_{K},n_{K+1}}- a_{n_1,...,n_{K},0}|
\end{multline}
\par
By induction hypothesis for the first term on the right hand side of \eqref{eq:2.17} we have
\begin{multline}
    \label{eq:2.18}
    \sup_{\substack{0 \leq n_j \leq m_j; j \in [K], \\ (\forall j \in [K]) \,2\,|\,n_j}} |a_{n_1,...,n_K,0}| \\ \le \sum_{\substack{U \subseteq [K], \\ U \neq  \emptyset}} \sum_{1 \leq i_u \leq s_u; u \in U} \Big(\sum_{1 \leq j_u \leq 2^{s_u-i_u}; u \in U} 
\Big| \sum_{\substack{k_u \in I_{j_u}^{i_u}; u \in U, \\ (\forall u \in U) \,2\,|\,k_u, \\ (\forall u \in U) k_u \leq m_u }} \hspace{-0.3cm}\Delta^U_{\overline{k}(U)}(a) \Big|^2
\Big)^{1/2}
\hspace{-0.3cm}+|a_{\overline{0}}|,
\end{multline}
where $\overline{k}(U)$ is the vector of length $K+1$, such that $\overline{k}(U)=k_u$ if $u \in U$ and $0$ otherwise.
\par
Now we will consider the second term of \eqref{eq:2.17}.
Let 
\[
\widetilde{a}_{n_1,...,n_{K},n}=a_{n_1,...,n_{K},n}- a_{n_1,...,n_{K},0}, \qquad 
b_n=\hspace{-0.4cm}\sup_{\substack{0 \leq n_j \leq m_j; j \in [K], \\ (\forall j \in [K]) \,2\,|\,n_j}} 
 \hspace{-0.4cm}|\widetilde{a}_{n_1,...,n_{K},n}|.
\]
By the case $K=1$ of Lemma \ref{lem:2.15} applied for $m=m_{K+1}$ we get 
\begin{equation}
 \label{eq:2.19}
\sup_{\substack{0 \leq n_{K+1} \leq m_{K+1} , \\ \,2\,|\,n_{K+1}}} |b_{n_{K+1}}| \leq \sum_{i=1}^{s_{K+1}} \Big( \sum_{\substack{1 \leq j \leq 2^{s_{K+1}-i}, \\ j2^i \leq m_{K+1}}}\Big| b_{j2^i}-b_{(j-1)2^i} \Big|^2 \Big)^{1/2}.
\end{equation}
Notice that for fixed $i,j$ by \eqref{eq:sumDelkdif} we have
\begin{equation}
\begin{split}
\label{eq:2.22}
    &|b_{j2^i}-b_{(j-1)2^i}|\\
    &= \Big|\sup_{\substack{0 \leq n_j \leq m_j; j \in [K], \\ (\forall j \in [K]) \,2\,|\,n_j}}
 |\widetilde{a}_{n_1,...,n_{K},j2^i}|-\sup_{\substack{0 \leq n_j \leq m_j; j \in [K], \\ (\forall j \in [K]) \,2\,|\,n_j}}
 |\widetilde{a}_{n_1,...,n_{K},(j-1)2^i}|\Big| 
 \\
 &\le \sup_{\substack{0 \leq n_j \leq m_j; j \in [K], \\ (\forall j \in [K]) \,2\,|\,n_j}} \Big|\widetilde{a}_{n_1,...,n_{K},j2^i}- \widetilde{a}_{n_1,...,n_{K},(j-1)2^i}\Big| \\
 &= \sup_{\substack{0 \leq n_j \leq m_j; j \in [K], \\ (\forall j \in [K]) \,2\,|\,n_j}} \Big|\sum_{\substack{k \in I_{j}^i \cap (2\N_0)}} \Delta^{K+1}_{n_1,...,n_K,k}(a)\Big|.
   \end{split}
\end{equation}
Let $j_{K+1}=j, i_{K+1}=i$ be such that $j2^i \leq m_{K+1}$. Applying induction hypothesis to \eqref{eq:2.22} we get
\begin{align*}
&|b_{j2^i}-b_{(j-1)2^i}| \le \Big|\sum_{\substack{k \in I_{j}^i \cap (2\N_0)}} \Delta^{K+1}_{0,...,0,k}(a)\Big|+
\\&
 \sum_{\substack{U \subseteq [K], \\ U \neq  \emptyset}} \sum_{1 \leq i_u \leq s_u; u \in U}\Big(\sum_{1 \leq j_u \leq 2^{s_u-i_u}; u \in U} \Big| \sum_{\substack{k_u \in I_{j_u}^{i_u}; u \in U, \\ (\forall u \in U) \,2\,|\,k_u, \\ (\forall u \in U)k_u \leq m_u }} \sum_{\substack{k_{K+1} \in I_{j}^i \cap (2\N_0)}}\Delta^{U\cup \{K+1\}}_{\overline{k}(U \cup \{K+1\})}(a)) \Big|^2 
\Big)^{1/2} 
\\
&=\Big|\sum_{\substack{k \in I_{j}^i \cap (2\N_0)}} \Delta^{K+1}_{0,...,0,k}(a)\Big|\\
&+\sum_{\substack{U \subseteq [K+1], \\ K+1 \in U, \\ U \neq  \{K+1\}}} \sum_{1 \leq i_u \leq s_u; u \in U \setminus \{K+1 \} } \Big(\sum_{1 \leq j_u \leq 2^{s_u-i_u}; u \in U \setminus \{K+1 \} } 
\Big| \sum_{\substack{k_u \in I_{j_u}^{i_u}; u \in U, \\ (\forall u \in U) \,2\,|\,k_u, \\ (\forall u \in U)k_u \leq m_u }} \Delta^{U}_{\overline{k}(U)}(a) \Big|^2 
\Big)^{1/2}.
\end{align*}
Plugging the above into \eqref{eq:2.19} and using Minkowski's inequality twice we get
\begin{align*}
&\sup_{\substack{0 \leq n_{K+1} \leq m_{K+1} , \\ n_{K+1} \equiv 0 \pmod{2}}} |b_{n_{K+1}}| 
\le \sum_{i_{K+1}=1}^{s_{K+1}} \Big( \sum_{\substack{1 \leq j_{K+1} \leq 2^{s_{K+1}-i_{K+1}}, \\ j_{K+1}2^{i_{K+1}} \leq m_{K+1} }}\Big| b_{j_{K+1}2^{i_{K+1}}}-b_{(j_{K+1}-1)2^{i_{K+1}}} \Big|^2 \Big)^{1/2} 
\\
&\le
\sum_{i_{K+1}=1}^{s_{K+1}} \Bigg( \sum_{\substack{1 \leq j_{K+1} \leq 2^{s_{K+1}-i_{K+1}}, \\ j_{K+1}2^{i_{K+1}} \leq m_{K+1} }}\Bigg[\Big|\sum_{\substack{k \in I_{j_{K+1}}^{i_{K+1}} \cap (2\N_0)}} \Delta^{K+1}_{0,...,0,k}(a)\Big| 
 \\
&+ \sum_{\substack{U \subseteq [K+1], \\ K+1 \in U, \\ U \neq  \{K+1\}}} \sum_{1 \leq i_u \leq s_u; u \in U\setminus \{K+1 \} } \Big(\sum_{1 \leq j_u \leq 2^{s_u-i_u}; u \in U\setminus \{K+1 \} } \Big| \sum_{\substack{k_u \in I_{j_u}^{i_u}; u \in U, \\ (\forall u \in U) \,2\,|\,k_u, \\ (\forall u \in U)k_u \leq m_u }} \Delta^{U}_{\overline{k}(U)}(a) \Big|^2\Big)^{1/2}\Bigg]^2\Bigg)^{1/2} 
\\
&\le
\sum_{i_{K+1}=1}^{s_{K+1}} \Big( \sum_{1 \leq j_{K+1} \leq 2^{s_{K+1}-i_{K+1}}}
\Big|\sum_{\substack{k \in I_{j_{K+1}}^{i_{K+1}} \cap (2\N_0) \cap [0,m_{K+1}]}} \Delta^{K+1}_{0,...,0,k}(a)\Big|
^2 \Big)^{1/2} +\sum_{i_{K+1}=1}^{s_{K+1}} \sum_{\substack{U \subseteq [K+1], \\ K+1 \in U, \\ U \neq  \{K+1\}}} 
\\
&\sum_{1 \leq i_u \leq s_u; u \in U\setminus \{K+1 \} } \Big( \sum_{1 \leq j_{K+1} \leq 2^{s_{K+1}-i_{K+1}}}\sum_{1 \leq j_u \leq 2^{s_u-i_u}; u \in U\setminus \{K+1 \} } \Big| \sum_{\substack{k_u \in I_{j_u}^{i_u}; u \in U, \\ (\forall u \in U) \,2\,|\,k_u, \\ (\forall u \in U) k_u \leq m_u }} \Delta^{U}_{\overline{k}(U)}(a) \Big|^2 
 \Big)^{1/2}.
\end{align*}
Rewriting the large sum in the two lines above we reach
\begin{equation} \label{eq:2.24}
\begin{split}
&\sup_{\substack{0 \leq n_{K+1} \leq m_{K+1} , \\ n_{K+1} \equiv 0 \pmod{2}}} |b_{n_{K+1}}|\\
&\le
 \sum_{\substack{U \subseteq [K+1], \\ K+1 \in U}} \sum_{1 \leq i_u \leq s_u; u \in U} \Big( \sum_{1 \leq j_u \leq 2^{s_u-i_u}; u \in U } 
\Big| \sum_{\substack{k_u \in I_{j_u}^{i_u}; u \in U, \\ (\forall u \in U) \,2\,|\,k_u, \\ (\forall u \in U) k_u \leq m_u }} \Delta^{U}_{\overline{k}(U)}(a) \Big|^2 
 \Big)^{1/2}.
 \end{split}
\end{equation}
Combining \eqref{eq:2.17}, \eqref{eq:2.18}, and \eqref{eq:2.24} we complete the proof of the inductive step and thus also the proof of Lemma \ref{lem:2.15}.
\end{proof}

Using Lemma \ref{lem:2.15} we shall now derive an estimate for the square function of full maximal functions corresponding to the averages $\mathcal D_{\overline{n}}$.  This estimate will be given in terms of the behavior of the differencing operator $\Delta^U_{(\cdot)}$ applied to the multipliers $\beta_{(\cdot)}(\xi)$ treated as the sequence $(\beta_{n}(\xi)\colon n\in (\N_0)^K).$ 
\begin{Cor} \label{cor:2.16}
For any $K \in \N$, $\epsilon \in \{0,1\}^K$, $V \subseteq [K]$, $d \in \N$ and $f \in \ell^2(\Z^d)$  we have
\begin{equation}
\begin{split}
\label{eq:cor:2.16}
&\Big\| \Big( \hspace{-0.5cm}\sum_{\substack{1 \leq s_j \leq \log_2(\frac{d}{2K}+\epsilon_j); j\in V}} \sup_{\substack{k_u \in [2^{s_u}-\epsilon_u,2^{s_u+1}-\epsilon_u]; u \in V, \\ (\forall u \in V)
k_u  \equiv \epsilon_u \pmod{2}, \\
(\forall u \in V)k_u \leq \frac{d}{2K}
}} \Big| \mathcal{D}_{\overline{k}(V)}f-\mathcal{D}_{\overline{2^{s}}(V)-\epsilon(V)} f \Big|^2 \Big)^{1/2} \Big\|_{2}
\\
&\le \sum_{\substack{U \subseteq V, \\ U \neq \emptyset}} \sum_{0 \leq i_u ; u \in U}   \Big( \sum_{\substack{i_j+1 \leq s_j\leq \log_2(\frac{d}{2K}+\epsilon_j); j\in V}}
\sum_{1 \leq j_u \leq 2^{i_u}; u \in U}
\\
&\hspace{2cm}\int_{\T^d} \Big| \hspace{-0.5cm}\sum_{\substack{k_u \in I_{j_u}^{s_u-i_u}; u \in U, \\ (\forall u \in U)
k_u  \equiv 0 \pmod{2}, \\
(\forall u \in U)k_u \leq \frac{d}{2K} -2^{s_u}+\epsilon_u}} \hspace{-0.5cm}\Delta^U_{\overline{k}(U)+\overline{2^s}(V)-\epsilon(V)}\big(\beta_{(\cdot)}(\xi)\big) \Big|^2  |\widehat{f}(\xi)|^2 d\xi \Big)^{1/2},
\end{split}
\end{equation}
where $\overline{2^s}(V)-\epsilon(V)$ is the vector of length $K$, whose $u$-th coordinate is $2^{s_u}-\epsilon_u$ if $u \in V$ and $0$ otherwise.
\end{Cor}
\begin{proof}
We apply Lemma \ref{lem:2.15} to the supremum above. For fixed $f$ and $x$ we take the sequence 
\[
a_{\overline{k}}=  
\Big|\mathcal{D}_{\overline{k}(V)+\overline{2^{s}}(V)-\epsilon(V)}f(x)-\mathcal{D}_{\overline{2^{s}}(V)-\epsilon(V)}f(x) \Big|
\]
and $\overline{m} \in \N_0^K \cap \prod_{i=1}^K [0,2^{s_i}]$ defined by $m_u= \min(2^{s_u},\frac{d}{2K}-2^{s_u}+\epsilon_u)$. Then $a_{\overline{0}}=0$, so
one obtains
\begin{align*}
&\sup_{\substack{k_u \in [2^{s_u}-\epsilon_u,2^{s_u+1}-\epsilon_u]; u \in V, \\ (\forall u \in V)
k_u  \equiv \epsilon_u \pmod{2}, \\
(\forall u \in V)k_u \leq \frac{d}{2K}
}} \Big| \mathcal{D}_{\overline{k}(V)}f(x)-\mathcal{D}_{\overline{2^{s}}(V)-\epsilon(V)} f(x) \Big|
\le
\sum_{\substack{U \subseteq V, \\ U \neq \emptyset}} \sum_{1 \leq i_u \leq s_u; u \in U} \\
&\Big(\sum_{1 \leq j_u \leq 2^{s_u-i_u}; u \in U} 
\Big|\hspace{-0.5cm} \sum_{\substack{k_u \in I_{j_u}^{i_u}; u \in U, \\ (\forall u \in U)
k_u  \equiv 0 \pmod{2}, \\
(\forall u \in U)k_u \leq \frac{d}{2K} -2^{s_u}+\epsilon_u
}}\hspace{-0.5cm} \Delta^U_{\overline{k}(U)}\Big(\mathcal{D}_{(\cdot)+\overline{2^{s}}(V)-\epsilon(V)}f(x)\Big) \Big|^2
\Big)^{1/2}
\\
&= \hspace{-0.2cm}
\sum_{\substack{U \subseteq V, \\ U \neq \emptyset}} \sum_{0 \leq i_u \leq s_u-1; u \in U} \hspace{-0.2cm}\Big(\sum_{1 \leq j_u \leq 2^{i_u}; u \in U} 
\Big| \hspace{-0.5cm} \sum_{\substack{k_u \in I_{j_u}^{s_u-i_u}; u \in U, \\ (\forall u \in U)
k_u  \equiv 0 \pmod{2}, \\
(\forall u \in U)k_u \leq \frac{d}{2K} -2^{s_u}+\epsilon_u
}} \hspace{-0.5cm}\Delta^U_{\overline{k}(U)}\Big(\mathcal{D}_{(\cdot)+\overline{2^{s}}(V)-\epsilon(V)}f(x)\Big) \Big|^2
\Big)^{1/2}\hspace{-0.3cm}.
\end{align*}
In the above inequality the term $\mathcal{D}_{\overline{2^{s}}(V)-\epsilon(V)} f(x)$ vanished (it does not depend on $\overline{k}$). Moreover, since $a$ depends only on coordinates from $V$ the outermost sum is over $U \subseteq V$ (otherwise $\Delta^U(a)$ would be $0$). In the last line we've replaced $i_u$ by $s_u-i_u$. Using the above inequality and \eqref{eq:DSnk} we may estimate the left hand side of \eqref{eq:cor:2.16} by 
\begin{align*}
&\Big\| \Big( \sum_{\substack{1 \leq s_j \leq \log_2(\frac{d}{2K}+\epsilon_j); j\in V}} \Big( \sum_{\substack{U \subseteq V, \\ U \neq \emptyset}} \sum_{ 0\leq i_u \leq s_u-1; u \in U} 
\Big(\sum_{1 \leq j_u \leq 2^{i_u}; u \in U} 
\\
&\hspace{2.5cm}\Big| \hspace{-0.5cm} \sum_{\substack{k_u \in I_{j_u}^{s_u-i_u}; u \in U, \\ (\forall u \in U)
k_u  \equiv 0 \pmod{2}, \\
(\forall u \in U)k_u \leq \frac{d}{2K} -2^{s_u}+\epsilon_u
}} \hspace{-0.5cm}\Delta^U_{\overline{k}(U)+\overline{2^{s}}(V)-\epsilon(V)}\Big(\mathcal{D}_{(\cdot)}f\Big) \Big|^2
\Big)^{1/2}\Big)^2 \Big)^{1/2} \Big\|_{2}.
\end{align*}
Hence, applying Minkowski's inequality twice we obtain

\begin{align*}
&\Big\| \Big( \sum_{\substack{1 \leq s_j \leq \log_2(\frac{d}{2K}+\epsilon_j); j\in V}} \sup_{\substack{k_u \in [2^{s_u}-\epsilon_u,2^{s_u+1}-\epsilon_u]; u \in V, \\ (\forall u \in V)
k_u  \equiv \epsilon_u \pmod{2}, \\
(\forall u \in V)k_u \leq \frac{d}{2K}
}} \Big| \mathcal{D}_{\overline{k}(V)}f-\mathcal{D}_{\overline{2^{s}}(V)-\epsilon(V)} f \Big|^2 \Big)^{1/2} \Big\|_{2}\le
 \sum_{\substack{U \subseteq V, \\ U \neq \emptyset}} \sum_{0 \leq i_u ; u \in U}\\
& \Big\|  \Big( \sum_{\substack{i_j+1 \leq s_j \leq \log_2(\frac{d}{2K}+\epsilon_j); j\in V,}}
\sum_{1 \leq j_u \leq 2^{i_u}; u \in U}  \Big| \hspace{-0.5cm}\sum_{\substack{k_u \in I_{j_u}^{s_u-i_u}; u \in U, \\ (\forall u \in U)
k_u  \equiv 0 \pmod{2}, \\
(\forall u \in U)k_u \leq \frac{d}{2K} -2^{s_u}+\epsilon_u
}} \hspace{-0.5cm}\Delta^U_{\overline{k}(U)+\overline{2^{s}}(V)-\epsilon(V)}\Big(\mathcal{D}_{(\cdot)}f\Big) \Big|^2
\Big)^{1/2}\Big\|_{2} 
\end{align*}
Finally, Parseval's formula shows that the quantity on the right hand side of the inequality above equals the right hand side of \eqref{eq:cor:2.16}. This completes the proof of Corollary \ref{cor:2.16}.
\end{proof} 
Our goal now will be to bound $|\Delta^U_{\overline{k}(U)+\overline{2^s}(V)-\epsilon(V)}(\beta_{(\cdot)}(\xi))|$. For this purpose we introduce a crucial formula.
\begin{Lem} \label{lem:2.17}
For any $K \in \N$, finite set $J \subseteq \N$ with $|J| \geq 2$, $j \in [K]$ and \\$\overline{n}=(n_1,...,n_K) \in [K]$ with $n_j \geq 2$ and $n_1+...+n_K \le |J|$ we have
\[
\Delta^j_{\overline{n}}(\beta^J(\xi))\hspace{-0.1cm}= \hspace{-0.1cm}\frac{-4}{|J|(|J|-1)} \hspace{-0.1cm}\sum_{\substack{x,y \in J , \\ x \neq y}} \sin^2(j \pi \xi_x) \hspace{-0.05cm}\cos^2(j \pi \xi_y) \beta_{n_1,...,n_{j-1},n_j-2,n_{j+1},...,n_K}^{J \setminus \{x,y\}}(\xi).
\]
For any $U \subseteq [K]$ and finite $J \subseteq \N$ with $|J| \geq 2|U|$, if  $\overline{n}=(n_1,...,n_K) \in [K]$ satisfy $n_k \geq 2$ for all $k \in U$ and $n_1+...+n_K \le |J|$, then we have
\begin{align*}
&\Delta_{\overline{n}}^U(\beta^J(\xi))= \frac{(-4)^{|U|}}{|J|(|J|-1)...(|J|-2|U|+1)} \\
&\cdot \sum_{\substack{x_i,y_i \in J; i \in U, \\ \text{all distinct} 
}} \prod_{k \in U}\big( \sin^2(k \pi \xi_{x_k}) \cos^2(k \pi \xi_{y_k}) \big) \cdot \beta^{J \setminus \{ x_k,y_k: k \in U\}}_{\overline{n} -\overline{2}(U)}(\xi),
\end{align*}
where $\overline{2}(U) \in \N_0^K$ is a vector, whose $i$-th coordinate is $2$ if $i \in U$ and $0$ otherwise.
\end{Lem}
\begin{proof}
    We will prove only the first part, since the second one follows from inductively applying the first part.
    
    Fix $j\in [K]$ and denote $m=\sum_{k\ne j}n_k.$ From Lemma \ref{rem:2.5} we have 
    \begin{multline} \label{eq:2.25}
    \Delta^j_{\overline{n}}(\beta^J_{\overline{n}}(\xi))=
    \se{\substack{I_1,...,I_{j-1},I_{j+1},...,I_K \subseteq J, \\
    (\forall k \in [K] \setminus \{j \}) |I_k|=n_k, \\
    (\forall i,k \in [K] \setminus \{j\}, i \neq k) I_i \cap I_k=\emptyset
    }} \prod_{k=1, k \neq j}^K \prod_{i \in I_k} \cos(2k \pi \xi_i) \cdot
    \\
    \sum_{U \subseteq J \setminus  ( \cup_{k \neq j} I_k )} \hspace{-0.3cm}\Big(\kr_{n_j}^{(|J|- m)}(|U|) -\kr_{n_j-2}^{(|J|- m )}(|U|) \Big)\cdot \hspace{-0.5cm}\prod_{i \in J \setminus ( \cup_{k \neq j}I_k) \setminus U} \hspace{-0.5cm} \cos^2(j \pi \xi_i) \cdot \prod_{i \in U} \sin^2(j \pi \xi_i).
    \end{multline}
    Fix $I_1,...,I_{j-1},I_{j+1},...,I_K \subseteq J$ and let $\widetilde{I}= \cup_{k \neq j}I_k$. We will look only at the second line of \eqref{eq:2.25}. By the last point of Theorem \ref{thm:2.4} we have
    \[
    \kr_{n_j}^{(|J|- m)}(|U|) -\kr_{n_j-2}^{(|J|- m)}(|U|)= \frac{-4|U|(|J|-m-|U|)}{(|J|-m)(|J|-m-1)} \kr_{n_j-2}^{(|J|-m-2)}(|U|-1).
    \]
   Denote $C(m,|J|)=-4((|J|-m)(|J|-m-1))^{-1}$ and notice that
    \[
 |U|(|J|-m-|U|) \kr_{n_j-2}^{(|J|-m-2)}(|U|-1)=   \sum_{x \in U} \sum_{y \in J \setminus \widetilde{I} \setminus U} \kr_{n_j-2}^{(|J|-m-2)} (|U \setminus \{x \}|),
    \]
    if $U= \emptyset $ or $U=J \setminus \widetilde{I}$, then we encounter sum over empty set on the right-hand side, which we define to be equal to 0. Using the above we get
    \begin{equation*} 
    \begin{split}
      &\sum_{U \subseteq J \setminus \widetilde{I}} \Big(\kr_{n_j}^{(|J|- m)}(|U|) -\kr_{n_j-2}^{(|J|- m)}(|U|) \Big)\cdot \prod_{i \in J \setminus  \widetilde{I} \setminus U} \cos^2(j \pi \xi_i) \cdot \prod_{i \in U} \sin^2(j \pi \xi_i) 
      \\
      &= C(m,|J|)\sum_{U \subseteq J \setminus \widetilde{I}} \sum_{x \in U} \sum_{y \in J \setminus \widetilde{I} \setminus U} \kr_{n_j-2}^{(|J|-m-2)} (|U \setminus \{x \}|)
      \\
      &\prod_{i \in J \setminus  \widetilde{I} \setminus U} \cos^2(j \pi \xi_i) \cdot \prod_{i \in U} \sin^2(j \pi \xi_i)
      =
      C(m,|J|) \\ 
      &\cdot \sum_{\substack{x,y \in J \setminus \widetilde{I}, \\ x \neq y}} \sum_{\substack{U \subseteq J \setminus \widetilde{I} \setminus \{y\}, \\
      x \in U}}  \kr_{n_j-2}^{(|J|-m-2)} (|U \setminus \{x \}|)\prod_{i \in J \setminus  \widetilde{I} \setminus U} \cos^2(j \pi \xi_i) \cdot \prod_{i \in U} \sin^2(j \pi \xi_i)
      \\
      &=\frac{-4}{(|J|-m)(|J|-m-1)}\sum_{\substack{x,y \in J \setminus \widetilde{I}, \\ x \neq y}}\sin^2(j \pi \xi_x) \cos^2(j \pi \xi_y) 
      \\
      &\cdot 
      \sum_{\substack{U' \subseteq J  \setminus \{x,y \} \setminus \widetilde{I}}}  \kr_{n_j-2}^{(|J \setminus \{x,y\} \setminus \widetilde{I}|)} (|U'|)\prod_{i \in J \setminus \{x,y\} \setminus  \widetilde{I} \setminus U'} \cos^2(j \pi \xi_i) \cdot \prod_{i \in U'} \sin^2(j \pi \xi_i).
      \end{split}
   \end{equation*}
   Now, the reasoning from the proof of Lemma \ref{rem:2.5}, cf.\ \eqref{eq:Lem2.5calc}, \eqref{eq:Lem2.5calc'} shows that
    \[
    \sum_{\substack{U' \subseteq J  \setminus \{x,y \} \setminus \widetilde{I}}}  \kr_{n_j-2}^{(|J \setminus \{x,y\} \setminus \widetilde{I}|)} (|U'|)\prod_{i \in J \setminus \{x,y\} \setminus  \widetilde{I} \setminus U'} \cos^2(j \pi \xi_i) \cdot \prod_{i \in U'} \sin^2(j \pi \xi_i)\]
    equals
    \[\se{\substack{I_j \subseteq J \setminus \{ x,y\} \setminus \widetilde{I}, \\ |I_j|=n_j-2}} \prod_{k \in I_j} \cos(2j \pi \xi_k).
    \]
    
    Combining the above and coming back to \eqref{eq:2.25} we get
   \begin{align*}
    &\Delta^j_{\overline{n}}(\beta^J_{\overline{n}}(\xi))=
    \frac{\prod_{i \neq j} n_i! \cdot (|J|-m)!}{|J|!} \hspace{-0.8cm}\sum_{\substack{I_1,...,I_{j-1},I_{j+1},...,I_K \subseteq J, \\
    (\forall k \in [K] \setminus \{j \}) |I_k|=n_k, \\
    (\forall i,k \in [K] \setminus \{j\}, i \neq k) I_i \cap I_k=\emptyset
    }} \hspace{-0.1cm}\prod_{k=1, k \neq j}^K \prod_{i \in I_k} \cos(2k \pi \xi_i) 
    \\
    &\cdot
    \frac{-4}{(|J|-m)(|J|-m-1)} \frac{(n_j-2)!(|J|-\sum_{i \in [K]} n_i)!}{(|J|-2-m)!}
   \\
   &
    \cdot \sum_{\substack{x,y \in J \setminus (\cup_{k \neq j}I_k), \\ x \neq y}}\sin^2(j \pi \xi_x) \cos^2(j \pi \xi_y) \sum_{\substack{I_j \subseteq J \setminus \{ x,y\} \setminus (\cup_{k \neq j}I_k), \\ |I_j|=n_j-2}} \prod_{k \in I_j} \cos(2j \pi \xi_k).
    \end{align*}
    Rewriting the right hand side of this equality we reach
   \begin{align*}
    & \Delta^j_{\overline{n}}(\beta^J_{\overline{n}}(\xi))= \frac{-4}{|J|(|J|-1)}\sum_{\substack{x,y \in J, \\ x \neq y}}\sin^2(j \pi \xi_x) \cos^2(j \pi \xi_y) \frac{\prod_{i \neq j} n_i! (n_j-2)! \cdot (|J|-\sum_{i \in [K]} n_i)!}{(|J|-2)!} 
    \\&\hspace{1cm}\sum_{\substack{I_1,...,I_{j-1},I_{j+1},...,I_K \subseteq J \setminus \{x,y\}, \\
    (\forall k \in [K] \setminus \{j \}) |I_k|=n_k, \\
    (\forall i,k \in [K] \setminus \{j\}, i \neq k) I_i \cap I_k=\emptyset
    }} \hspace{+0.1cm} \prod_{k=1, k \neq j}^K \prod_{i \in I_k} \cos(2k \pi \xi_i) \cdot
    \sum_{\substack{I_j \subseteq J \setminus \{ x,y\} \setminus (\cup_{k \neq j}I_k), \\ |I_j|=n_j-2}} \prod_{k \in I_j} \cos(2j \pi \xi_k)
    \\
     &=\frac{-4}{|J|(|J|-1)} \sum_{\substack{x,y \in J, \\ x \neq y}} \sin^2(j \pi \xi_x) \cos^2(j \pi \xi_y) \beta_{n_1,...,n_{j-1},n_j-2,n_{j+1},...,n_K}^{J \setminus \{x,y\}}(\xi),
   \end{align*}
   which completes the proof.
    \end{proof}
 Lemma \ref{lem:2.17} together with Lemma \ref{lem:2.8}  allows us to prove crucial bounds for $|\Delta_{\overline{n}}^U(\beta_{(.)}(\xi))|.$
    \begin{Cor} \label{cor:2.18}
        Fix $K \in \N$, $d \in \N$ with $d \geq 4K$, $\xi \in \T^d$. Take any set $U \subseteq [K]$ and $n_1,...n_K \leq \frac{d}{2K}$ satisfying $n_k \geq 2$ for all $k \in U$.
        Let 
        \[
A= A(\xi)=\{ j \in [K]: \|j \xi+ 1/2 \| \leq \| j \xi \| \}.
\]
Then we have
        \[
        \Big|\Delta_{\overline{n}}^U(\beta(\xi)) \Big| \lesssim_K \prod_{k \in U} \frac{1}{n_k} \cdot \prod_{k \in U} \min\Big( \frac{n_k}{d} \|k\xi + \frac{\mathds{1}_A(j)}{2} \|^2, \big( \frac{n_k}{d} \|k\xi + \frac{\mathds{1}_A(j)}{2} \|^2\big)^{-1} \Big).
        \]
    \end{Cor}
    \begin{proof}
        Since $d \geq 4K$, by Lemma \ref{lem:2.17} and Lemma \ref{lem:2.8}  we get
        \begin{align*}
         &\Big|\Delta_{\overline{n}}^U(\beta(\xi)) \Big| \lesssim_K \frac{1}{d^{2|U|}} \sum_{\substack{x_i,y_i \in [d]; i \in U, \\ \text{all distinct} 
}} \prod_{k \in U} \sin^2(k \pi \xi_{x_k}) \cos^2(k \pi \xi_{y_k}) \cdot 
 \Big|\beta^{[d] \setminus \{ x_k,y_k: k \in U\}}_{\overline{n} -\overline{2}(U)}(\xi) \Big|
       \\
        &\lesssim_K \frac{1}{d^{2|U|}} \sum_{\substack{x_i,y_i \in [d]; i \in U, \\ \text{all distinct} 
}} \prod_{k \in U} \sin^2(k \pi \xi_{x_k}) \cos^2(k \pi \xi_{y_k}) 
\\
&\cdot \prod_{j \in U}\exp \Big(- \frac{c(n_j-2)}{80K(d-2|U|)}  \min \Big( \sum_{i \in [d] \setminus \{x_k,y_k: k \in U \}} \sin^2(j \pi \xi_i), \sum_{i \in [d] \setminus \{x_k,y_k: k \in U \}} \cos^2(j \pi \xi_i) \Big) \Big).
\end{align*}
Hence, observing that
\begin{align*}
     &\exp \Big(- \frac{c(n_j-2)}{80K(d-2|U|)}  \min \Big( \sum_{i \in [d] \setminus \{x_k,y_k: k \in U \}} \sin^2(j \pi \xi_i), \sum_{i \in [d] \setminus \{x_k,y_k: k \in U \}} \cos^2(j \pi \xi_i) \Big) \Big)\\
     &\lesssim_K \exp \Big(- \frac{cn_j}{80 Kd}  \min \Big( \sum_{i \in [d] } \sin^2(j \pi \xi_i), \sum_{i \in [d] } \cos^2(j \pi \xi_i) \Big) \Big) 
\end{align*}
we obtain
\begin{align*}
 &\Big|\Delta_{\overline{n}}^U(\beta(\xi)) \Big|\\\
 &\lesssim_K \prod_{k \in U} \exp \Big(- \frac{cn_k}{80 Kd}  \min \Big( \sum_{i \in [d] } \sin^2(k \pi \xi_i), \sum_{i \in [d] } \cos^2(k \pi \xi_i) \Big) \Big) \cdot \prod_{k \in U} \Big( \frac{\| k \xi \|^2}{d} \cdot \frac{\|k \xi + \frac{1}{2} \|^2}{d} \Big)
\\
&\lesssim_K \prod_{k \in U}  \frac{1}{n_k}\frac{n_k}{d}\| k \xi+ \frac{\mathds{1}_A(k)}{2}\|^2  \exp \Big( - \frac{cn_k}{80Kd} \| k \xi+ \frac{\mathds{1}_A(k)}{2}\|^2 \Big) 
\\
&\lesssim_K \prod_{k \in U} \frac{1}{n_k} \min\Big( \frac{n_k}{d} \|k\xi + \frac{\mathds{1}_A(k)}{2} \|^2, \big( \frac{n_k}{d} \|k\xi + \frac{\mathds{1}_A(k)}{2} \|^2\big)^{-1} \Big),
\end{align*}
where in the last inequality we used the simple bound $se^{-s}\lesssim \min(s,s^{-1}),$ $s>0.$ This completes the proof of Corollary \ref{cor:2.18}.
\end{proof}

\subsection{Proof of Theorem \ref{thm:1.5}.}
\label{sec:24}
We are finally ready to prove Theorem \ref{thm:1.5}. For $p= \infty$ the theorem is trivial; due to interpolation it suffices to consider only $p=2$. \par
Assume that $d \geq 4K$ (in the other case the theorem is trivial).
We will prove the following.\\
\textbf{Statement.}\textit{
For any $ V \subseteq [K]$ there exists $C_{V,K}$ such that for all $d \geq 4K$ and  $f \in \ell^2(\Z^d)$ we have
\[
 \Big\| \sup_{\substack{
 1 \leq n_j\leq \frac{d}{2K}; j \in V}}   |\mathcal{D}_{\overline{n}(V)}f| \Big\|_{2} \leq C_{V,K} \big\| f \big\|_{2},
\]
}\\
Notice that the statement implies Theorem \ref{thm:1.5} when $d \geq 4K$, because 
\[
 \Big\| \sup_{\substack{
 0 \leq n_j \leq \frac{d}{2K}; j \in [K]}}   |\mathcal{D}_{\overline{n}}f| \Big\|_{2} \leq \sum_{V \subseteq [K]}  \Big\| \sup_{\substack{
 1 \leq n_j \leq \frac{d}{2K}; j \in V}}   |\mathcal{D}_{\overline{n}(V)}f| \Big\|_{2} \leq C_K \|f \|_{2}, 
\]
where $C_K=\sum_{V \subseteq [K]} C_{V,K}$. The reason for introducing the above statement instead of justifying Theorem \ref{thm:1.5} directly is purely technical and is reflected in \eqref{eq:thm:1.5:1} below.

\par
\begin{proof}[Proof of the statement above]
Fix $V \subseteq [K].$ We decompose any $f \in \ell^2(\Z^d)$ it in the following way
\[
f= \sum_{A \subseteq [K]} f_A,
\]
where $f_A$ satisfies
\[
\supp(\widehat{f_A}) \hspace{-0.1cm} \subseteq \hspace{-0.1cm}\Big\{\xi \in \T^d \hspace{-0.1cm}:  \hspace{-0.1cm}(\forall j \in A) \| j \xi+ \frac{1}{2} \| \hspace{-0.1cm}\leq \hspace{-0.1cm}\|j \xi \|,   (\forall j \in [K] \setminus A) \|j \xi \| \hspace{-0.1cm}< \hspace{-0.1cm}\| j \xi+ \frac{1}{2} \| \Big\}.
\]
By triangle inequality we have
\[
  \Big\| \sup_{\substack{
 1 \leq n_j \leq \frac{d}{2K}; j \in V}}   |\mathcal{D}_{\overline{n}(V)}f| \Big\|_{2} \leq \sum_{A \subseteq [K]}   \Big\| \sup_{\substack{
 1 \leq n_j \leq \frac{d}{2K}; j \in V}}   |\mathcal{D}_{\overline{n}(V)}f_A| \Big\|_{2}.
\]
Fix $A \subseteq [K]$. Since
\begin{align*}
&\sup_{\substack{
 1 \leq n_j \leq \frac{d}{2K}; j \in V}}   |\mathcal{D}_{\overline{n}(V)}f|\le \sup_{\epsilon\in\{0,1\}^K} \hspace{-0.4cm}\sup_{\substack{
 n_j+\epsilon_j \in \D; j \in V, \\
(\forall j \in V) 2-\epsilon_j \leq n_j\leq \frac{d}{2K}}}\sup_{\substack{k_u \in [n_u,2n_{u}+\epsilon_u]; u \in V, \\ (\forall u \in V)
k_u  \equiv \epsilon_u \pmod{2}, \\
(\forall u \in V)k_u \leq \frac{d}{2K}
}} \hspace{-0.3cm} \Big| \mathcal{D}_{\overline{k}(V)}f\Big|\\
&\le \sup_{\epsilon\in\{0,1\}^K}\sup_{\substack{
 n_j+\epsilon_j \in \D; j \in V, \\
(\forall j \in V) n_j\leq \frac{d}{2K}}}\Big| \mathcal{D}_{\overline{n}(V)}f\Big|\\
&\hspace{0.5cm}+\sup_{\epsilon\in\{0,1\}^K}\sup_{\substack{
 n_j+\epsilon_j \in \D; j \in V, \\
(\forall j \in V) 2- \epsilon_j \leq n_j\leq \frac{d}{2K}}}\sup_{\substack{k_u \in [n_u,2n_{u}+\epsilon_u]; u \in V, \\ (\forall u \in V)
k_u  \equiv \epsilon_u \pmod{2}, \\
(\forall u \in V)k_u \leq \frac{d}{2K}
}} \Big| \mathcal{D}_{\overline{k}(V)}f-\mathcal{D}_{\overline{n}(V)}f\Big|
\end{align*}
we see that
\begin{equation}
\label{eq:thm:1.5:1}
\begin{split}
&\Big\| \sup_{\substack{
 1 \leq n_j \leq \frac{d}{2K}; j \in V}}   |\mathcal{D}_{\overline{n}(V)}f_A| \Big\|_{2} \leq  \sum_{ \epsilon \in \{0,1 \}^K} \Big\| \sup_{\substack{
 n_j+\epsilon_j \in \D; j \in V, \\
(\forall j \in V) n_j\leq \frac{d}{2K}}}   |\mathcal{D}_{\overline{n}(V)}f_A| \Big\|_{2} 
+ \sum_{ \epsilon \in \{0,1 \}^K} \\
&\Big\| \Big( \hspace{-0.4cm} \sum_{\substack{1 \leq s_j \leq \log_2(\frac{d}{2K}+\epsilon_j); j\in V}} \sup_{\substack{k_u \in [2^{s_u}-\epsilon_u,2^{s_u+1}-\epsilon_u]; u \in V, \\ (\forall u \in V)
k_u  \equiv \epsilon_u \pmod{2}, \\
(\forall u \in V)k_u \leq \frac{d}{2K}
}} \Big| \mathcal{D}_{\overline{k}(V)}f-\mathcal{D}_{\overline{2^{s}}(V)-\epsilon(V)} f \Big|^2 \Big)^{1/2} \Big\|_{2},
\end{split}
\end{equation}
where $\overline{2^s}(V)-\epsilon(V)$ is the vector of length $K$, whose $u$-th coordinate is $2^{s_u}-\epsilon_u$ if $u \in V$ and $0$ otherwise. The first term on the right hand side of \eqref{eq:thm:1.5:1} is bounded by $C_K \| f_A \|_{2}$ due to Proposition \ref{prop:2.9}. 
\par For the second term we use Corollary \ref{cor:2.16} and get
\begin{equation} \label{eq:2.27}
\begin{split}
&\sum_{ \epsilon \in \{0,1 \}^K} \Big\| \Big( \sum_{\substack{1 \leq s_j \leq \log_2(\frac{d}{2K}+\epsilon_j); j\in V}} \\
&\hspace{3cm}\sup_{\substack{k_u \in [2^{s_u}-\epsilon_u,2^{s_u+1}-\epsilon_u]; u \in V, \\ (\forall u \in V)
k_u  \equiv \epsilon_u \pmod{2}, \\
(\forall u \in V)k_u \leq \frac{d}{2K}
}} \Big| \mathcal{D}_{\overline{k}(V)}f-\mathcal{D}_{\overline{2^{s}}(V)-\epsilon(V)} f \Big|^2 \Big)^{1/2} \Big\|_{2} 
\\
&\le 
 \sum_{ \epsilon \in \{0,1 \}^K} \sum_{\substack{U \subseteq V, \\ U \neq \emptyset}} \sum_{0 \leq i_u; u \in U} \Big( \sum_{i_u+1 \leq s_u \leq \log_2(\frac{d}{2K}+\epsilon_u); u \in U} \sum_{1 \leq j_u \leq 2^{i_u}; u \in U} 
 \\
 &\hspace{1cm}\int_{\T^d} \Big|  \hspace{-0.5cm}\sum_{\substack{k_u \in I_{j_u}^{s_u-i_u}; u \in U, \\ (\forall u \in U)
k_u  \equiv 0 \pmod{2}, \\
(\forall u \in U) k_u \leq \frac{d}{2K}-2^{s_u}+\epsilon_u
}} \hspace{-0.5cm}\Delta^U_{\overline{k}(U)+\overline{2^s}(V)-\epsilon(V)}\big(\beta(\xi)\big) \Big|^2  |\widehat{f_A}(\xi)|^2 d\xi \Big)^{1/2}.
\end{split}
\end{equation} 
Notice that for any $u \in U$ and $k_u \in I_{j_u}^{s_u-i_u}$ we have that $u$-th coordinate of the vector $\overline{k}(U)+\overline{2^s}(V)-\epsilon(V)$ is  between $2^{s_u}$ and $2^{s_u+1}$. Moreover  every coordinate of this vector does not exceed $ \frac{d}{2K}$. Thus by Corollary \ref{cor:2.18} for any $\xi \in \supp(\widehat{f_A})$ we have
\begin{align*}
&\Big| \sum_{\substack{k_u \in I_{j_u}^{s_u-i_u}; u \in U, \\ (\forall u \in U)
k_u  \equiv \epsilon_u \pmod{2}, \\
(\forall u \in U) k_u \leq \frac{d}{2K}-2^{s_u}+\epsilon_u
}} \Delta^U_{\overline{k}(U)+\overline{2^s}(V)-\epsilon(V)}\big(\beta(\xi)\big) \Big|^2 
\\
&\lesssim_K \prod_{u \in U}2^{2s_u-2i_u} \cdot \prod_{k \in U} 2^{-2s_k}\cdot \prod_{k \in U} \min\Big( \frac{2^{s_k}}{d} \|k\xi + \frac{\mathds{1}_A(k)}{2} \|^2, \big( \frac{2^{s_k}}{d} \|k\xi + \frac{\mathds{1}_A(k)}{2} \|^2\big)^{-1} \Big)^2.
\end{align*}
Using the above inequality and Lemma 
\ref{lem:2.13} we obtain
\begin{align*}
 &\sum_{i_u+1 \leq s_u \leq \log_2(\frac{d}{2K}+\epsilon_u); u \in U} \sum_{1 \leq j_u \leq 2^{i_u}; u \in U} 
 \int_{\T^d} \Big| \sum_{\substack{k_u \in I_{j_u}^{s_u-i_u}; u \in U, \\ (\forall u \in U)
k_u  \equiv 0 \pmod{2}, \\
(\forall u \in U) k_u \leq \frac{d}{2K}-2^{s_u}+\epsilon_u
}} \Delta^U_{\overline{k}(U)+\overline{2^s}(V)-\epsilon(V)}\big(\beta(\xi)\big) \Big|^2  |\widehat{f_A}(\xi)|^2 d\xi 
\\
&\lesssim_K
 \sum_{i_u+1 \leq s_u \leq \log_2(\frac{d}{2K}+\epsilon_u); u \in U} \sum_{1 \leq j_u \leq 2^{i_u}; u \in U} \\
 &\hspace{1cm}\int_{\T^d}  \prod_{k \in U}\Big( 2^{-2i_k} \min\Big( \frac{2^{s_k}}{d} \|k\xi + \frac{\mathds{1}_A(k)}{2} \|^2, \big( \frac{2^{s_k}}{d} \|k\xi + \frac{\mathds{1}_A(k)}{2} \|^2\big)^{-1} \Big)^2 \Big)|\widehat{f_A}(\xi)|^2 d\xi  \lesssim_{K} \prod_{u \in U} 2^{-i_u}
\\
&\hspace{0.75cm}\cdot \int_{\T^d}  \sum_{1\leq s_u \leq \log_2(\frac{d}{2K}+\epsilon_u); u \in U}  \prod_{k \in U} \min\Big( \frac{2^{s_k}}{d} \|k\xi + \frac{\mathds{1}_A(k)}{2} \|^2, \big( \frac{2^{s_k}}{d} \|k\xi + \frac{\mathds{1}_A(k)}{2} \|^2\big)^{-1} \Big)^2|\widehat{f_A}(\xi)|^2 d\xi 
\\
&\lesssim_K \prod_{u \in U} 2^{-i_u} \| \widehat{f_A} \|_{L^2( \T^d)}^2= \prod_{u \in U} 2^{-i_u} \| f_A \|_{\ell^2( \Z^d)}^2.
\end{align*}
Plugging the above into \eqref{eq:2.27} we get
\begin{align*}
&\sum_{ \epsilon \in \{0,1 \}^K} \Big\| \Big( \sum_{\substack{1 \leq s_j \leq \log_2(\frac{d}{2K}+\epsilon_j); j\in V}} \sup_{\substack{k_u \in [2^{s_u}-\epsilon_u,2^{s_u+1}-\epsilon_u]; u \in V, \\ (\forall u \in V)
k_u  \equiv \epsilon_u \pmod{2}, \\
(\forall u \in V)k_u \leq \frac{d}{2K}
}} \Big| \mathcal{D}_{\overline{k}(V)}f-\mathcal{D}_{\overline{2^{s}}(V)-\epsilon(V)} f \Big|^2 \Big)^{1/2} \Big\|_{2} 
\\
&\lesssim_K 
\sum_{ \epsilon \in \{0,1 \}^K} \sum_{\substack{U \subseteq [K], \\ U \neq \emptyset}} \sum_{1 \leq i_u; u \in U} \prod_{u \in U} 2^{-i_u/2} \| f_A \|_{\ell^2( \Z^d)} \lesssim_K \| f_A \|_{2}.
\end{align*}

Coming back to \eqref{eq:thm:1.5:1} we conclude that
\[
\Big\| \sup_{\substack{
 1 \leq n_j \leq \frac{d}{2K}; j \in V}}   |\mathcal{D}_{\overline{n}(V)}f_A| \Big\|_{2} \lesssim_K \| f_A \|_{2}.
\]
and thus we finally obtain
\[
 \Big\| \sup_{\substack{
 1 \leq n_j\leq \frac{d}{2K}; j \in V}}   |\mathcal{D}_{\overline{n}(V)}f| \Big\|_{2} \leq C_{V,K} \big\| f \big\|_{2}.
\]
This completes the proof of the statement, hence concludes the proof of Theorem \ref{thm:1.5}.
\end{proof}

\section{Lattice point counting}
\label{sec:3}
In this section we will count lattice points in balls $\sqrt{n}B$ with small radii. With applications towards maximal operators in mind, our main goals are the following two results. 
\begin{Thm}
    \label{thm:3.1}
    For any $K \in \N, \varepsilon>0$ there exists $a \in \N$, and $C=C(K, \varepsilon)>0$ such that for any $n, d \in \N$, satisfying $n \leq d^{1-\frac{1+\varepsilon}{(K+1)^2}}$ we have
    \begin{equation} \label{eq:3.1}  
    |\{ x \in  \sqrt{n}B: \sum_{\substack{i=1, \\ |x_i| \leq K}}^d |x_i|^2 \leq n-a \}| \leq \frac{C}{d} |\sqrt{n}B |. 
    \end{equation}
\end{Thm}
\begin{Cor} \label{cor:3.2}
For any $K \in \N, \varepsilon>0$ there exists $a \in \N$, and $C=C(K,\varepsilon)>0$ such that for any $n, d \in \N$, satisfying $n \leq d^{1-\frac{1+\varepsilon}{(K+1)^2}}$ we have
    \begin{equation} \label{eq:3.2}
    |\{ x \in  \sqrt{n}S : \sum_{\substack{i=1, \\ |x_i| \leq K}}^d |x_i|^2 \leq n-a \}| \leq \frac{C}{d} |\sqrt{n}S|. 
    \end{equation}
\end{Cor}
Note that if $n <a $ or $d \leq C$, then the statements of Theorem \ref{thm:3.1} and Corollary \ref{cor:3.2} trivially hold.
The point of Theorem \ref{thm:3.1} is to capture two different phenomena simultaneously. The first one is a concentration of lattice points in a small annulus near the boundary of the ball. The second one is the fact that huge proportion of points in the ball has "almost all" coordinates not bigger than $K$ (if the radius is sufficiently small in terms of $K$ and $d$). Exact powers of $d$ on the right-hand sides of \ref{eq:3.1} and \ref{eq:3.2} are not that important, they can be arbitrarily increased by taking bigger $C$ and $a$.
We expect the statements of Theorem \ref{thm:3.1} and Corollary \ref{cor:3.2} to be optimal, in the sense that conclusions should fail for $\varepsilon=0$. 
\par

Our main tool to prove Theorem \ref{thm:3.1} and Corollary \ref{cor:3.2} will be a formula for the number of lattice points in balls and spheres with small radii from Theorem \ref{thm:3.4} below. This formula will be expressed in terms of an auxiliary function $h$, which is a case of the Jacobi theta function.
\begin{Defn}
    \label{def:3.3}
We define $h : \{ z \in \mathbb{C}: |z|<1 \} \to \mathbb{C}$ by the formula
\[
h(z)= \sum_{k \in \Z} z^{k^2}=1+2 \sum_{k=1}^{\infty} z^{k^2}.
\]
\end{Defn}
\noindent Notice that for any $d \in \N$ we have
\[
h(z)^d= \sum_{n=0}^{\infty} |\sqrt{n}S | z^n
\]
and 
\[
\frac{h(z)^d}{1-z}= \sum_{n=0}^\infty \Big( \sum_{k=0}^n \big|\sqrt{k}S \big| \Big)z^n= \sum_{n=0}^{\infty} \big|\sqrt{n}B\big| z^n.
\]
Thus, due to Cauchy's formula for any $r\in (0,1)$ we have 
\begin{equation}
\label{eq:Cauchyball}
|\sqrt{n}S|= \frac{1}{2 \pi i} \oint_{|z|=r} \frac{h(z)^d}{z^{n+1}} dz,\qquad|\sqrt{n}B|= \frac{1}{2 \pi i} \oint_{|z|=r} \frac{h(z)^d}{z^{n+1}}\frac{1}{1-z} dz.
\end{equation}

Recall that  
\[ \alpha=\frac{n}{d}\] and note that since we are in the regime $n\lesssim d$ we have $\alpha\lesssim 1.$ We are now ready to present our most important tool in this section. Theorem \ref{thm:3.4} gives a quantitative and uniform (dimension-free) asypmtotic formula for the number of lattice points in balls and spheres within the regime $1 \leq n \leq c d.$ The authors are clueless, whether a result of this form was already established in the literature, however the case $n=cd$ for fixed $c>0$ as $d \to \infty$ was studied in \cite{HRou} and \cite{MO}.
\begin{Thm}
\label{thm:3.4} There exists $c \in (0,1)$  such that for  $n,d \in \N$ satisfying  $1 \leq n \leq c d$ we have
\begin{equation}
\label{eq:{thm:3.4}:hform}
|\sqrt{n}B| \approx \frac{h(r)^d}{r^n} \frac{1}{\sqrt{n}} \approx |\sqrt{n}S|,
\end{equation}
where $r \in (0,1)$ is the unique number satisfying $r \frac{h'(r)}{h(r)}=\alpha$ and the implicit constants do not depend on $n$ and $d$. Moreover there are coefficients $b_1,b_2,...$ such that for any $K \in \N$, $n,d \in \N$ satisfying $1 \leq n \leq cd$ we have
\begin{equation}
\label{eq:{thm:3.4}:expform}
|\sqrt{n}B| \approx 2^n e^n \alpha^{-n}  \frac{1}{\sqrt{n}
}\exp \Big( \sum_{k=1}^{K} b_k n\alpha^k + O_K(n \alpha^{K+1}) \Big) \approx |\sqrt{n}S|.
\end{equation}
Numbers $b_1,b_2,...$ are coefficients of a certain power series, we have in particular \[b_1=-\frac{1}{2},\quad b_2=-\frac{1}{6},\quad b_3=\frac{1}{24}.\]
\end{Thm}
\begin{Rem} The error term in \eqref{eq:{thm:3.4}:expform} is in the exponent, which might be somewhat disappointing. However, if we fix $K$ and restrict our attention to $n,d$ such that $n \alpha^{K+1} \leq 1$ (that is $n \leq d^{\frac{K+1}{K+2}})$, then we obtain a reasonable approximation for $|\sqrt{n}B|$ and $|\sqrt{n}S|$. 
\end{Rem}

We move towards the proof of Theorem \ref{thm:3.4}. It is not elaborate and uses saddle point method, which was highlighted for $n\approx d$ at the end of \cite{MO}. We start with a lemma that is essentially contained in \cite{MO}.

\begin{Lem}
    \label{lem:3.5}
    For any $0<\alpha<1$ there exists unique $r \in (0,1)$ such that 
\[
r \frac{h'(r)}{h(r)}=\alpha.
\]
\end{Lem}
\begin{proof}
We claim that for any $z \in (0,1)$ we have
    \begin{equation*}
    \frac{d}{dz}\Big(z\frac{h'(z)}{h(z)}\Big)>0.
    \end{equation*}
This inequality was justified in \cite[Lemma 2]{MO}, however, we include the proof for completeness. 
    Note that
    \[
       \frac{d}{dz}\Big(z\frac{h'(z)}{h(z)}\Big)= \frac{h'(z)}{h(z)}+ \frac{zh''(z)}{h(z)}- \frac{zh'(z)^2}{h(z)^2},
    \]
    thus is sufficient to prove that
    \[
    h(z)\big(zh'(z)+z^2h''(z)\big)> z^2h'(z)^2,\qquad z\in (0,1).
    \]
    The above inequality is equivalent to
    \[
    \Big( \sum_{k \in \Z} z^{k^2} \Big) \Big( \sum_{ k \in \Z} k^4 z^{k^2} \Big) > \Big( \sum_{k \in \Z} k^2z^{k^2} \Big)^2,
    \]
    and follows from Cauchy-Schwarz inequality. Equality holds if and only if there is $C$ such that $Cz^{k^2}=k^4z^{k^2}$ for all $k \in \Z$, so it never holds.

    Using the claim and noticing that $z \frac{h'(z)}{h(z)} \Big|_{z=0}=0$ and $ \lim_{z \to 1^-} z \frac{h'(z)}{h(z)}= \infty$ we complete the proof.
\end{proof}

To obtain some  understanding of such implicitly defined $r$ we use classical Lagrange-Bürmann inversion theorem (see \cite{Lagrange}, \cite{Burmann} for original articles, or see \cite[Theorem 5.4.2]{EnComb} for various different proofs). We apply this theorem to the equation $H(r)=\alpha,$ with $H(r)=rh'(r)/h(r)$ and this is possible since $H'(0)=2\neq 0.$ Note that by the implicit function theorem if $\alpha$ is sufficiently small, then $r$ is given by some convergent power series, hence to compute its coefficients it is sufficient to work on formal power series as in the above cited works.
\begin{Prop}
 \label{thm:3.6}
There is $c_1>0$ such that for all $|\alpha|\le c_1$ the equation
\[
z \frac{h'(z)}{h(z)}= \alpha
\]
has a unique solution $r$ satisfying $r\in (0,1).$ The solution is given by the power series
\[
r= \sum_{k=1}^{\infty} a_k \alpha^k,
\]
with radius of convergence at least $2c_1$, where
\[
a_k= \frac{1}{k!} \Big\{  \frac{d^{k-1}}{dz^{k-1}} \Big(  \frac{h(z)}{h'(z)}\Big)^k \Big\}_{z=0},
\]
in particular \[a_1=a_2=a_3=\frac{1}{2},\quad  a_4= \frac{1}{4}.\]
\end{Prop}

Before proceeding to the proof of Theorem \ref{thm:3.4} we need to derive some consequences of Proposition \ref{thm:3.6}. Let $\alpha=n/d \leq c_1$, where $c_1$ is from Proposition \ref{thm:3.6}. Then there exists some $r \in (0,1)$ such that 
    \[
    r \frac{h'(r)}{h(r)}= \alpha
    \]
and by Proposition \ref{thm:3.6} for each $K\in 
N$ we have
   \begin{equation*}
    r=\sum_{k=1}^{K} a_k \alpha^k + O_K(\alpha^{K+1}).
    \end{equation*}
    Since $a_1= \frac{1}{2} \neq 0$ for $\alpha$ small enough we have  \[\alpha/4<r<\alpha<1/4.\]  In what follows we consider the analytic function
   \begin{equation}
   \label{eq:fdefi}
    f(z):=\log(h(z))- \alpha \log(z)
    \end{equation}
    on a sufficiently small neighborhood of  $\{z \in \mathbb{C}: |z|=r, |\arg(z)| \leq \delta\},$ where $\delta$ is a small quantity to be fixed later. Notice that
    \[\exp(df(r))=\frac{h(r)^d}{r^n}\]
    and
    \[
    f'(r)=\frac{h'(r)}{h(r)}- \frac{\alpha}{r}=0.
    \]
    Denote $f^{(2)}(r)=\beta$, then
    \[
    \beta= \frac{\alpha}{r^2}+ \frac{d}{dr} \Big( \frac{h'(r)}{h(r)} \Big)= \frac{\alpha}{r^2}+O(1),
    \]
   and since $r \in ( \alpha/4, \alpha)$, for  small enough $\alpha$ we have $\beta \in ( \frac{1}{2\alpha}, \frac{5}{\alpha})$. By complex Taylor's theorem for any $z \in \mathbb{C}$  such that $|z|=r$, $|\arg(z)| \leq \delta$ we have 
    \[
    f(z)=f(r)+ \frac{\beta}{2}(z-r)^2 + \frac{1}{2} \int_{\wideparen{r,z}} (w-z)^2 f^{(3)}(w) dw,
    \]
    where $\wideparen{r,z}$ denotes the arc from $r$ to $z$ along the circle of radius $r$. Notice that 
    \[
    f^{(3)}(w)= -\frac{2\alpha}{w^3} + \frac{d^2}{dw^2} \Big( \frac{h'(w)}{h(w)} \Big).
    \]
    Thus for $|w|=r \approx \alpha$ we have
    \[
    |f^{(3)}(w)|\lesssim \frac{\alpha}{r^3} +1 \lesssim \frac{1}{\alpha^2}
    \]
    and hence
   \begin{equation}
   \label{eq:ftay}
     f(z)=f(r)+ \frac{\beta}{2}(z-r)^2 + O\big( \frac{1}{\alpha^2}|z-r|^3\big).
    \end{equation}

\par We are now ready to prove Theorem \ref{thm:3.4}.
\begin{proof}[Proof of Theorem \ref{thm:3.4}]
Observe first that, if $C>0$ is a constant and $n<C$, then \\ $r\approx \alpha \approx_C 1/d$ and thus 
    \[
\frac{h(r)^d}{r^n} \frac{1}{\sqrt{n}}\approx_C d^n.
    \]
    Furthermore, a simple combinatorial argument shows that 
    \[|\sqrt{n}B| \approx_C d^n\approx_C |\sqrt{n}S|.\]
    Indeed, on the one hand we have
    \[
    |\sqrt{n}B|\ge |\sqrt{n}S|\ge |\{-1,0,1\}^d\cap\sqrt{n}S|=2^n\binom{d}{n}\approx_C d^n.
    \]
    On the other hand using the inclusion $\sqrt{n}B\subseteq n B^1(d),$ where $B^1(d)$ denotes the $\ell^1$ ball \eqref{eq: Bqballs}, we obtain
    \[
    |\sqrt{n}B|\le |n B^1(d)|=\sum_{j=0}^n 2^j\binom{d}{j}\binom{n}{j}\lesssim_C d^n.
    \]
The formula for the number of lattice points in $n B^1(d)$ used above is easy to establish, see e.g.\ \cite[Lemma 2.2]{Ni1} for a proof. In summary, in the case $n<C$ we have
    \[
    |\sqrt{n}B|\approx  |\sqrt{n}S|\approx \frac{h(r)^d}{r^n} \frac{1}{\sqrt{n}}.
    \]
  which establishes both \eqref{eq:{thm:3.4}:hform} and  \eqref{eq:{thm:3.4}:expform}.
  
It remains to consider $n>C.$ We start with justifying \eqref{eq:{thm:3.4}:hform}. For this purpose we will estimate $|\sqrt{n}B|$ from above and $|\sqrt{n}S|$ from below. In the proof several steps will include taking $\alpha, r,\delta$ 'small enough' or $n$ 'large enough'. Each such steps means that we are taking $\alpha, r,\delta$ smaller than a universal constant or $n$ larger than a universal constant. The number of these steps will be finite and thus there will be universal constants $c>0$ and $C>0$ such that all the statements in the proof below hold for $\alpha<c,$ $r<c,$ $\delta<c$ and $n>C.$ Throughout the proof $c_2,c_3,c_4,c_5$ will denote non-negative universal constants.

\noindent {\bf 1) Estimate from above for $|\sqrt{n}B|$}. 
  Recalling \eqref{eq:Cauchyball} we have
    \[
    |\sqrt{n}B|= \frac{1}{2 \pi i } \oint_{|z|=r} \frac{h(z)^d}{z^{n+1}} \frac{1}{1-z} dz.
    \]
    Now take $\delta \in (0, \pi)$ (how small we will determine later, it will not depend on $n,d$). We divide the integral 
    \begin{equation} \label{eq:3.3}
   \frac{1}{2 \pi i } \int_{\substack{|z|=r, \\ |\arg(z)| \leq \delta }} \frac{h(z)^d}{z^{n+1}} \frac{1}{1-z} dz + \frac{1}{2 \pi i } \int_{\substack{|z|=r, \\ |arg(z)|> \delta}} \frac{h(z)^d}{z^{n+1}} \frac{1}{1-z} dz=:W_1+W_2.
    \end{equation}
    Above for any $z \in \mathbb{C} \setminus \{0 \}$ we take $\arg(z) \in (-\pi, \pi]$. We will treat two integrals in \eqref{eq:3.3} separately. 
    \par  We start with the first integral $W_1$.
   Using \eqref{eq:ftay} we obtain
  \begin{align*}
W_1&= \frac{1}{2 \pi i } \int_{\substack{|z|=r, \\ |\arg(z)| \leq \delta }} e^{df(z)} \frac{1}{z(1-z)} dz\\
    &= \frac{1}{2 \pi i } \int_{\substack{|z|=r, \\ |\arg(z)| \leq \delta }} \exp\Big( d f(r) + \frac{d\beta}{2}(z-r)^2+O\big( \frac{d}{\alpha^2} |z-r|^3\big)\Big) \frac{1}{z(1-z)} dz
   \\ 
    &= 
    \frac{h(r)^d}{r^{n}} \frac{1}{2 \pi} \int_{-\delta}^{\delta} \exp\Big( \frac{d\beta r^2}{2}(1- e^{i\theta})^2+O\big( \frac{d r^3}{\alpha^2} |1- e^{i\theta}|^3 \big) \Big) \frac{1}{(1-re^{i \theta})} d \theta
    \\
    &=\frac{h(r)^d}{r^{n}} \frac{1}{2 \pi} \int_{-\delta}^{\delta} \exp\Big( \frac{d\beta r^2}{2}(-\theta^2+O(|\theta|^3))+O\big(d \alpha |\theta|^3 \big) \Big) \frac{1}{(1-re^{i \theta})} d \theta
  \\
    &=\frac{h(r)^d}{r^{n}}\frac{1}{2 \pi} \int_{-\delta}^{\delta} \exp\Big( -\theta^2 \frac{d\beta r^2}{2}\big(1+ O(\delta)\big) \Big) \big(1+ O(r) \big) d \theta.
   \end{align*}
    We fix $\delta \in (0,1)$ and take $\alpha$ small enough so that both expressions $O(\delta)$ and $O(r)$ above have absolute values smaller than $1/2$. Then we have
\[
    |W_1| \leq  \frac{h(r)^d}{r^{n}}\frac{1}{2 \pi}  \cdot \frac{3}{2}\int_{-\delta}^{\delta} \exp\Big( -\theta^2 \frac{d\beta r^2}{4} \Big)  d \theta \lesssim \frac{h(r)^d}{r^{n}} \cdot \frac{1}{\sqrt{d \beta r^2}} \int_{- \infty}^{\infty} e^{- x^2/2} dx \approx \frac{h(r)^d}{r^{n}} \frac{1}{\sqrt{n}},
    \]
    last inequality holds because $d \beta r^2 \in ( \frac{d\alpha}{32}, 5 d \alpha)= (\frac{n}{32}, 5n)$.

\par Now we will consider $W_2$, the second integral of $\eqref{eq:3.3}$.
    Take any $z \in \mathbb{C}$ such that $|z|=r$ and $|\arg(z)|=\theta \in (\delta, \pi]$. Then, using the inequality $\cos(x) \leq 1- x^2/20$, $x \in [0, \pi],$ we obtain
\begin{align*}
    |1+2z|^2&=(1+2r \cos(\theta))^2+ (2r\sin \theta)^2= 1+4r^2+4r \cos \theta  \leq  (1+2r)^2- \frac{r\theta^2}{5} \le(1+2r)^2- \frac{r\delta^2}{5}\\
    &=1+2r(2-\delta^2/10)+4r^2\le (1+2(1-\delta^2/40)r)^2,
\end{align*}
provided that $\alpha$ and hence $r$ is small enough. Denoting $c_2=\delta^2/40$ we obtain
   \[
\frac{|h(z)|}{h(r)}\le \frac{(1+2(1-c_2)r+O(r^4))}{1+2r}\le 1-\frac{c_2 r+ O(r^4)}{1+2r}\le \exp(-c_2r/2)
\]
for small enough $r.$ This implies
    \begin{equation}
\label{eq:W2bound}
   |W_2| \lesssim \frac{h(r)^d}{r^{n}}   \exp (-dr c_2/2)\lesssim \frac{1}{\sqrt{n}}\frac{h(r)^d}{r^{n}},
\end{equation}
where the last inequality holds for large enough $n$ since $dr\approx d\alpha=n.$ Recalling \eqref{eq:3.3} we conclude that
\begin{equation}
\label{eq:snBabo}
|\sqrt{n}B|\lesssim \frac{1}{\sqrt{n}}\frac{h(r)^d}{r^{n}}.
\end{equation}

\noindent {\bf 2) Estimate from below for $|\sqrt{n}S|.$} This time we have
 \[
    |\sqrt{n}S|= \frac{1}{2 \pi i } \oint_{|z|=r} \frac{h(z)^d}{z^{n+1}}dz
    \]
and split
\begin{equation*} 
   \frac{1}{2 \pi i } \int_{\substack{|z|=r, \\ |\arg(z)| \leq \delta }} \frac{h(z)^d}{z^{n+1}} dz + \frac{1}{2 \pi i } \int_{\substack{|z|=r, \\ |arg(z)|> \delta}} \frac{h(z)^d}{z^{n+1}}  dz=:V_1+V_2.
    \end{equation*}
For further reference note that repeating the argument used in estimating $W_2,$ cf.\ \eqref{eq:W2bound}, it is easy to see that
\begin{equation}
\label{eq:V2ab}
|V_2|\lesssim  \exp (-c_3 n)\frac{h(r)^d}{r^{n}},   
\end{equation}
where $c_3$ is a constant that depends only on $\delta$. 

Considering $V_1$ we let
\[
\varphi(\theta)=\Ima (f(re^{i\theta})-f(r))=\Ima f(re^{i\theta}),\qquad \psi(\theta)=\Rea (f(re^{i\theta})-f(r)),
\]
with $f$ being defined in \eqref{eq:fdefi}.
Then, since $\exp(df(r))=h(r)^d r^{-n}$ we may write
\begin{equation}
\label{eq:ReaV1}
\Rea (V_1)\hspace{-0.1cm}= \hspace{-0.1cm}\Rea\left( \frac{1}{2 \pi i } \int_{\substack{|z|=r, \\ |\arg(z)| \leq \delta }} \hspace{-0.3cm}e^{df(z)} \frac{dz}{z}\right)\hspace{-0.1cm}=\hspace{-0.1cm}\frac{1}{2\pi} \frac{h(r)^d}{r^{n}}\hspace{-0.2cm}\int_{-\delta}^{\delta} \hspace{-0.2cm}\cos(d\varphi(\theta))\exp(d\psi(\theta))\,d\theta.
\end{equation}
Recalling \eqref{eq:ftay} and using the observation that $\Ima \big((r-re^{i\theta})^2 \big)=O(r^2\theta^3)$ we see that
\[
|d\varphi(\theta)|\lesssim d(\beta r^2 +r^3/\alpha^2)\theta^3 \lesssim n \theta^3,\qquad |\theta|\le \delta.
\]
Taking large enough $n>C>0$ we may absorb the implicit constant above and get
\[
|d\varphi(\theta)|\le n^{-1/4}(\sqrt{n}\theta)^{3},\qquad |\theta|\le \delta.
\]
Furthermore, \eqref{eq:ftay} also implies that for small enough $\delta>0$ we have
\[
-c_4n\theta^2\le d\psi(\theta) \le -c_5 n \theta^2,\qquad |\theta|\le \delta,
\]
where $c_4$ and $c_5$ are universal constants. Changing variables $s=\sqrt{n}\theta$ in \eqref{eq:ReaV1} we get
\[
\Rea (V_1)=\frac{1}{2\pi\sqrt{n}} \frac{h(r)^d}{r^{n}}\int_{-\delta\sqrt{n}}^{\delta\sqrt{n}} \cos\Big(d\varphi\Big(\frac{s}{\sqrt{n}}\Big)\Big)\exp\Big(d\psi\Big(\frac{s}{\sqrt{n}}\Big)\Big)\,ds.
\]
Since for $|s|\le  \delta\sqrt{n}$ it holds
\[
\Big|d\varphi\Big(\frac{s}{\sqrt{n}}\Big)\Big|\le n^{-1/4} s^3,\qquad -c_4 s^2\le d\psi(s/\sqrt{n})\le -c_5 s^2,
\]
we see that for each $C>0$ and each $n$ such that $\delta\sqrt{n}>C$ and $n^{-1/4}C^3 \le \pi /3$ we have
\begin{align*}
&\int_{-\delta\sqrt{n}}^{\delta\sqrt{n}} \cos(d\varphi(s/\sqrt{n}))\exp(d\psi(s/\sqrt{n}))\\
&\ge \int_{-C}^C   \cos(d\varphi(s/\sqrt{n}))\exp(d\psi(s/\sqrt{n}))\,ds-\int_{\delta\sqrt{n}>|s|>C} \exp(-c_5 s^2)\,ds\\
& \ge \frac{1}{2}\int_{-C}^C \exp(-c_4 s^2)\,ds - \int_{|s|>C} \exp(-c_5 s^2)\,ds\gtrsim 1,
\end{align*}
where the last inequality holds, provided we take large enough $C>0.$ In summary we have justified that
\[
\Rea (V_1)\gtrsim \frac{1}{\sqrt{n}}\frac{h(r)^d}{r^n} 
\]
and together with \eqref{eq:V2ab} we obtain for large enough $n$
\begin{equation}
\label{eq:snSbel}
|\sqrt{n}S|\gtrsim \frac{1}{\sqrt{n}}\frac{h(r)^d}{r^n} .
\end{equation}

Collecting \eqref{eq:snSbel} and \eqref{eq:snBabo} we obtain \eqref{eq:{thm:3.4}:hform}.

\noindent {\bf 3) Verification of \eqref{eq:{thm:3.4}:expform}.} It remains to establish \eqref{eq:{thm:3.4}:expform}. We write
    \begin{equation}
    \label{eq:thm:3.4expformproof}
    \begin{split}
    \frac{h(r)^d}{r^n}& \hspace{-0.1cm}= \hspace{-0.1cm}\exp \Big( d\log(h(r))\hspace{-0.1cm}-\hspace{-0.1cm} n \log(r) \Big)\hspace{-0.1cm}=\hspace{-0.1cm} \exp\Big(d \log\big(1\hspace{-0.1cm}+\hspace{-0.1cm}2\sum_{k=1}^{\infty} r^{k^2}\big)\hspace{-0.1cm} - \hspace{-0.1cm}n\log(r) \Big)
  \\
    &= 2^n \alpha^{-n} \exp\Big(n\Big(\alpha^{-1} \log\big(1+2\sum_{k=1}^{\infty} r^{k^2}\big) - \log(2r/\alpha)\Big) \Big)
    \end{split}
\end{equation}
and use Proposition \ref{thm:3.6} to express the quantity under the last exponential as a function of $\alpha,$ namely
\begin{equation*}
   \alpha^{-1} \log\big(1+2\sum_{k=1}^{\infty} r^{k^2}\big) - \log(2r/\alpha) = \alpha^{-1} \log\Big(1+2\sum_{k=1}^{\infty} \Big( \sum_{j=1}^{\infty} a_j \alpha^j \Big)^{k^2}\Big) - \log\Big(1+2\sum_{j=2}^{\infty} a_j \alpha^{j-1}\Big).
  \end{equation*}
Note that
\begin{align*}
\sum_{k=1}^{\infty} \Big( \sum_{j=1}^{\infty} a_j \alpha^j \Big)^{k^2}&=\sum_{k=1}^{K+1} \Big( \sum_{j=1}^{\infty} a_j \alpha^j \Big)^{k^2}+O_K(\alpha^{(K+1)^2})=\sum_{k=1}^{K}\tilde{a}_j\alpha^{j}+O_K(\alpha^{K+1}),\\
\sum_{j=2}^{\infty} a_j \alpha^{j-1}&=\sum_{j=2}^{K+1} a_j \alpha^{j-1}+O_K(\alpha^{K+1}),
\end{align*}
for some coefficients $\tilde{a}_j.$ Hence, using the power series expansion of the logarithm we reach  
  \[
 \alpha^{-1} \log\big(1+2\sum_{k=1}^{\infty} r^{k^2}\big) - \log(2r/\alpha) = \sum_{k=0}^{K} b_k \alpha^k + O_K( \alpha^{K+1}) 
  \]
    for some coefficients $b_k$.

    An elementary though tedious calculation shows that $b_0=1$, $b_1=-\frac{1}{2}$, $b_2=-\frac{1}{6}$, $b_3=\frac{1}{24}.$ Finally, coming back to \eqref{eq:thm:3.4expformproof} 
   we see that \eqref{eq:{thm:3.4}:expform} holds.

    Thus we completed the proof of Theorem \ref{thm:3.4}.

\end{proof}
Before proving Theorem \ref{thm:3.1} and Corollary \ref{cor:3.2} we will  need  two easy consequences of Theorem \ref{thm:3.4}. The constant $c$ in Lemma \ref{lem:3.8} and Corollary \ref{cor:3.9} below is the one from Theorem \ref{thm:3.4}. 
\begin{Lem}
    \label{lem:3.8} Fix $K \in \N$, $\delta>0$. Then there is a constant $L(K,\delta):=L>0$, such that for any numbers $n,d,m \in \N$ satisfying $(1+\delta)m \leq n \leq d^{1-\frac{1}{(K+1)^2}}$, $d \geq c^{-(K+1)^2}$ we have
    \[
    |\sqrt{n-m}B| \lesssim_{K,\delta} L^m\alpha^m |\sqrt{n}B|.
    \]
\end{Lem}
\begin{proof}
    By Theorem \ref{thm:3.4} we have 
    \[
    |\sqrt{n}B| \approx  2^n e^n \alpha^{-n}  \frac{1}{\sqrt{n}
}\exp \Big( \sum_{k=1}^{(K+1)^2-1} b_k n\alpha^k + O_K(n \alpha^{(K+1)^2}) \Big)
    \]
    and our assumptions imply that $n \alpha^{(K+1)^2} \leq 1$. Let 
    $\tilde{\alpha}= \frac{n-m}{d}=\alpha-\frac{m}{d}=\alpha(1-\frac{m}{n})$, then we get
    \[
    \tilde{\alpha}^{-n+m}= \alpha^{-n+m}\Big(1-\frac{m}{n}\Big)^{-n+m}= \alpha^{-n+m}\Big(1+ \frac{m}{n-m}\Big)^{n-m} \leq \alpha^{-n+m} e^m.
    \]
    Moreover for any $k \leq (K+1)^2-1$ we have
    \[
    \tilde{\alpha}^k=\alpha^k\Big( 1- \frac{m}{n} \Big)^k= \alpha^k +O_K\Big(\frac{m}{n}\Big).
    \]
    Notice  also that $n \geq (1+\delta)m$ implies 
    \[
    \frac{1}{\sqrt{n-m}} \lesssim_{\delta} \frac{1}{\sqrt{n}}.
    \]
    Thus by Theorem \ref{thm:3.4} we obtain
    \begin{align*}
    &|\sqrt{n-m}B| \approx 2^{n-m} e^{n-m} \tilde{\alpha}^{-n+m}  \frac{1}{\sqrt{n-m}
} \\
&\cdot\exp \Big( \sum_{k=1}^{(K+1)^2-1} b_k (n-m)\tilde{\alpha}^k + O_K\Big((n-m) \tilde{\alpha}^{(K+1)^2}\Big) \Big)
 \\
    &\lesssim_{\delta}  2^{n-m} e^{n} \alpha^{-n+m}  \frac{1}{\sqrt{n}
}\exp \Big( \sum_{k=1}^{(K+1)^2-1} b_k n\alpha^k + O_K(m) \Big)
    \lesssim  (L\alpha)^m |\sqrt{n}B|,
    \end{align*}
    for some constant $L=L(K,\delta)$ and the proof is completed.
\end{proof}
\begin{Cor}
\label{cor:3.9}   
Fix $K \in \N$. Then for any $n,d \in \N$ satisfying $32 K^2 \leq n \leq d^{1-\frac{1}{(K+1)^2}}$, $d \geq c^{-(K+1)^2}$ we have
\[
\big| \{x \in \sqrt{n}B: (\exists i \in [d])( |x_i| \geq 4K) \} \big| \lesssim_{K} \frac{1}{d} |\sqrt{n}B|.
\]
\end{Cor}
\begin{proof}
Denote $n'=n-16K^2$ and notice that
\begin{align*}
&\big| \{x \in \sqrt{n}B: (\exists i \in [d])(  |x_i| \geq 4K) \} \big| \le d \big| \{x \in \sqrt{n}B: (\sqrt{d}\ge |x_1| \geq 4K) \} \big|\\
&\le 2d^{3/2} |\sqrt{n'}B^{d-1}|\leq 2d^{3/2} |\sqrt{n'} B|,
\end{align*}
above $|\sqrt{n'}B^{d-1}|$ denotes the number of lattice points in $\Z^{d-1}$ inside the ball $\sqrt{n'} B^{d-1}.$
Hence, by Lemma \ref{lem:3.8} with $\delta=1$, $m=16K^2$ we obtain
\begin{align*}
&\big| \{x \in \sqrt{n}B \hspace{-0.05cm}: \hspace{-0.05cm}(\exists i \in [d])( |x_i| \geq 4K) \} \big| \hspace{-0.05cm}\le \hspace{-0.05cm} 2d^{3/2} |\sqrt{n'} B| \hspace{-0.05cm} \lesssim \hspace{-0.05cm} L^{16K^2} d^{3/2} \alpha^{16K^2} |\sqrt{n}B|  \\
&\lesssim_{K} d^{3/2- \frac{16K^2}{(K+1)^2}} |\sqrt{n}B|\lesssim_{K} \frac{1}{d} |\sqrt{n}B|.
\end{align*}
In the third inequality above we used the fact that
$ \alpha \leq d^{\frac{-1}{(K+1)^2}}$ for
$n \leq d^{1-\frac{1}{(K+1)^2}}.$
\end{proof}

We are now ready to prove Theorem \ref{thm:3.1}. In the proof we shall use the following subset of $\sqrt{n}B$.
\begin{Defn} \label{def:3.7}
For any $K,n,d \in \N$ we let 
\[
A_K(n)= \{ x \in \{-K,...,K\}^d : \sum_{i=1}^d x_i^2 \leq n \}.
\]
\end{Defn}
\begin{proof}[Proof of Theorem \ref{thm:3.1}]
Fix $K \in \N$, $\varepsilon>0$. Consider any $n,d \in \N$ such that 
$32K^2 \leq n \leq  d^{1-\frac{1+\varepsilon}{(K+1)^2}}$, $d \geq c^{-(K+1)^2}$. By Corollary \ref{cor:3.9} we have
\[
|\sqrt{n}B \setminus A_{4K}(n)| \lesssim_K \frac{1}{d} |\sqrt{n}B|.
\]
Thus it remains to show that 
\begin{equation} \label{eq:3.4}
    \big|\{ x \in A_{4K}(n): \sum_{\substack{i=1, \\ |x_i| \leq K}}^d x_i^2 \leq n-a \}\big| \lesssim_{K,\varepsilon} \frac{1}{d} |\sqrt{n}B|,
\end{equation}
for properly chosen large $a$ (which will be independent of $n$ and $d$). Let 
\[
E_m= \{ x \in A_{4K}(n): \sum_{\substack{i=1, \\ |x_i| \leq K}}^d x_i^2=n-m \},
\]
where $n\ge m \geq a.$ Note that \eqref{eq:3.4} will be established once we show that 
\begin{equation}
\label{eq:3.4'}
|E_m| \lesssim_{K, \varepsilon} \frac{1}{d^2}|\sqrt{n}B|.
\end{equation}

We claim that 
\begin{equation}\label{eq:3.5}
|E_m|  \leq 2\cdot (6K)^{\frac{m}{(K+1)^2}} d^{\frac{m}{(K+1)^2}}|\sqrt{n-m}B|.
\end{equation}
Indeed, denoting $m'=\lfloor \frac{m}{(K+1)^2} \rfloor$ we see that if $x\in E_m$ then
\[
 \sum_{\substack{i=1, \\ K<|x_i| \le 4K}}^d |x_i|^2 \leq m,
\]
and hence
\[
\# \{ i \in [d]: K<|x_i|\le 4K\} \leq m'.
\]
Therefore we obtain
\begin{align*}
 &|E_m|\leq \sum_{j=0}^{m'} |\{x\in E_m\colon \# \{ i \in [d]: K<|x_i|\le 4K\}=j\} |\le    \sum_{j=0}^{m'}{d\choose j}(6K)^j|A_K(n-m)|\\
 &\le   (6K)^{m'}\bigg(\sum_{j=0}^{m'}d^j\bigg)|A_K(n-m)|\le (6K)^{m'} d^{m'}\big(1+\frac{m'}{d}\big)|A_K(n-m)|\leq 2\cdot  (6K)^{m'} d^{m'}|A_K(n-m)|,
\end{align*}
where in the last inequality we used the fact that $m'\le d,$ and \eqref{eq:3.5} follows.

We now come back to the proof of \eqref{eq:3.4'}. We fix $\delta= \varepsilon/2$ and consider two cases separately.
\par
\textbf{Case 1) $a \leq m \leq \frac{n}{1+\delta}$}.
Using \eqref{eq:3.5} and Lemma \ref{lem:3.8} there is a constant $L=L(K,\varepsilon) \geq 1$ depending only on $K$ and $\varepsilon$ such that
\begin{align*}
|E_m| \hspace{-0.05cm}&\leq \hspace{-0.05cm}2\cdot  (6K)^{\frac{m}{(K+1)^2}} d^{\frac{m}{(K+1)^2}}|\sqrt{n-m}B| \hspace{-0.05cm}\lesssim \hspace{-0.05cm}L^m (6K)^{\frac{m}{(K+1)^2}} d^{\frac{m}{(K+1)^2}} \alpha^m | \sqrt{n}B|
\\
& \lesssim  L^m (6K)^{\frac{m}{(K+1)^2}} d^{\frac{m}{(K+1)^2}} d^{- \frac{m(1+\varepsilon)}{(K+1)^2}} | \sqrt{n}B| =  \Big(\frac{6KL^{(K+1)^2}}{d^{\varepsilon} } \Big)^{ \frac{m}{(K+1)^2}}| \sqrt{n}B|.
\end{align*}
If $d \geq (6KL^{(K+1)^2})^{2/\varepsilon} $, then we obtain
\[
|E_m|\lesssim d^{-\frac{\varepsilon m}{2(K+1)^2}} |\sqrt{n}B| \leq d^{- \frac{\varepsilon a}{2(K+1)^2}} |\sqrt{n}B|\leq \frac{1}{d^2} |\sqrt{n}B|,
\] 
where the last inequality holds if we choose $a \geq  \frac{4(K+1)^2}{\varepsilon}.$
\par \textbf{Case 2) $\frac{n}{1+\delta} \leq m \leq n$}.
Using  \eqref{eq:3.5} and Lemma \ref{lem:3.8} with $ m'=\lfloor \frac{n}{(1+\delta)} \rfloor$ in place of $m$ we get
\begin{align*}
&|E_m| \leq 2 \hspace{-0.05cm}\cdot \hspace{-0.05cm}  (6K)^{\frac{m}{(K+1)^2}} d^{\frac{m}{(K+1)^2}}|\sqrt{n-m}B|\le 2 \hspace{-0.05cm} \cdot \hspace{-0.05cm} (6K)^{\frac{m}{(K+1)^2}} d^{\frac{m}{(K+1)^2}} |\sqrt{n-m'}B|
\\ 
&\lesssim L^{m'}   (6K)^{\frac{m}{(K+1)^2}} d^{\frac{m}{(K+1)^2}} \alpha^{m' }|\sqrt{n}B| \leq  L^{\frac{n}{1+\delta}}(6K)^{\frac{n}{(K+1)^2}} d^{\frac{n}{(K+1)^2}}  \alpha^{\frac{n}{(1+\delta) } } d |\sqrt{n}B|
\\
&\lesssim \Big(6K L ^{(K+1)^2} \Big)^{\frac{n}{(K+1)^2}} d^{\frac{n}{(K+1)^2}} d^{- \frac{n}{(1+\delta)(K+1)^2} - \frac{\varepsilon n}{(1+\delta)(K+1)^2}}  d|\sqrt{n}B|
\\
&=  \Big(6K L ^{(K+1)^2} \Big)^{\frac{n}{(K+1)^2}} d^{ \frac{n(\delta- \varepsilon)}{(1+\delta)(K+1)^2}} d|\sqrt{n}B|.
\end{align*}
Since $\delta= \varepsilon/2$ we have
\[
d^{ \frac{n(\delta- \varepsilon)}{(1+\delta)(K+1)^2}}=d^{- \frac{ \varepsilon n}{(2+\varepsilon)(K+1)^2}},
\]
and thus for 
\[
d \geq  \Big(6K L ^{(K+1)^2} \Big)^{2+ \frac{4}{\varepsilon}}
\]
we get
\[
\Big(6K L ^{(K+1)^2} \Big)^{\frac{n}{(K+1)^2}} d^{ \frac{n(\delta- \varepsilon)}{(1+\delta)(K+1)^2}} d|\sqrt{n}B| \leq d^{- \frac{ \varepsilon n}{(4+2\varepsilon)(K+1)^2}} d|\sqrt{n}B|. 
\]
Hence, for $n \geq (\frac{12}{\varepsilon}+6)(K+1)^2$
also in case 2) we obtain $|E_m|\lesssim \frac{1}{d^2} |\sqrt{n}B|.$

Summarizing cases 1) and 2), we have proved that \eqref{eq:3.4'} and thus also \eqref{eq:3.4} hold provided we take $a \geq  \frac{4(K+1)^2}{\varepsilon} $, $d\ge C_1(K,\varepsilon)$ and $n\ge C_2(K,\varepsilon).$ Let $a= \lceil \max\big(\frac{4(K+1)^2}{\varepsilon}, C_2(K,\varepsilon)\big) \rceil+1$, then we have established Theorem \ref{thm:3.1} for  $d \ge C_1(K,\varepsilon)$ and any $n \leq d^{1- \frac{1+\varepsilon}{(K+1)^2}}$. In the remaining case, when $d<C_1(K,\varepsilon)$ the statement of Theorem \ref{thm:3.1} trivially holds.
\end{proof}
We can now easily deduce Corollary \ref{cor:3.2}.
\begin{proof}[Proof of Corollary \ref{cor:3.2}] 
Fix $K \in \N$, $\varepsilon>0$ and let $a,$ $C$ be the constants from Theorem $\ref{thm:3.1}$. Then by Theorem \ref{thm:3.1} we have
\[
|\{ x \in  \sqrt{n}S : \sum_{\substack{i=1, \\ |x_i| \leq K}}^d |x_i|^2 \leq n-a \}|\]
\[
\leq |\{ x \in  \sqrt{n}B : 
\sum_{\substack{i=1, \\ |x_i| \leq K}}^d |x_i|^2 \leq n-a \}| \leq   \frac{C}{d} |\sqrt{n}B|.
\]
Recall that $n/d \leq d^{-\frac{1+\varepsilon}{(K+1)^2}}.$ Thus, in the case $d\ge c^{-\frac{(K+1)^2}{1+\varepsilon}}$ we have $n/d\le c$ and an application of Theorem \ref{thm:3.4} gives
\[
|\{ x \in  \sqrt{n}S : \sum_{\substack{i=1, \\ |x_i| \leq K}}^d |x_i|^2 \leq n-a \}| \le  \frac{C}{d} |\sqrt{n}B|\approx \frac{C}{d} |\sqrt{n}S|.
\]
In the remaining case $d<c^{-\frac{(K+1)^2}{1+\varepsilon}}$ Corollary \ref{cor:3.2} trivially holds.
\end{proof}
We finish this section with the proof of Corollary \ref{cor:1.4}, which was mentioned in the introduction. We also introduce another consequence of Theorem \ref{thm:3.4} which is given in Corollary \ref{rem:3.10}. These results are not necessary for the proofs of Theorems \ref{thm:ball} or \ref{thm:sphere}, however they are insightful on their own.
\begin{proof}[Proof of Corollary \ref{cor:1.4}]
By \eqref{eq:{thm:3.4}:expform} from Theorem \ref{thm:3.4} for $n,d \in \N$ satisfying $1 \leq n \leq cd$ we have
\[
|\sqrt{n}B|\approx |\sqrt{n}S| \approx 2^n e^n \alpha^{-n}  \frac{1}{\sqrt{n}
}\exp \Big( -\frac{n \alpha}{2}-\frac{n \alpha^2}{6}+\frac{n \alpha^3}{24}+ O(n \alpha^{4}) \Big).
\]
By Stirling's formula $n!\approx \sqrt{n} (n/e)^n$ we get
\[
\binom{d}{n}= \frac{d(d-1)...(d-n+1)}{n!} \approx \frac{d^n}{n^n} e^n \frac{1}{\sqrt{n}} \prod_{k=1}^{n-1}(1- \frac{k}{d}).
\]
Notice also that 
\begin{align*}
&\prod_{k=1}^{n-1}(1- \frac{k}{d}) = \exp\Big( \sum_{k=1}^{n-1} \log(1- \frac{k}{d}) \Big)= \exp\Big( \sum_{k=1}^{{n-1} }(- \frac{k}{d}- \frac{k^2}{2d^2}- \frac{k^3}{3d^3}) + O(\frac{n^5}{d^4}) \Big)
\\
&= \exp \Big( -\frac{n^2}{2d}- \frac{n^3}{6d^2} - \frac{n^4}{12d^3}+O(\frac{n^5}{d^4})+O(1) \Big)= \exp \Big( -\frac{n \alpha}{2} -\frac{n \alpha^2}{6}- \frac{n \alpha^3}{12}+ O(n \alpha^4)+O(1) \Big).
\end{align*}
Combining the above observations we get the desired conclusion.
\end{proof}

Finally we introduce a slightly weaker generalization of Corollary \ref{cor:1.4}. Note that in the case $K=1$ we have $A_1(n)=\sqrt{n} S\cap\{-1,0,1\}^d$ and Corollary \ref{rem:3.10} boils down to Corollary \ref{cor:1.4} except for the term inside the exponential. Corollary \ref{rem:3.10} is not a direct consequence of Theorem \ref{thm:3.4}, however its proof is similar and thus we only sketch it.
\begin{Cor}
    \label{rem:3.10}
    For any $K\in \N$ there is constant $c_K>0$ such that for any $n,d \in \N$ satisfying $n \leq c_Kd$ we have
    \[
    |\sqrt{n}B| \approx |\sqrt{n}S| \approx_K |A_K(n)| \exp\Big( O(n \alpha^{(K+1)^2-1}) \Big).
    \]
    In particular if $n \leq d^{1- \frac{1}{(K+1)^2}}$, then
    \[
    |\sqrt{n}B| \approx |\sqrt{n}S| \approx_K |A_K(n)|.
    \]
\end{Cor}
\begin{proof}[Proof (sketch)]
    Let 
    \[
    h_K(z)=1+2 \sum_{k=1}^K z^{k^2}.
    \]
    Repeating the argument from the proof of Lemma \ref{lem:3.5} one may prove that for each $\alpha\le 2/3$ the equation
    \[
    z \frac{h'_K(z)}{h_K(z)}=\alpha
    \]
    has exactly one solution $r_K\in (0,1).$ By applying Lagrange-Bürmann inversion theorem we have 
    \[
    r_K= \sum_{k=1}^{\infty} a_{k,K} \alpha^k,
    \]
    provided that $\alpha$ is small enough, where 
    \[
    a_{k,K}= \frac{1}{k!} \Big\{  \frac{d^{k-1}}{dz^{k-1}} \Big(  \frac{h_K(z)}{h_K'(z)}\Big)^k \Big\}_{z=0},
    \]
    in particular $a_{1,K}= \frac{1}{2}$. Computation on formal power series shows that 
    \begin{align*}
    \frac{h(z)}{h'(z)}&= \frac{h_K(z)+O(z^{(K+1)^2})}{h_K'(z)+O(z^{(K+1)^2-1})}= \frac{h_K(z)}{h'_K(z)} \cdot \frac{1+O\big( \frac{z^{(K+1)^2}}{h_K(z)} \big)}{1+O\big( \frac{z^{(K+1)^2-1}}{h_K'(z)} \big)}
    \\
    &=\frac{h_K(z)}{h'_K(z)} \cdot \frac{1+O(z^{(K+1)^2})}{1+O(z^{(K+1)^2-1})}= \frac{h_K(z)}{h'_K(z)}(1+ O(z^{(K+1)^2-1}))=\frac{h_K(z)}{h'_K(z)}+O(z^{(K+1)^2-1}).
    \end{align*}
    From the above we see that for any $k \leq (K+1)^2-1$ the coefficients $a_k$ from Proposition \ref{thm:3.6} satisfy
   \begin{align*}
    a_k&=  \frac{1}{k!} \Big\{  \frac{d^{k-1}}{dz^{k-1}} \Big(  \frac{h(z)}{h'(z)}\Big)^k \Big\}_{z=0}= \frac{1}{k!} \Big\{  \frac{d^{k-1}}{dz^{k-1}} \Big(  \frac{h_K(z)}{h_K'(z)} + O(z^{(K+1)^2-1})\Big)^k \Big\}_{z=0}
  \\
    &=\frac{1}{k!} \Big\{  \frac{d^{k-1}}{dz^{k-1}} \Big(  \Big(\frac{h_K(z)}{h_K'(z)} \Big)^k + O(z^{(K+1)^2-1})\Big) \Big\}_{z=0}= \frac{1}{k!} \Big\{  \frac{d^{k-1}}{dz^{k-1}} \Big(\frac{h_K(z)}{h_K'(z)} \Big)^k ) \Big\}_{z=0}=a_{k,K}.
   \end{align*}
    In particular we have 
    \[r=r_K+O(\alpha^{(K+1)^2}).\]
    
    Using the same method as in the proof of Theorem \ref{thm:3.4} based on writing
   \[|A_K(n)|= \frac{1}{2 \pi i} \oint_{|z|=r_K} \frac{h_K(z)^d}{z^{n+1}}\frac{1}{1-z} dz\]
    one can show that there is a constant $c_K$ such that for all $n,d \in \N$ satisfying $1 \leq n \leq c_K d$ we have
    \[
    |A_K(n)| \approx_K  \frac{h_K(r_K)^d}{r_K^n} \frac{1}{\sqrt{n}}.
    \]
Moreover, by Theorem \ref{thm:3.4} for any $n,d \in \N$ satisfying $n \leq cd$ we have
    \[
    |\sqrt{n}S| \approx |\sqrt{n}B| \approx  \frac{h(r)^d}{r^n} \frac{1}{\sqrt{n}}
    \]
    with $r\approx \alpha.$ 
    
    Previous calculations show that for $\alpha \leq \min(c_K,c)/2$ we have 
 \begin{align*}
    &|\sqrt{n}B| \approx  \frac{h(r)^d}{r^n} \frac{1}{\sqrt{n}} \hspace{-0.1cm}=\hspace{-0.1cm}\frac{\big(h_K(r)\hspace{-0.1cm}+ \hspace{-0.1cm}O(r^{(K+1)^2}) \big)^d}{r^n} \frac{1}{\sqrt{n}}\hspace{-0.1cm}=\hspace{-0.1cm}\frac{\big(h_K(r)\hspace{-0.1cm}+\hspace{-0.1cm} O(\alpha^{(K+1)^2}) \big)^d}{r^n} \frac{1}{\sqrt{n}}
   \\
    &= \frac{\big(h_K\big( r_K\hspace{-0.1cm}+\hspace{-0.1cm}O(\alpha^{(K+1)^2}) \big)\hspace{-0.1cm}+\hspace{-0.1cm} O(\alpha^{(K+1)^2}) \big)^d}{\big( r_K\hspace{-0.1cm}+\hspace{-0.1cm}O(\alpha^{(K+1)^2})\big)^n} \frac{1}{\sqrt{n}}\hspace{-0.1cm}= \hspace{-0.1cm}\frac{\big(h_K(r_K)\hspace{-0.1cm}+ \hspace{-0.1cm}O(\alpha^{(K+1)^2}) \big)^d}{\big( r_K\hspace{-0.1cm}+\hspace{-0.1cm}O(\alpha^{(K+1)^2})\big)^n} \frac{1}{\sqrt{n}}
   \\
    &= \hspace{-0.1cm}\frac{h_K(r_K)^d}{r_K^n} \cdot \frac{(1\hspace{-0.1cm}+\hspace{-0.1cm}O(\alpha^{(K+1)^2}))^d}{(1\hspace{-0.1cm}+\hspace{-0.1cm}O(\alpha^{(K+1)^2-1}))^n} \frac{1}{\sqrt{n}}\hspace{-0.1cm}= \hspace{-0.1cm}\frac{h_K(r_K)^d}{r_K^n} \frac{1}{\sqrt{n}} \exp\Big( O(n \alpha^{(K+1)^2-1}) \Big).
    \end{align*}
Hence for any $n,d \in \N$ satisfying $1 \leq n \leq c_K'd$ for some $c_K'>0$ we have 
    \[
    |\sqrt{n}S| \approx |\sqrt{n}B| \approx_K |A_K(n)| \exp \Big( O(n \alpha^{(K+1)^2-1}) \Big).
    \]
    This proves the first part of the lemma. 
    
    The second part follows from the formula above once we note that for  $n,d \in \N$ satisfying $1 \leq n \leq d^{1- \frac{1}{(K+1)^2}}$ and  $1 \leq n \leq c_K'd$ we have  \\ $O(n \alpha^{(K+1)^2-1})=O(1).$ Note that the opposite case $1 \leq n \leq d^{1- \frac{1}{(K+1)^2}}$ and $n>  c_K'd$ can only happen for a finite number of dimensions $d$ in which case Corollary \ref{rem:3.10} is trivial.
\end{proof}

\section{Consequences for the maximal function corresponding to discrete spheres} 
\label{sec:4}
Using the results from previous sections we now proceed towards the proof of  Theorem \ref{thm:sphere}. We start with a lemma from \cite{balls} and its immediate corollary.
\begin{Lem}[{\cite[Lemma 3.2]{balls}}]
 \label{lem:4.1}
 For every $d,n \in \N$, if $n \leq \frac{d}{25}$ and $n\geq k \geq2^9 \max(1,\frac{n^4}{d^3})$, then
 \[
|\{x \in \sqrt{n}B: |\{i \in [d]: x_i= \pm 1\}| \leq n-k\}| \leq 2^{-k+1}|\sqrt{n}B|.
 \]
\end{Lem}
\begin{Cor} \label{cor:4.2}
 For every $d,n \in \N$, if $n \leq d \min(c, \frac{1}{25})$, then
 \[
|\{x \in \sqrt{n}S: |\{i \in [d]: x_i= \pm 1\}| \leq n/2\}| \lesssim 2^{-n/2}|\sqrt{n}S|.
 \]   
\end{Cor}
\begin{proof}

 For $n < 2^{10}$ the statement is trivial, we assume that $n \geq 2^{10}$.
 Using Lemma \ref{lem:4.1} for $k= \lfloor \frac{n}{2} \rfloor$ and Theorem \ref{thm:3.4} we obtain
\begin{align*}
&|\{x \in \sqrt{n}S: |\{i \in [d]: x_i= \pm 1\}| \leq n/2\}| \leq |\{x \in \sqrt{n}B: |\{i \in [d]: x_i= \pm 1\}| \leq n/2\}|
\\
&\leq 2^{-\lfloor n/2 \rfloor +1} |\sqrt{n}B| \approx 2^{-n/2} |\sqrt{n}S|.
\end{align*}
\end{proof}
Now we state crucial proposition, which gives the connection between Theorem \ref{thm:1.5} and Theorem \ref{thm:sphere}.
\begin{Prop}
  \label{lem:4.3}
   For any $K \in \N, \varepsilon>0$ there exists $d_0\in \N$ such that for any $d \geq d_0$ and $f \in \ell^2(\Z^d)$ we have
   \[
   \Big\| \sup_{n \leq d^{1-\frac{1+\varepsilon}{(K+1)^2}}} |\mathcal{S}_{\sqrt{n}}f| \Big\|_{2} \lesssim_{K,\varepsilon} \sum_{i=0}^1 \Big\| \sup_{n_1,n_2,...,n_K \leq \frac{d}{2K}} |\mathcal{D}_{n_1,...,n_K}f_i| \Big\|_{2} + \|f\|_{2},
   \]
   where $f=f_0+f_1$ and
\[
\supp(\widehat{f_0}) \subseteq \{\xi \in \T^d: \|\xi \| \leq \| \xi + 1/2 \| \},
\]
\[
\supp(\widehat{f_1}) \subseteq \{\xi \in \T^d: \|\xi \| > \| \xi + 1/2 \| \}.
\]
\end{Prop}
\begin{proof}
Throughout the proof we write
\[
s_{\sqrt{n}}(\xi):=m_{\sqrt{n}}^{S}(\xi)=\frac{1}{|\sqrt{n}S|}\sum_{x\in \sqrt{n}S}e(x\cdot \xi),\qquad \xi \in \T^d
\] 
for the multiplier symbol of the operator $\mathcal S_{\sqrt{n}}.$ Fix $K \in \N,$  $\varepsilon>0$, and take $a$ from Corollary \ref{cor:3.2}. We decompose every $x \in \Z^d$ into $x=y(x)+z(x)$, where
\[
z(x)_j= \begin{cases}
    x_j, \ \ \text{if $|x_j| \geq K+1$} \\
    0, \ \ \text{otherwise}
\end{cases}
\]
and denote $Z(\sqrt{n}S)=\{z(x) \in \Z^d: x \in \sqrt{n}S\}.$ Then we split
\[
s_{\sqrt{n}}(\xi)=\tilde{s}_{\sqrt{n}}(\xi)+r_{\sqrt{n}}(\xi),
\]
where
\begin{align*}
\tilde{s}_{\sqrt{n}}(\xi)&=\frac{1}{|\sqrt{n}S|} \sum_{\substack{x \in \sqrt{n}S, \\ |y(x) |^2 \geq n-a, \\ |\{i: |x_i|=1\} | \geq n/2}} e(x \cdot \xi),
\\
r_{\sqrt{n}}(\xi)&=\frac{1}{|\sqrt{n}S|} \sum_{\substack{x \in \sqrt{n}S, \\ |y(x)|^2 < n-a, \\ |\{i: |x_i|=1\} | \ge n/2}} e(x \cdot \xi) + \frac{1}{|\sqrt{n}S|} \sum_{\substack{x \in \sqrt{n}S \\ |\{i: |x_i|=1\} | < n/2}} e(x \cdot \xi).
\end{align*}
Above and throughout the proof of Proposition \ref{lem:4.3} we denote by $|y(x)|$ and $|z|$ the Euclidean lengths of $y(x)$ and $z,$ respectively.

By Corollary \ref{cor:3.2} and Corollary \ref{cor:4.2}, if $ n \leq d^{1-\frac{1+\varepsilon}{(K+1)^2}}$ and \\ $d \geq \max(c^{-(K+1)^2}, (25K)^{(K+1)^2})$, then 
\begin{equation}
\label{eq:rsnbound}
|r_{\sqrt{n}}(\xi)| \lesssim_{K,\varepsilon} \frac{1}{d} + 2^{-n/2},
\end{equation}
uniformly in $\xi \in \T^d$. 

To treat the term $\tilde{s}_{\sqrt{n}}(\xi)$ we rewrite
\begin{align*}
&\tilde{s}_{\sqrt{n}}(\xi)= \frac{1}{|\sqrt{n}S|} \sum_{\substack{z \in Z(\sqrt{n}S), \\ |z|^2 \leq a}} e(z \cdot \xi) \hspace{-0.5cm}\sum_{\substack{y \in \{-K,...,K\}^d, \\
n-a \leq |y|^2 \leq n- |z|^2, 
\\ \supp(y) \cap \supp(z)= \emptyset, 
\\ |\{i: |y_i|=1 \}| \geq n/2
}} \hspace{-0.5cm}e(y \cdot \xi)
\\
&=\frac{1}{|\sqrt{n}S|} \sum_{l=0}^a \sum_{\substack{z \in Z(\sqrt{n}S),
\\ |\supp(z)|=l, \\ |z|^2 \le a
}} e(z \cdot \xi) \hspace{-0.5cm} \sum_{\substack{i_1,...,i_K \in [d], 
\\ n-a \leq \sum_{j=1}^K j^2i_j \leq n-|z|^2,
\\
i_1 \geq n/2
}} 
\hspace{-0.2cm} \sum_{\substack{y \in \{-K,...,K\}^d, 
\\ \supp(y) \cap \supp(z)= \emptyset, 
\\
(\forall j \in [K]) |\{i \in [d]: |y_i|=j\}|=i_j
}} \hspace{-0.5cm} e(y \cdot \xi).
\end{align*}
Now we define two new multipliers
\begin{align*}
\phi_{\sqrt{n},0}(\xi)&=\frac{1}{|\sqrt{n}S|} \sum_{l=0}^a \sum_{\substack{z \in Z(\sqrt{n}S),
\\ |\supp(z)|=l, \\ |z|^2 \le a
}} \hspace{+0.1cm} \sum_{\substack{i_1,...,i_K \in [d], 
\\ n-a \leq \sum_{j=1}^K j^2i_j \leq n-|z|^2,
\\
i_1 \geq n/2
}} 
\hspace{-0.2cm}\sum_{\substack{y \in \{-K,...,K\}^d, 
\\ \supp(y) \cap \supp(z)= \emptyset, 
\\
(\forall j \in [K]) |\{i \in [d]: |y_i|=j\}|=i_j
}} \hspace{-1cm}e(y \cdot \xi),
\\
\phi_{\sqrt{n},1}(\xi)&=\frac{1}{|\sqrt{n}S|} \sum_{l=0}^a \sum_{\substack{z \in Z(\sqrt{n}S),
\\ |\supp(z)|=l, \\ |z|^2 \le a
}} (-1)^{\sum_{i=1}^d z_i} \hspace{-1.0cm}\sum_{\substack{i_1,...,i_K \in [d], 
\\ n-a \leq \sum_{j=1}^K j^2i_j \leq n-|z|^2,
\\
i_1 \geq n/2
}} 
\hspace{-0.2cm}\sum_{\substack{y \in \{-K,...,K\}^d, 
\\ \supp(y) \cap \supp(z)= \emptyset, 
\\
(\forall j \in [K]) |\{i \in [d]: |y_i|=j\}|=i_j
}} \hspace{-1cm} e(y \cdot \xi).
\end{align*}
We claim that two crucial inequalities hold:
\begin{enumerate} \item If $\| \xi \| \leq \| \xi+ 1/2 \|$ then
\begin{equation} \label{eq:4.1}
|\tilde{s}_{\sqrt{n}}(\xi)-\phi_{\sqrt{n},0}(\xi)| \lesssim_{K,\varepsilon} \frac{1}{n}.
\end{equation}
\item  If $\| \xi \| \geq \| \xi+ 1/2 \|$ then
\begin{equation} \label{eq:4.2}
|\tilde{s}_{\sqrt{n}}(\xi)-\phi_{\sqrt{n},1}(\xi)| \lesssim_{K,\varepsilon} \frac{1}{n}.
\end{equation}
\end{enumerate}

We start with the proof of  \eqref{eq:4.1} and assume that $\|\xi\| \leq \| \xi+1/2 \|$. First, notice that for any $z \in Z(\sqrt{n}S)$ and $i_1,...,i_K \in [d]$ we have
\begin{equation}
\label{eq:syebd}
\begin{split}
&\sum_{\substack{y \in \{-K,...,K\}^d,  
\\ \supp(y) \cap \supp(z)= \emptyset, 
\\
(\forall j \in [K]) |\{i \in [d]: |y_i|=j\}|=i_j
}}\hspace{-1cm} e(y \cdot \xi)= 2^{\sum_{j=1}^K i_j} \hspace{-1cm}\sum_{\substack{I_1,...,I_K \subseteq [d] \setminus \supp(z), \\
 (\forall j \in [K]) |I_j|=i_j, \\
    (\forall i,j \in [K], i \neq j) I_i \cap I_j=\emptyset}} \prod_{k=1}^K \prod_{i \in I_k} \cos(2k \pi \xi_i)   
\\
&=  2^{\sum_{j=1}^K i_j} \frac{(d-|\supp(z)|)!}{\prod_{j \in [K]}i_j! \cdot (d-|\supp(z)|-\sum_{j \in [K]}i_j)!} \beta_{\overline{i}}^{[d] \setminus \supp(z)} (\xi);
\end{split}
\end{equation}
recall Definition \ref{def:2.1} for the meaning of $\beta_{\overline{i}}^{[d] \setminus \supp(z)}.$ 
Let 
\[
F(l,\overline{i})=  2^{\sum_{j=1}^K i_j} \frac{(d-l)!}{\prod_{j \in [K]}i_j! (d-l-\sum_{j \in [K]}i_j)!}.
\]
Exploiting invariance of $Z(\sqrt{n}S)$ under sign changes of coordinates, that is using similar arguments as in the proof of Lemma \ref{rem:2.2}, we get
\begin{align*}
&\tilde{s}_{\sqrt{n}}(\xi)-\phi_{\sqrt{n},0}(\xi)\\
&= \frac{1}{|\sqrt{n}S|} \sum_{l=0}^a \sum_{\substack{z \in Z(\sqrt{n}S),
\\ |\supp(z)|=l, \\ |z|^2 \le a
}} \Big( \prod_{j=1}^d \cos(2\pi z_j \xi_j)-1 \Big)  \hspace{-1cm}\sum_{\substack{i_1,...,i_K \in [d], 
\\ n-a \leq \sum_{t=1}^K t^2i_t \leq n-|z|^2,
\\
i_1 \geq n/2
}} \hspace{-0.8cm} F(l, \overline{i}) \beta_{\overline{i}}^{[d] \setminus  \supp(z)}(\xi).
\end{align*}
From the inequality $|\prod_{j \in [d]} a_j - 1| \leq \sum_{j \in [d]} |a_j-1|$, valid if $\sup_j |a_j| \leq 1$, and Lemma \ref{lem:2.8} we obtain 
\begin{align*} 
&|\tilde{s}_{\sqrt{n}}(\xi)-\phi_{\sqrt{n},0}(\xi)| \lesssim 
\frac{1}{|\sqrt{n}S|} \sum_{l=0}^a \sum_{\substack{z \in Z(\sqrt{n}S),
\\ |\supp(z)|=l, \\ |z|^2 \le a
}}  \sum_{j=1}^d \sin^2( \pi z_j \xi_j)  \hspace{-0.8cm}\sum_{\substack{i_1,...,i_K \in [d], 
\\ n-a \leq \sum_{t=1}^K t^2i_t \leq n-|z|^2,
\\
i_1 \geq n/2
}} \hspace{-0.8cm} F(l, \overline{i}) 
\\
 &\cdot\exp\Bigg( -\frac{ci_1}{80K(d-|\supp(z)|)} \min\bigg( \hspace{-0.4cm}\sum_{k \in [d] \setminus \supp(z)} \hspace{-0.4cm} \sin^2( \pi  \xi_k),  \hspace{-0.4cm}\sum_{k \in [d] \setminus \supp(z)} \hspace{-0.4cm}\cos^2( \pi  \xi_k)  \bigg) \Bigg)
 \\
 &\lesssim_{K,\varepsilon}
 \frac{1}{|\sqrt{n}S|} \sum_{l=0}^a \sum_{\substack{z \in Z(\sqrt{n}S),
\\ |\supp(z)|=l, \\ |z|^2 \le a
}}  \sum_{j=1}^d \sin^2( \pi z_j \xi_j)  \hspace{-0.8cm}\sum_{\substack{i_1,...,i_K \in [d], 
\\ n-a \leq \sum_{t=1}^K t^2i_t \leq n-|z|^2,
\\
i_1 \geq n/2
}} \hspace{-0.8cm} F(l, \overline{i}) \exp\Big( -\frac{cn}{160Kd} \|\xi\|^2 \Big),
\end{align*}
where in the last estimate above we used our assumption $\|\xi\|\le \|\xi+1/2\|.$ Using the inequality $\sin^2(\pi xt) \leq |x|^2 \sin^2(\pi t),$ which holds for all $x \in \Z, t \in \mathbb{R},$ we get
\begin{align*}
&|\tilde{s}_{\sqrt{n}}(\xi)-\phi_{\sqrt{n},0}(\xi)| 
\\
& \lesssim \exp\Big( -\frac{cn}{160Kd} \| \xi \|^2 \Big)  \sum_{j=1}^d \sin^2( \pi \xi_j) \\
&\hspace{5cm}\frac{1}{|\sqrt{n}S|} \sum_{l=0}^a \hspace{-0.3cm}\sum_{\substack{i_1,...,i_K \in [d], 
\\ n-a \leq \sum_{t=1}^K t^2i_t \leq n,
\\
i_1 \geq n/2
}} \hspace{-0.6cm} F(l, \overline{i}) \hspace{-0.5cm}\sum_{\substack{z \in Z(\sqrt{n}S),
\\ |\supp(z)|=l, \\ |z|^2 \le n-\sum_{t=1}^K t^2i_t
}}  \hspace{-0.5cm} z_j^2.
\end{align*}
By symmetry invariance of $Z(\sqrt{n}S)$ the innermost sum above is
\[
\sum_{\substack{z \in Z(\sqrt{n}S),
\\ |\supp(z)|=l, \\ |z |^2\le n-\sum_{t=1}^K t^2i_t
}}  \hspace{-0.5cm} z_j^2\hspace{+0.1cm}=  \hspace{+0.1cm}\frac{1}{d} \hspace{-0.5cm}\sum_{\substack{z \in Z(\sqrt{n}S),
\\ |\supp(z)|=l, \\ |z |^2\le n-\sum_{t=1}^K t^2i_t
}} \hspace{-0.5cm}|z |^2 \hspace{+0.1cm}\le \hspace{+0.1cm}\frac{a}{d} \hspace{-0.5cm}\sum_{\substack{z \in Z(\sqrt{n}S),
\\ |\supp(z)|=l, \\ |z |^2\le n-\sum_{t=1}^K t^2i_t
}} \hspace{-0.5cm} 1
\]
and thus we obtain 
\begin{align*}
&|\tilde{s}_{\sqrt{n}}(\xi)-\phi_{\sqrt{n},0}(\xi)| \\
&\lesssim_{K,\varepsilon} 
\exp\Big( -\frac{cn}{160Kd} \| \xi \|^2 \Big)  \frac{ \| \xi \|^2}{d} \frac{1}{|\sqrt{n}S|} \sum_{l=0}^a \hspace{-0.3cm} \sum_{\substack{i_1,...,i_K \in [d], 
\\ n-a \leq \sum_{t=1}^K t^2i_t \leq n,
\\
i_1 \geq n/2
}} \hspace{-0.1cm} \sum_{\substack{z \in Z(\sqrt{n}S),
\\ |\supp(z)|=l, \\ |z|^2\le n-\sum_{t=1}^K t^2i_t
}} \hspace{-0.5cm} F(l, \overline{i}).
\end{align*}
However, putting $\xi=0$ in \eqref{eq:syebd} we notice that 

\begin{align*}
&\frac{1}{|\sqrt{n}S|} \sum_{l=0}^a \hspace{-0.3cm} \sum_{\substack{i_1,...,i_K \in [d], 
\\ n-a \leq \sum_{t=1}^K t^2i_t \leq n,
\\
i_1 \geq n/2
}} \hspace{-0.1cm} \sum_{\substack{z \in Z(\sqrt{n}S),
\\ |\supp(z)|=l, \\ |z|^2\le n-\sum_{t=1}^K t^2i_t
}} \hspace{-0.5cm} F(l, \overline{i})
\\
&=\frac{1}{|\sqrt{n}S|} \sum_{l=0}^a \sum_{\substack{z \in Z(\sqrt{n}S),
\\ |\supp(z)|=l, \\ |z|^2 \le a
}}\hspace{+0.1cm} \sum_{\substack{i_1,...,i_K \in [d], 
\\ n-a \leq \sum_{j=1}^K j^2i_j \leq n-|z|^2,
\\
i_1 \geq n/2
}}  \hspace{-0.2cm}
\sum_{\substack{y \in \{-K,...,K\}^d, 
\\ \supp(y) \cap \supp(z)= \emptyset, 
\\
(\forall j \in [K]) |\{i \in [d]: |y_i|=j\}|=i_j
}} \hspace{-0.5cm} 1 \\
&= \tilde{s}_{\sqrt{n}}(0) \leq s_{\sqrt{n}}(0)=1.
\end{align*}
Thus we have
\[
|\tilde{s}_{\sqrt{n}}(\xi)-\phi_{\sqrt{n},0}(\xi)| \lesssim_{K,\varepsilon} \exp\Big( -\frac{cn}{160Kd} \| \xi \|^2 \Big)  \frac{ \| \xi \|^2}{d} \lesssim_{K,\varepsilon} \frac{1}{n}
\]
and this finishes the proof of \eqref{eq:4.1}.

\par The proof of  \eqref{eq:4.2} is similar and thus we only sketch it. If $\| \xi+1/2 \| \leq \| \xi \|$ then one has
\[
\tilde{s}_{\sqrt{n}}(\xi)=\frac{1}{|\sqrt{n}S|} \sum_{l=0}^a \sum_{\substack{z \in Z(\sqrt{n}S),
\\ |\supp(z)|=l, \\ |z|^2 \le a
}}(-1)^{\sum_{j=1}^d z_j} e\big(z \cdot (\xi+1/2)\big) 
\] \[\sum_{\substack{i_1,...,i_K \in [d], 
\\ n-a \leq \sum_{j=1}^K j^2i_j \leq n-|z|^2,
\\
i_1 \geq n/2
}} 
\hspace{-0.8cm}(-1)^{\sum_{j \in [K]} ji_j} \hspace{-0.8cm}
\sum_{\substack{y \in \{-K,...,K\}^d, \\
n-a \leq |y|^2 \leq n- u, 
\\ \supp(y) \cap \supp(z)= \emptyset, 
\\
(\forall j \in [K]) |\{i \in [d]: |y_i|=j\}|=i_j
}}\hspace{-0.8cm} e\big(y \cdot (\xi+1/2) \big),
\]
\[\phi_{\sqrt{n},1}(\xi)=\frac{1}{|\sqrt{n}S|} \sum_{l=0}^a \sum_{\substack{z \in Z(\sqrt{n}S),
\\ |\supp(z)|=l, \\ |z|^2 \le a
}} (-1)^{\sum_{i=1}^d z_i} \]
\[\sum_{\substack{i_1,...,i_K \in [d], 
\\ n-a \leq \sum_{j=1}^K j^2i_j \leq n-|z|^2,
\\
i_1 \geq n/2
}} \hspace{-0.8cm} (-1)^{\sum_{j \in [K]} ji_j} \hspace{-0.8cm}
\sum_{\substack{y \in \{-K,...,K\}^d, 
\\ \supp(y) \cap \supp(z)= \emptyset, 
\\
(\forall j \in [K]) |\{i \in [d]: |y_i|=j\}|=i_j
}} \hspace{-0.8cm} e\big(y \cdot (\xi+1/2) \big).
\]
Hence using similar arguments as before and denoting $\eta=\xi+1/2$ we get
\begin{align*}
&|\tilde{s}_{\sqrt{n}}(\xi)-\phi_{\sqrt{n},1}(\xi)| \\
&\lesssim \frac{1}{|\sqrt{n}S|} \sum_{l=0}^a \sum_{\substack{z \in Z(\sqrt{n}S),
\\ |\supp(z)|=l, \\ |z|^2 \le a
}} \Big| \prod_{j=1}^d \cos\big(2\pi z_j \eta_j\big)-1 \Big|
\hspace{-0.8cm}\sum_{\substack{i_1,...,i_K \in [d], 
\\ n-a \leq \sum_{j=1}^K j^2i_j \leq n-|z|^2,
\\
i_1 \geq n/2
}} \hspace{-0.8cm}F(l, \overline{i}) |\beta_{\overline{i}}^{[d] \setminus  \supp(z)}(\eta)|.
\end{align*}
Since $\|\eta\|\le \|\eta+1/2\|$ we see that the quantity above was estimated by $O_K(1/n)$ during the proof of \eqref{eq:4.1}. This justifies \eqref{eq:4.2}.

Having proven \eqref{eq:4.1} and \eqref{eq:4.2} we take a closer look at $\phi_{\sqrt{n},0}(\xi)$ and $\phi_{\sqrt{n},1}(\xi).$ Notice that 
\begin{align*}
&\phi_{\sqrt{n},0}(\xi)= \frac{1}{|\sqrt{n}S|} \sum_{l=0}^a \sum_{\substack{I \subseteq [d], \\ |I|=l}}\sum_{\substack{z \in Z(\sqrt{n}S),
\\ \supp(z)=I, \\ |z|^2 \le a
}}  \hspace{+0.1cm}\sum_{\substack{i_1,...,i_K \in [d], 
\\ n-a \leq \sum_{j=1}^K j^2i_j \leq n-|z|^2,
\\
i_1 \geq n/2
}} 
\hspace{-0.5cm} \sum_{\substack{y \in \{-K,...,K\}^d,  
\\ \supp(y) \cap I= \emptyset, 
\\
(\forall j \in [K]) |\{i \in [d]: |y_i|=j\}|=i_j
}} \hspace{-0.8cm} e(y \cdot \xi) 
\\
&=\frac{1}{|\sqrt{n}S|} \sum_{l=0}^a \sum_{\substack{I \subseteq [d], \\ |I|=l}} \sum_{\substack{i_1,...,i_K \in [d], 
\\ n-a \leq \sum_{j=1}^K j^2i_j \leq n,
\\
i_1 \geq n/2
}} \hspace{-0.2cm} \sum_{\substack{z \in Z(\sqrt{n}S),
\\ \supp(z)=I, \\ |z|^2 \le n-\sum_{j=1}^K j^2i_j
}} 1 \cdot
\hspace{-0.4cm}
\sum_{\substack{y \in \{-K,...,K\}^d, 
\\ \supp(y) \cap I= \emptyset, 
\\
(\forall j \in [K]) |\{i \in [d]: |y_i|=j\}|=i_j
}} \hspace{-0.8cm}e(y \cdot \xi).
\end{align*}
 Due to symmetry invariance of $Z(\sqrt{n}S)$, for any $i_1,\ldots, i_K \in [d]$ and $\ I \subseteq [d]$ with $|I|=l$ and $\sum_{j=1}^K j^2i_j\le n$, we have
\[
|\{z \in Z(\sqrt{n}S):\supp(z)=I, | z|^2 \le n-\sum_{j=1}^K j^2i_j \}|= \frac{1}{\binom{d}{l}} \sum_{\substack{z \in Z(\sqrt{n}S),
\\ |\supp(z)|=l, \\ |z|^2 \le n-\sum_{j=1}^K j^2i_j
}}  1,
\]
hence
\[
\phi_{\sqrt{n},0}(\xi)
=\frac{1}{|\sqrt{n}S|} \sum_{l=0}^a \sum_{\substack{z \in Z(\sqrt{n}S),
\\ |\supp(z)|=l, \\ |z|^2 \le a
}}  
\hspace{+0.1cm}\sum_{\substack{i_1,...,i_K \in [d], 
\\ n-a \leq \sum_{j=1}^K j^2i_j \leq n-|z|^2,
\\
i_1 \geq n/2
}} 
\se{\substack{I \subseteq [d], \\ |I|=l}}
\hspace{-0.5cm}\sum_{\substack{y \in \{-K,...,K\}^d,  
\\ \supp(y) \cap I= \emptyset, 
\\
(\forall j \in [K]) |\{i \in [d]: |y_i|=j\}|=i_j
}} \hspace{-0.8cm}e(y \cdot \xi).
\]
Notice that 
\begin{align*}
&\se{\substack{I \subseteq [d], \\ |I|=l}}
\hspace{-0.5cm}\sum_{\substack{y \in \{-K,...,K\}^d,  
\\ \supp(y) \cap I= \emptyset, 
\\
(\forall j \in [K]) |\{i \in [d]: |y_i|=j\}|=i_j
}}\hspace{-0.8cm} e(y \cdot \xi)\\
&= \frac{1}{\binom{d}{l}}\hspace{-0.6cm}\sum_{\substack{y \in \{-K,...,K\}^d,  
\\
(\forall j \in [K]) |\{i \in [d]: |y_i|=j\}|=i_j
}} \hspace{-0.8cm}e(y \cdot \xi) | \{ I \subseteq [d]: |I|=l, I \cap \supp(y)= \emptyset \}|
\\
&=\frac{\binom{d-\sum_{j \in[K]}i_j}{l}}{\binom{d}{l}} \hspace{-0.6cm}\sum_{\substack{y \in \{-K,...,K\}^d,  
\\
(\forall j \in [K]) |\{i \in [d]: |y_i|=j\}|=i_j
}} \hspace{-0.8cm}e(y \cdot \xi)= \frac{\binom{d-\sum_{j \in[K]}i_j}{l}}{\binom{d}{l}} |D_{\overline{i}}| \beta_{\overline{i}}(\xi), 
\end{align*}
where $\overline{i}=(i_1,\ldots,i_K),$
and this implies 
\[
\phi_{\sqrt{n},0}(\xi)
=\frac{1}{|\sqrt{n}S|} \sum_{l=0}^a \sum_{\substack{z \in Z(\sqrt{n}S),
\\ |\supp(z)|=l, \\ |z|^2 \le a
}}  
\hspace{+0.1cm}\sum_{\substack{i_1,...,i_K \in [d], 
\\ n-a \leq \sum_{j=1}^K j^2i_j \leq n-|z|^2,
\\
i_1 \geq n/2
}} 
\frac{\binom{d-\sum_{j \in[K]}i_j}{l}}{\binom{d}{l}} |D_{\overline{i}}| \beta_{\overline{i}}(\xi).
\]
Using exactly the same argument one can obtain
\[\phi_{\sqrt{n},1}(\xi)
=\frac{1}{|\sqrt{n}S|} \sum_{l=0}^a \sum_{\substack{z \in Z(\sqrt{n}S),
\\ |\supp(z)|=l, \\ |z|^2 \le a
}}  (-1)^{\sum_{j=1}^dz_j}
\hspace{-1cm}\sum_{\substack{i_1,...,i_K \in [d], 
\\ n-a \leq \sum_{j=1}^K j^2i_j \leq n-|z|^2,
\\
i_1 \geq n/2
}} 
\hspace{-0.8cm}\frac{\binom{d-\sum_{j \in[K]}i_j}{l}}{\binom{d}{l}} |D_{\overline{i}}| \beta_{\overline{i}}(\xi).
\]
Since $\phi_{\sqrt{n},0}(0)=\tilde{s}_{\sqrt{n}}(0)=s_{\sqrt{n}}(0)-r_{\sqrt{n}}(0) \leq 1$, we see that
\begin{equation*} 
\frac{1}{|\sqrt{n}S|} \sum_{l=0}^a \sum_{\substack{z \in Z(\sqrt{n}S),
\\ |\supp(z)|=l, \\ |z|^2 \le a
}}  
\hspace{+0.1cm}\sum_{\substack{i_1,...,i_K \in [d], 
\\ n-a \leq \sum_{j=1}^K j^2i_j \leq n-|z|^2,
\\
i_1 \geq n/2
}} 
\frac{\binom{d-\sum_{j \in[K]}i_j}{l}}{\binom{d}{l}} |D_{\overline{i}}| \leq 1.
\end{equation*}
This implies that for $\epsilon \in \{0,1\}$ and $f\in \ell^2$ we have the pointwise bound, valid for $x\in \Z^d,$
\begin{equation}
\label{eq:4.4}
\begin{split}
\Big|\mathcal{F}^{-1} \Big( \phi_{\sqrt{n},\epsilon}\widehat{f} \Big)(x)\Big| &\leq \frac{1}{|\sqrt{n}S|} \sum_{l=0}^a \sum_{\substack{z \in Z(\sqrt{n}S),
\\ |\supp(z)|=l, \\ |z|^2 \le a
}}  
\hspace{+0.1cm}\sum_{\substack{i_1,...,i_K \in [d], 
\\ n-a \leq \sum_{j=1}^K j^2i_j \leq n-|z|^2,
\\
i_1 \geq n/2
}} 
\frac{\binom{d-\sum_{j \in[K]}i_j}{l}}{\binom{d}{l}} |D_{\overline{i}}| \Big|\mathcal{F}^{-1} \Big( \beta_{\overline{i}}\widehat{f} \Big)(x)\Big|
\\
&\leq \sup_{n_1,...,n_K \leq \frac{d}{2K}} |\mathcal{D}_{\overline{n}}f|(x).
\end{split}
\end{equation}
Above we used the fact that $n \leq \frac{d}{2K}$, which follows from assumptions $d \geq (25K)^{(K+1)^2}$ and $n \leq d^{1-\frac{1+\varepsilon}{(K+1)^2}}$.

Having established \eqref{eq:rsnbound}, \eqref{eq:4.1}, \eqref{eq:4.2} and \eqref{eq:4.4} we can now finish the proof of Proposition \ref{lem:4.3}. Take any $f \in \ell^2(\Z^d)$ and decompose $f=f_0+f_1$, where
\begin{align*}
\supp(\widehat{f_0}) &\subseteq \{\xi \in \T^d: \|\xi \| \leq \| \xi + 1/2 \| \},\\
\supp(\widehat{f_1}) &\subseteq \{\xi \in \T^d: \|\xi \| > \| \xi + 1/2 \| \}.
\end{align*}
Then we have
\begin{equation}
\label{eq:sf0f1}
\Big\| \sup_{ n \leq d^{1-\frac{1+\varepsilon}{(K+1)^2}}} |\mathcal{S}_{\sqrt{n}}f| \Big\|_{2} \lesssim \Big\| \sup_{ n \leq d^{1-\frac{1+\varepsilon}{(K+1)^2}}} |\mathcal{S}_{\sqrt{n}}f_0| \Big\|_{2} +\Big\| \sup_{ n \leq d^{1-\frac{1+\varepsilon}{(K+1)^2}}} |\mathcal{S}_{\sqrt{n}}f_1| \Big\|_{2}. 
\end{equation}

The  term on the right hand side of \eqref{eq:sf0f1}  containing $f_0$ is estimated with the aid of \eqref{eq:4.4} 
\begin{align*}
&\Big\| \sup_{ n \leq d^{1-\frac{1+\varepsilon}{(K+1)^2}}} |\mathcal{S}_{\sqrt{n}}f_0| \Big\|_{2} \\
&\leq \Big\| \hspace{-0.3cm} \sup_{ n \leq d^{1-\frac{1+\varepsilon}{(K+1)^2}}} \Big|\mathcal{F}^{-1} \Big( \phi_{\sqrt{n},0}\widehat{f_0} \Big)\Big| \Big\|_{2} +\Big\| \hspace{-0.3cm} \sup_{ n \leq d^{1-\frac{1+\varepsilon}{(K+1)^2}}} \Big|\mathcal{F}^{-1} \Big( (s_{\sqrt{n}}-\phi_{\sqrt{n},0})\widehat{f_0} \Big)\Big| \Big\|_{2}\\
&\le \Big\|\sup_{n_1,...,n_K \leq \frac{d}{2K}} |\mathcal{D}_{\overline{n}}f_0|\Big\|_2+\Big\| \sup_{ n \leq d^{1-\frac{1+\varepsilon}{(K+1)^2}}} \Big|\mathcal{F}^{-1} \Big( (s_{\sqrt{n}}-\phi_{\sqrt{n},0})\widehat{f_0} \Big)\Big| \Big\|_{2}.
\end{align*} 
For the second term of the last inequality we use Parseval's formula, \eqref{eq:rsnbound}, and \eqref{eq:4.1} to get
\begin{align*}
&\Big\| \sup_{ n \leq d^{1-\frac{1+\varepsilon}{(K+1)^2}}} \Big|\mathcal{F}^{-1} \Big( (s_{\sqrt{n}}-\phi_{\sqrt{n},0})\widehat{f_0} \Big)\Big| \Big\|_{2}^2 \leq \Big\| \Big(\sum_{ n \leq d^{1-\frac{1+\varepsilon}{(K+1)^2}}} \Big|\mathcal{F}^{-1} \Big( (s_{\sqrt{n}}-\phi_{\sqrt{n},0})\widehat{f_0} \Big)\Big|^2 \Big)^{1/2} \Big\|_{2}^2\\
&= \sum_{ n \leq d^{1-\frac{1+\varepsilon}{(K+1)^2}}} \Big\|\mathcal{F}^{-1} \Big( (s_{\sqrt{n}}-\phi_{\sqrt{n},0})\widehat{f_0} \Big)\Big\|_{2}^2  =  \int_{\T^d} | \widehat{f_0}(\xi)|^2 \sum_{n \leq d^{1-\frac{1+\varepsilon}{(K+1)^2}}} |s_{\sqrt{n}}(\xi)-\phi_{\sqrt{n},0}(\xi)|^2 \,d\xi 
\\
&\lesssim_{K, \varepsilon}  \int_{\T^d} | \widehat{f_0}(\xi)|^2\,d\xi \sum_{n \leq d}\big( \frac{1}{d^2}+ 2^{-n}+ \frac{1}{n^2}\big)  \lesssim_{K, \varepsilon} \| f_0 \|_2^2.
\end{align*}
In summary, we justified that
\[
\Big\| \sup_{ n \leq d^{1-\frac{1+\varepsilon}{(K+1)^2}}} |\mathcal{S}_{\sqrt{n}}f_0| \Big\|_{2} \lesssim_{K,\varepsilon}  \Big\| \sup_{n_1,...,n_K \leq \frac{d}{2K}} |\mathcal{D}_{\overline{n}}f_0| \Big\|_{2}+ \|f_0 \|_{2}.
\]

Using similar arguments, but with $\phi_{\sqrt{n},0}$ replaced by $\phi_{\sqrt{n},1}$ and \eqref{eq:4.2} used in place of \eqref{eq:4.1}, we can also estimate the term on the right hand side of \eqref{eq:sf0f1} containing $f_1,$ namely
\[
\Big\| \sup_{ n \leq d^{1-\frac{1+\varepsilon}{(K+1)^2}}} |\mathcal{S}_{\sqrt{n}}f_1| \Big\|_{2} \lesssim_{K,\varepsilon}  \Big\| \sup_{n_1,...,n_K \leq \frac{d}{2K}} |\mathcal{D}_{\overline{n}}f_1| \Big\|_{2}+ \|f_1 \|_{2}.
\]
Thus if $d$ is sufficiently large we obtain
\[
\Big\| \sup_{ n \leq d^{1-\frac{1+\varepsilon}{(K+1)^2}}} |\mathcal{S}_{\sqrt{n}}f| \Big\|_{2} \lesssim_{K,\varepsilon} \sum_{i=0}^1  \Big\| \sup_{n_1,...,n_K \leq \frac{d}{2K}} |\mathcal{D}_{\overline{n}}f_i| \Big\|_{2}+ \|f \|_{2}
\]
and this completes the proof of Proposition \ref{lem:4.3}.
\end{proof}
We are finally ready to prove Theorem \ref{thm:sphere}.
\begin{proof}[Proof of Theorem \ref{thm:sphere}] For $p= \infty$ the statement is trivial, due to interpolation it suffices to consider only $p=2$. \par
Take any $\varepsilon>0$, then there exists $K \in \N$ such that
\[
1-  \varepsilon \leq 1-\frac{1+\varepsilon}{(K+1)^2},
\]
for this $K$ take $d_0$ from Proposition \ref{lem:4.3} and consider $d \geq d_0.$ By Proposition \ref{lem:4.3} and Theorem \ref{thm:1.5} for $f \in \ell^2(\Z^d)$ we obtain
\begin{align*}
   &\Big\| \sup_{n \leq d^{1- \varepsilon}}  |\mathcal{S}_{\sqrt{n}}f| \Big\|_{2} \leq \Big\| \sup_{n \leq d^{1-\frac{1+\varepsilon}{(K+1)^2}}} |\mathcal{S}_{\sqrt{n}}f| \Big\|_{2} \\
   &\lesssim_{K,\varepsilon} \sum_{i=0}^1 \Big\| \sup_{n_1,n_2,...,n_K \leq \frac{d}{2K}} |\mathcal{D}_{n_1,...,n_K}f_i| \Big\|_{2} + \|f\|_{2} \\
&\lesssim_{K,\varepsilon} \sum_{i=0}^1  \|f_i\|_{2} + \|f\|_{2} \lesssim_{K,\varepsilon}  \|f\|_{2}. 
\end{align*}
On the other hand if $4\le d \leq d_0$, then we have
\[
\Big\| \sup_{n \leq d^{1-\varepsilon}}  |\mathcal{S}_{\sqrt{n}}f| \Big\|_{2} \leq \sum_{n \leq d^{1-\varepsilon}} \Big\|  |\mathcal{S}_{\sqrt{n}}f| \Big\|_{2} \le d_0 \|f \|_{2},
\]
since $\|\mathcal{S}_{\sqrt{n}}f \|_{2} \leq \| f \|_{2}$. Therefore we proved that in any dimension $d\ge 4$ it holds
\[
\Big\| \sup_{n \leq d^{1-\varepsilon}}  |\mathcal{S}_{\sqrt{n}}f| \Big\|_{2}  \lesssim_{K,\varepsilon} \|f \|_{2}
\]
and this concludes the proof of Theorem \ref{thm:sphere}.
\end{proof}

\section*{Acknowledgments}
The authors are grateful to the referee and Dariusz Kosz for helpful comments that led to the improvement of the presentation.

\bibliographystyle{plainnat} 
\bibliography{references}

\end{document}